\documentclass{article}

\usepackage{authblk}

\usepackage[utf8]{inputenc}
\usepackage[T1]{fontenc}
\usepackage[english]{babel}
\usepackage{textcomp}
\usepackage{amsmath,amssymb,amsthm}
\usepackage{lmodern}
\usepackage[a4paper]{geometry}
\usepackage{xcolor, pict2e}
\usepackage{microtype}
\usepackage{caption}
\usepackage{listings}
\usepackage{multicol}
\usepackage{moreverb}
\usepackage{hyperref}
\hypersetup{pdfstartview=XYZ}
\usepackage{wrapfig}
\usepackage[sans]{dsfont}

\usepackage{tikz, pgfplots}
\pgfplotsset{compat=1.18}
\usetikzlibrary{positioning}
\usepackage{amsmath,amssymb}
\usetikzlibrary{decorations.pathreplacing}

\usepackage{ytableau}

\usepackage{hyperref}
\hypersetup{
    colorlinks=true,
    linkcolor=blue,
    citecolor=magenta,
    urlcolor=blue,
    pdfborder={0 0 0}
}

\definecolor{fond}{rgb}{0.05,0.05,0.25}

\DeclareMathOperator{\Poiss}{Poiss}

\DeclareMathOperator{\dtv}{d_{TV}}

\DeclareMathOperator{\ch}{\mathrm{ch}}
\DeclareMathOperator{\Id}{\mathrm{Id}}

\DeclareMathOperator{\Cay}{\mathrm{Cay}}
\DeclareMathOperator{\Var}{\mathrm{Var}}

\DeclareMathOperator{\Unif}{\mathrm{Unif}}
\DeclareMathOperator{\cyc}{\mathrm{cyc}}
\DeclareMathOperator{\imin}{i_\mathrm{min}}
\DeclareMathOperator{\ibis}{i_\mathrm{bis}}
\DeclareMathOperator{\sgn}{sgn}

\usepackage{todonotes}

\usepackage{amsmath,amssymb,amsthm}
\usepackage{mathtools,mathrsfs}
\usepackage{dsfont}
\usepackage{stmaryrd}

\usepackage{enumitem}
\setlist[itemize,1]{nosep}
\setlist[enumerate,1]{nosep,label=(\alph*)}

\usepackage{float}
\usepackage[caption = false]{subfig}
\usepackage{graphicx}

\newtheorem*{theorem*}{Theorem}
\newtheorem{theorem}{Theorem}[section]
\newtheorem{proposition}[theorem]{Proposition}
\newtheorem{lemma}[theorem]{Lemma}
\newtheorem{corollary}[theorem]{Corollary}
\newtheorem{example}[theorem]{Example}
\newtheorem{definition}[theorem]{Definition}
\newtheorem{conjecture}[theorem]{Conjecture}

\theoremstyle{remark}
\newtheorem{remark}{Remark}[section]

\newcommand{\cC}{{\ensuremath{\mathcal C}} }

\newcommand{\cL}{{\ensuremath{\mathcal L}} }

\newcommand{\bbZ}{{\ensuremath{\mathbb Z}} }

\newcommand{\bbR}{{\ensuremath{\mathbb R}} }

\newcommand{\bbP}{{\ensuremath{\mathbb P}} }

\newcommand{\kS}{{\ensuremath{\mathfrak{S}}} }
\newcommand{\kA}{{\ensuremath{\mathfrak{A}}} }
\newcommand{\kE}{{\ensuremath{\mathfrak{E}}} }

\newcommand{\ag}{\left\{ } 

\newcommand{\ad}{\right\} }

\newcommand{\cg}{\left[}
\newcommand{\cd}{\right]}
\newcommand{\pg}{\left(} 
\newcommand{\pd}{\right)}
\newcommand{\bg}{\left|}
\newcommand{\bd}{\right|}
\newcommand{\lf}{\left\lfloor}
\newcommand{\rf}{\right\rfloor}
\newcommand{\lc}{\left\lceil}
\newcommand{\rc}{\right\rceil}

\newcommand{\du}{{\ensuremath{\;:\;}}} 

\renewcommand{\P}{\mathbb{P}}
\newcommand{\E}{\mathbb{E}}

\DeclareMathOperator{\ST}{ST}

\makeatletter
\newcommand*\bigcdot{\mathpalette\bigcdot@{.5}}
\newcommand*\bigcdot@[2]{\mathbin{\vcenter{\hbox{\scalebox{#2}{$\m@th#1\bullet$}}}}}
\makeatother

\numberwithin{equation}{section}

\title{Characters of symmetric groups: sharp bounds on virtual degrees and the Witten zeta function}

\author[1]{Lucas Teyssier}
\author[2]{Paul Thévenin}
\affil[1]{
Université de Lorraine, \texttt{lucas.teyssier@univ-lorraine.fr}
}
\affil[2]{
Université d'Angers, \texttt{paul.thevenin@univ-angers.fr
}}
\date{}
\begin{document}

\renewcommand{\theparagraph}{\thesubsection.\arabic{paragraph}}
\maketitle

\begin{abstract}
We prove sharp bounds on the virtual degrees introduced by Larsen and Shalev. This leads to improved bounds on
characters of symmetric groups. We then sharpen bounds of Liebeck and Shalev concerning the Witten zeta function. Our main application is a characterization of the fixed-point free conjugacy classes whose associated random walk mixes in 2 steps.
\end{abstract}

\renewcommand\abstractname{Résumé}
\begin{abstract}
Nous prouvons des bornes précises sur les degrés virtuels introduits par Larsen et Shalev. Cela induit de meilleures bornes sur les caractères du groupe symétrique. Dans un second temps, nous améliorons certaines bornes de Liebeck et Shalev sur la fonction zeta de Witten. Notre application principale est une caractérisation des classes de conjugaison sans point fixe dont la marche aléatoire associée est mélangée au temps 2.
\end{abstract}

\renewcommand\abstractname{Resumo}
\begin{abstract}
Ni pruvas precizajn barojn je la virtualaj gradoj enkondukitaj de Larsen kaj Shalev. Tio induktas plibonigitajn barojn je karakteroj de la simetria grupo. Ni sekve pliakrigas barojn de Liebeck kaj Shalev pri la zeto-funkcio de Witten. Nia ĉefa aplikaĵo estas karakterizo de la senfikspunktaj konjugklasoj kies asociata hazarda promenado miksiĝas post 2 paŝoj.
\end{abstract}

\tableofcontents

\section{Introduction}

\subsection{Context}
Representation theory has been used to solve various problems in different areas of mathematics. A striking example is the pioneering work of Diaconis and Shahshahani on a random process called the random transposition shuffle \cite{DiaconisShahshahani1981}. They used character estimates to prove a sharp phase transition -- now known as the cutoff phenomenon -- concerning the minimum number of random transpositions whose product yields an almost uniform random permutation. This was the starting point of the field of mixing times, whose goal is to understand how long it takes for a random process to approach stationarity.

Other techniques to study mixing times were developed in the following years, notably by Aldous and Diaconis \cite{Aldous1983mixing, AldousDiaconis1986}, and mixing properties of emblematic card shufflings were precisely understood, also for small decks of cards \cite{BayerDiaconis1992}. We refer to \cite{LivreLevinPeres2019MarkovChainsAndMixingTimesSecondEdition} for an introduction to mixing times and to \cite{LivreDiaconisFulman2023} for the mathematics of card shuffling.

\medskip

A natural way to generalize Diaconis and Shahshahani's random transposition shuffle on the symmetric group $\kS_n$ is to replace the conjugacy class of transpositions -- from which we pick uniform elements -- by another conjugacy class.
In specific cases, a similar cutoff phenomenon was proved; for $k$-cycles (where $k>n/2$) in \cite{LulovPak2002}; for $k$-cycles (where $k$ is fixed) in \cite{BerestyckiSchrammZeitouni2011}; for $k$-cycles (where $k=o(n)$) in \cite{Hough2016}; and for all conjugacy classes whose support has size $o(n)$ in \cite{BerestyckiSengul2019} (the support of a permutation is the set of its non-fixed points). Most of these results rely on representation theory. However, interestingly, the proof of \cite{BerestyckiSchrammZeitouni2011} follows the cycle structure of the permutation obtained after several multiplications via probabilistic arguments, and that of \cite{BerestyckiSengul2019}, which was the first cutoff result for a large class of conjugacy classes, relies on a curvature argument. 
In the companion paper \cite{Olesker-TaylorTeyssierThevenin2025sharpboundsandcutoff}, we extend the cutoff result of \cite{BerestyckiSengul2019} to all conjugacy classes with a number of fixed points of order $n$, also with respect to the $L^2$ distance.

There has also been significant effort devoted to understanding conjugacy classes with full support (i.e.\ without fixed points). For the cycle types $[r^{n/r}]$ Lulov \cite{Lulovthesis1996} proved using the character bounds of Fomin and Lulov \cite{FomLul95} that the mixing time is $3$ when $r=2$ and $2$ when $r\geq 3$. Lulov and Pak later conjectured in \cite[Conjecture 4.1]{LulovPak2002} that the mixing time is at most 3 for all fixed-point free conjugacy classes (i.e., without fixed points), and Larsen and Shalev \cite{LarsenShalev2008} proved that conjecture.

\medskip

Motivated by the study of mixing times of Markov chains (among various other problems), uniform bounds on characters have been established and sharpened over the past few decades. 
Roichman \cite{Roichman1996} proved bounds that led to the correct order of the mixing time of conjugacy class walks on $\mathfrak{S}_n$ that have a number of fixed points of order $n$. This result was later strengthened by Müller and Schlage-Puchta \cite{MullerSchlagePuchta2007precutoff}, who obtained results that are uniform over all conjugacy classes.

On the other hand, sharp bounds on characters of permutations with few fixed points were proved in the landmark paper \cite{LarsenShalev2008} of Larsen and Shalev. They proved several conjectures, among which the Lulov--Pak conjecture as discussed earlier. The bounds that we obtain here build on this paper; see Section \ref{s: description of the results of Larsen and Shalev} for a more detailed description of their character bounds.
In the companion paper \cite{Olesker-TaylorTeyssierThevenin2025sharpboundsandcutoff} we prove an optimal uniform character bound as well as the sharpest possible \textit{exponential} character bounds for permutations with enough fixed points.

An important tool to estimate characters of $\kS_n$ is the Murnaghan--Nakayama rule (see e.g.\ Theorem 3.10 in \cite{LivreMeliot2017RepresentationTheoryofSymmetricGroups}), which is a key element in the proofs of \cite{Roichman1996, MullerSchlagePuchta2007precutoff,LarsenShalev2008, Olesker-TaylorTeyssierThevenin2025sharpboundsandcutoff}. On the other hand, other techniques have proved to be more powerful for certain diagrams. Féray and Śniady \cite{FeraySniady2011} found uniform bounds on characters that are especially good for diagrams that do not have a long first row or column, using a reformulation of Stanley's conjectured character formula. More precisely, \cite{Feray2010Stanleyformula} proves Conjecture 3 from \cite{Stanley2006conj}, while the reformulation of Stanley's formula appears as Theorem 2 in \cite{FeraySniady2011}.

In a different direction, Lifshitz and Marmor \cite{LifshitzMarmor2024charactershypercontractivity} established character bounds using hypercontractivity on symmetric groups (see also \cite{FilmusKindlerLifshitzMinzer2024,KeevashLifshitz2023sharp}). 
They obtained uniform bounds on characters for permutations with at most a given number of cycles (instead of considering the full cycle structure, as done in the previously mentioned works). This allows them to consider permutations having up to $n/(\ln n)^{O(1)}$ cycles, thus covering new cases.

\medskip

Bounds on characters play a crucial role in the broader context of finite groups, and were notably used to prove a conjecture due to Ore \cite{LiebeckOBrienShalevTiep2010Oreconjecture}. General bounds on characters of finite classical groups were recently obtained by Guralnick, Larsen and Tiep \cite{GuralnickLarsenTiep2020characterlevels1,GuralnickLarsenTiep2024characterlevels2}, and uniform bounds on characters for all finite quasisimple groups of Lie type were established by Larsen and Tiep \cite{LarsenTiep2024FiniteClassicalGroups}.

\medskip

A canonical tool used to convert character bounds for a group $G$ into practical estimates for applications is the Witten zeta function. It is defined as $\zeta(s) = \sum_\alpha (d_\alpha)^{-s}$ for $s>0$, where the sum is taken over all irreducible representations $\alpha$ of $G$, and $d_\alpha$ is the dimension of $\alpha$. Such a zeta function initially appeared in the computation of volumes of moduli spaces in Witten's work \cite{Witten1991gaugetheory} (see Equation (4.72) there). Liebeck and Shalev proved bounds for the symmetric group \cite{LiebeckShalev2004} and for Fuchsian groups \cite{LiebeckShalev2005}, which were useful for a variety of applications, including mixing times and covering numbers \cite{LarsenShalev2008, LarsenShalevTiep2024characteristiccoveringnumbers,  KellerLifshitzScheinfeld2024}.  We refer to \cite{LarsenShalev2008}, \cite{LarsenTiep2024FiniteClassicalGroups}, and the survey article  \cite{Liebeck2017survey} for further applications of character bounds.

\medskip

Throughout the paper, $\mathfrak{S}_n$ denotes the symmetric group of degree $n$, and $\widehat{\mathfrak{S}_n}$ the set of its irreducible representations. For a permutation $\sigma$ and $i \in \{1, \ldots, n\}$, $f_i(\sigma)$ denotes the number of cycles of length $i$ in $\sigma$. By extension, for a conjugacy class $\cC$, we denote by $f_i(\cC)$ the number of cycles of length $i$ in any permutation $\sigma \in \cC$. In addition, we will use the notation $f=\max(f_1,1)$. For convenience, we will use $\lambda$ to simultaneously denote a representation, the associated integer partition, and the associated Young diagram (see Section \ref{ssec:young diagrams representations and partitions} for definitions); and we may write $\lambda \vdash n$ for $\lambda \in \widehat{\kS_n}$. If $E$ is a finite set, we denote by $\Unif_E$ the uniform probability measure on $E$.

\subsection{The Larsen--Shalev character bounds}\label{s: description of the results of Larsen and Shalev}
We describe here the results of Larsen and Shalev \cite{LarsenShalev2008}, whose improvement is the main purpose of the paper.
We assume familiarity with the representation theory of symmetric groups and will follow the notation from the textbook \cite{LivreMeliot2017RepresentationTheoryofSymmetricGroups}: we denote by $d_\lambda$ the dimension of a representation $\lambda$, $\ch^\lambda(\sigma)$ the character of $\lambda$ evaluated at a permutation $\sigma \in \mathfrak{S}_n$, and $\chi^\lambda(\sigma) = \frac{\ch^\lambda(\sigma)}{d_\lambda}$ the associated normalized character.

\medskip

Larsen and Shalev obtained character bounds by introducing and studying new objects. One of them is called the virtual degree $D(\lambda)$ of a representation $\lambda$. It is defined as

\begin{equation}\label{eq: def virtual degree}
    D(\lambda) := \frac{(n-1)!}{\prod_i a_i!b_i!},
\end{equation}
where $a_i = \lambda_i-i$ and $b_i = \lambda'_i-i$. Here $\lambda_i$ (resp.\  $\lambda'_i$) denotes the size of the $i$-th row (resp.\  $i$-th column) of the Young diagram associated to $\lambda$ (see Section \ref{ssec:young diagrams representations and partitions} for precise definitions). 
The virtual degree is a convenient substitute for the dimension $d_\lambda$, which allows for easier computations. See Section \ref{s:diagram notation} for more details.

Another object introduced in \cite{LarsenShalev2008} is the orbit growth exponent $E(\sigma)$. Let $n\geq 1$ and $\sigma\in\mathfrak{S}_n$. Denote by $f_i(\sigma)$ the number of cycles of length $i\geq 1$ of $\sigma$, and for $k\geq 1$, set $\Sigma_k = \Sigma_k(\sigma):= \sum_{1\leq i\leq k} if_i(\sigma)$. We define the orbit growth sequence $(e_i)_{i \geq 1}$ by $n^{e_1+\ldots+e_i}=\max\pg \Sigma_i, 1 \pd$ for all $i \geq 1$, and set
\begin{equation}
\label{eq:orbit growth exponent}
    E(\sigma) = \sum_{i\geq 1} \frac{e_i}{i}.
\end{equation}

A key intermediate result that Larsen and Shalev proved is the following, which we will refer to as a theorem.

\begin{theorem}[\cite{LarsenShalev2008}]\label{thm:LS concrete bound}
    For any $n\geq 1$, $\sigma \in \mathfrak{S}_n$ and $\lambda \in \widehat{\mathfrak{S}_n}$, we have the following character bound:
    \begin{equation}
        \bg \ch^\lambda(\sigma)\bd\leq D(\lambda)^{E(\sigma)}.
    \end{equation}
\end{theorem}

This is proved via an elegant induction in the proof of their main theorem (for reference, the induction hypothesis is \cite[Eq. (17)]{LarsenShalev2008}), where the orbit growth exponent $E(\sigma)$ naturally appears. They also showed that $D(\lambda)$ cannot be much larger than $d_\lambda$.

\begin{theorem}[\cite{LarsenShalev2008}, Theorem 2.2]\label{thm:LS virtual degree asymptotic bound}
    As $|\lambda|\to \infty$, we have
    \begin{equation}
        D(\lambda) \leq d_\lambda^{1+o(1)}.
    \end{equation}
\end{theorem}

Combining Theorem \ref{thm:LS concrete bound}, Theorem \ref{thm:LS virtual degree asymptotic bound} and the fact that $E(\sigma) \leq 1$ for all $\sigma$, they obtain the following character bound.

\begin{theorem}[\cite{LarsenShalev2008}, Theorem 1.1 (a)]\label{thm:LS form E(sigma) + o(1)}
    As $n\to \infty$, we have uniformly over all $\sigma \in \mathfrak{S}_n$ and $\lambda \in \widehat{\mathfrak{S}_n}$, 
    \begin{equation}
        \bg \ch^\lambda(\sigma)\bd \leq d_\lambda^{E(\sigma) + o(1)}.
    \end{equation}
\end{theorem}

The orbit growth exponent $E(\sigma)$ defined in \eqref{eq:orbit growth exponent} looks complicated at first glance, but it is actually simple to compute and appears to lead to the best possible character bounds in different regimes such as in \eqref{eq:bound Larsen Shalev on the Lulov Pak conjecture} below.

Let us give a few important examples of values of $E(\sigma)$. We always denote by $f_1$ the number of fixed points of a permutation $\sigma$ and set $f=\max(f_1,1)$.
\begin{example}\label{ex:values of E(sigma)}
\begin{enumerate}
    \item Assume that $\sigma \sim [2^{n/2}]$, then $e_2 = 1$ and $e_i = 0$ for $i\ne 2$, so $E(\sigma) = 1/2$.
    \item Assume that $\sigma$ has no fixed point, then $e_2+e_3+\ldots = 1$ so $E(\sigma) \leq 1/2$.
    \item More generally, for any $\sigma \in \kS_n$ we have $e_1 = \frac{\ln f}{\ln n}$, so $E(\sigma) \leq e_1 + \frac{1}{2}(1-e_1) = \frac{1+e_1}{2} = \frac{1}{2} + \frac{\ln f}{2\ln n}$, that is, $E(\sigma) - 1 \leq -\frac{1}{2}\frac{\ln (n/f)}{\ln n}$
    \item Assume that $\sigma\sim [k,1^{n-k}]$, then $E(\sigma) = 1/n$ if $k=n$ (i.e.\ $\sigma$ is an $n$-cycle) and if $f_1\geq 1$, $E(\sigma) = e_1+\frac{e_k}{k} = e_1 + \frac{1-e_1}{k} = \frac{\ln f}{\ln n} + \frac{\ln(n/f)}{k\ln n}$.
\end{enumerate}
\end{example}

Combining Theorem \ref{thm:LS form E(sigma) + o(1)} with Example \ref{ex:values of E(sigma)} (b), Larsen and Shalev obtained the following bound as $n\to \infty$, uniform over fixed-point free permutations $\sigma \in \mathfrak{S}_n$ and irreducible representations $\lambda \in \widehat{\mathfrak{S}_n}$,
\begin{equation}\label{eq:bound Larsen Shalev on the Lulov Pak conjecture}
    \bg\ch^\lambda(\sigma)\bd \leq d_\lambda^{1/2+o(1)}.
\end{equation}
This allowed them to prove the above-mentioned conjecture of Lulov and Pak, and even extend it (in Theorem 1.8) to the mixing time of conjugacy classes that have a small number of fixed points.

\subsection{Main results}\label{s: main results}

We present here our main contributions: the main result, Theorem \ref{thm:ALS virtual degree asymptotic bound}, which improves on Theorem \ref{thm:LS virtual degree asymptotic bound}, provides uniform bounds on virtual degrees and allows us to obtain refined character bounds (Theorem \ref{thm:ALS improved character bound}). Proposition \ref{prop: bounds witten zeta intro iff} provides bounds on the Witten zeta function, which we apply together with Theorem \ref{thm:ALS improved character bound} to characterize fixed-point free conjugacy classes that mix in 2 steps (Theorem \ref{thm: application Lulov Pak extended}). We also deduce from our bounds that conjugacy classes with $o(\sqrt{n})$ cycles mix fast (Theorem \ref{thm: mixing time 2 if few cycles}), and finally that there is a cutoff whenever the conjugacy class has a large support and no short cycles (Theorem \ref{thm: cutoff large support no short cycles}).

\begin{theorem}
    \label{thm:ALS virtual degree asymptotic bound}
    There exists a universal constant $C$ such that, for every $n\geq 2$ and any integer partition $\lambda\vdash n$, we have
    \begin{equation}\label{eq:virtual degree asymptotic bound}
        D(\lambda) \leq d_\lambda^{1+\frac{C}{\ln n}}.
    \end{equation}
\end{theorem}

We will show in Section \ref{s:examples} that Theorem \ref{thm:ALS virtual degree asymptotic bound} is sharp for different shapes of diagrams. 

\medskip

Plugging this into Theorem \ref{thm:LS concrete bound}, we get the following character bound.

\begin{theorem}\label{thm:ALS improved character bound}
    For every $n\geq 2$, any permutation $\sigma \in \mathfrak{S}_n$ and any integer partition $\lambda\vdash n$, we have
    \begin{equation}
        \bg \ch^\lambda(\sigma)\bd \leq d_\lambda^{\pg 1+ \frac{C}{\ln n} \pd E(\sigma)},
    \end{equation}
where $C$ is the universal constant from Theorem $\ref{thm:ALS virtual degree asymptotic bound}$.
\end{theorem}

In Proposition \ref{prop: borne E sigma avec max cyc 2} we prove that $E(\sigma) \leq \frac{\ln \left(\max(\cyc(\sigma),2)\right)}{\ln n}$, where $\cyc(\sigma)$ is the number of cycles of $\sigma$. This allows us to deduce the following character bound, which improves on \cite[Theorem 1.4]{LarsenShalev2008} and \cite[Theorem 1.7]{LifshitzMarmor2024charactershypercontractivity}.

\begin{theorem}\label{thm: borne caractères nombre de cycles}
    As $n\to \infty$, uniformly over all $\sigma \in \kS_n$ and $\lambda \in \widehat{\kS_n}$, we have
    \begin{equation}
        \bg \ch^\lambda(\sigma)\bd \leq d_\lambda^{\frac{\ln \cyc(\sigma)}{\ln n} + O\pg \frac{1}{\ln n} \pd}.
    \end{equation}
\end{theorem}

We now define the Witten zeta function. For $n\geq 1$, $A \subset \widehat{\mathfrak{S}_n}$ and $s \geq 0$,
\begin{equation}\label{eq:def witten zeta}
    \zeta_n (A,s) := \sum_{\lambda\in A} \frac{1}{\pg d_\lambda\pd^{s}}.
\end{equation}
Note that compared to the usual definition where the sum is over all irreducible representations, we add the subset of representations $A$ on which we sum as a parameter of the function $\zeta_n$.

In \cite{LiebeckShalev2004}, Liebeck and Shalev proved that $\zeta_n(\widehat{\mathfrak{S}_n}^{**},s) = O(n^{-s})$, if $s>0$ is fixed and $n\to \infty$, where $\widehat{\mathfrak{S}_n}^{**} = \widehat{\mathfrak{S}_n}\backslash\ag [n], [1^n] \ad = \ag \lambda \vdash n \mid d_\lambda >1 \ad$. This was later improved by Hough \cite{Hough2016} who showed that, for some large constant $c_0>0$, we have $\zeta_n(\widehat{\mathfrak{S}_n}^{**},c_0/\ln n) = O(1)$.
We prove the following asymptotics for the Witten zeta function. Let $p(n)$ be the number of integer partitions of the integer $n$, with the convention that $p(0) = 1$. Denote the associated generating function by $\Phi$, that is $\Phi(x) = \sum_{n=0}^\infty p(n)x^n$ for $x\in [0,1)$.
\begin{theorem}\label{thm: witten zeta asymptotics generating function}
 For any $c>0$, we have
 \begin{equation}
   \zeta_n(\widehat{\kS_n},c/\ln n) := \sum_{\lambda \in \widehat{\kS_n}} d_\lambda^{-c/\ln n} \xrightarrow[n\to \infty]{} 2 \Phi(e^{-c}).
 \end{equation}
\end{theorem}
We deduce the following convergence result, which could also be obtained as a consequence of Hough’s estimates.
\begin{proposition}\label{prop: bounds witten zeta intro iff}
Let $(s_n)$ be a sequence of positive real numbers. We have
    \begin{equation}
  \zeta_n(\widehat{\mathfrak{S}_n}^{**},s_n)\xrightarrow[n\to \infty]{} 0 \quad \text{ if and only if } \quad s_n\ln n \xrightarrow[n\to \infty]{} \infty.
\end{equation}
\end{proposition}

In Section \ref{s: bounds Witten zeta improvement Liebeck Shalev}, we will also prove useful variants of Proposition \ref{prop: bounds witten zeta intro iff}, for other subsets $A\subset \widehat{\kS_n}$, and when $s_n$ is of order $\frac{1}{\ln n}$.

\medskip

As an application of our bounds on characters and on the Witten zeta function, we are able to characterize which fixed-point free conjugacy classes mix in 2 steps. We recall that the convolution product of two measures $\mu, \nu$ on a group $G$ is defined by $\mu \ast \nu (g) = \sum_{x \in G} \mu(x)\nu(g x^{-1}) $ for $g \in G$. As such, if $S\subset G$, then $\Unif_{S}^{*t}$ is the distribution of the simple random walk (started at the neutral element) on the Cayley graph $\Cay(G, S)$ after $t$ steps. Finally, the total variation distance between two probability measures $\mu$ and $\nu$ on a finite set $E$ is given by
\begin{equation}
    \dtv(\mu, \nu) = \frac{1}{2} \sum_{x\in E} |\mu(x)-\nu(x)|.
\end{equation}
\begin{theorem}\label{thm: application Lulov Pak extended}
For each $n\geq 2$ let $\cC^{(n)}$ be a conjugacy class of $\kS_n$, which is fixed-point free (i.e.\ $f_1(\cC^{(n)})=0$). Recall that $f_2(\cC^{(n)})$ denotes the number of transpositions of a permutation $\sigma \in \cC^{(n)}$. Then we have

\begin{equation}
    \dtv\pg \Unif_{\cC^{(n)}}^{*2}, \Unif_{\kA_n} \pd \xrightarrow[n\to \infty]{} 0 \quad \quad \text{ if and only if } \quad \quad  f_2( \cC^{(n)} ) = o(n).
\end{equation}

\end{theorem}

We also prove that conjugacy classes with few cycles mix fast.

\begin{theorem}\label{thm: mixing time 2 if few cycles}
For each $n\geq 2$ let $\cC^{(n)}$ be a conjugacy class of $\kS_n$. Assume that $\cyc(\sigma) = o\pg \sqrt{n} \pd$ for $\sigma \in \cC^{(n)}$. Then
\begin{equation}
    \dtv\pg \Unif_{\cC^{(n)}}^{*2}, \Unif_{\kA_n} \pd \xrightarrow[n\to \infty]{} 0.
\end{equation}
\end{theorem}

We believe that the only parts of the cycle structure that affect mixing times are the number of fixed points and the number of transpositions. We make the following conjecture.

\begin{conjecture}\label{conj: characterization mixing time 2}
    For each $n\geq 2$ let $\cC^{(n)}$ be a conjugacy class of $\kS_n$. Then
\begin{equation}
    \dtv\pg \Unif_{\cC^{(n)}}^{*2}, \Unif_{\kA_n} \pd \xrightarrow[n\to \infty]{} 0  \quad \text{ if and only if } \quad [f_1 = o(\sqrt{n}) \text{ and } f_2 = o(n)],
\end{equation}
where $f_1 = f_1(\cC^{(n)})$ and $f_2 = f_2(\cC^{(n)})$.
\end{conjecture}

As a third application, we finally show that conjugacy classes with a large support and no short cycles exhibit cutoff. If $\cC$ is a conjugacy class of $\kS_n$ and $t\geq 0$, we denote by $\mathfrak{E}(\cC, t)$ 
the set $\mathfrak{S}_n \backslash\mathfrak{A}_n$ if $\cC \subset \mathfrak{S}_n \backslash\mathfrak{A}_n$ and $t$ is odd, and $\mathfrak{A}_n$ otherwise.

\begin{theorem}\label{thm: cutoff large support no short cycles}
    For each $n\geq 2$ let $\cC^{(n)} \ne \ag \Id\ad$ be a conjugacy class of $\kS_n$. Write $f = \max(f_1(\cC^{(n)}),1)$. 
    Assume that $f=o(n)$ and $\min \ag i\geq 2 \du f_i(\cC^{(n)}) \geq 1 \ad \xrightarrow[n\to \infty]{} \infty$. Then for every $\varepsilon>0$, we have
    \begin{equation}
       \mathrm{d}^{(n)}\pg \lf \frac{\ln n}{\ln(n/f)}(1-\varepsilon)\rf\pd \xrightarrow[n\to \infty]{} 1 \quad \text{ and } \quad  \mathrm{d}^{(n)}\pg \lc \frac{\ln n}{\ln(n/f)}(1+\varepsilon)\rc\pd \xrightarrow[n\to \infty]{} 0,
    \end{equation}
where $ \mathrm{d}^{(n)}(t) := \dtv\pg \Unif_{\cC^{(n)}}^{*t}, \Unif_{\kE(\cC^{(n)}, t)}\pd$ for $t\geq 0$.
\end{theorem}

\begin{remark}
Our improved bounds are useful in several ways. They lead to sharp results for permutations with few fixed points or few cycles, as we saw in Theorem \ref{thm: application Lulov Pak extended} and Theorem \ref{thm: mixing time 2 if few cycles}. They also extend the applicability of the Larsen--Shalev character bounds (recalled as Theorem \ref{thm:LS form E(sigma) + o(1)}) from $f=n^{1-\varepsilon(n)}$ (where $\varepsilon(n)$ is the non-explicit $o(1)$ from Theorem \ref{thm:LS form E(sigma) + o(1)}) to $f=o(n)$. In the regime $f=o(n)$, recalling Example \ref{ex:values of E(sigma)} (c), this leads to the uniform bound
\begin{equation}
\label{eq:bound with constant 1/2}
    \frac{|\ch^\lambda(\sigma)|}{d_\lambda} \leq d_\lambda^{-\pg\frac{1}{2}+o(1)\pd \frac{\ln (n/f)}{\ln n}}.
\end{equation}
The uniform constant $1/2$ above cannot be improved, as it is sharp for fixed-point free involutions. In the follow-up paper \cite{Olesker-TaylorTeyssierThevenin2025sharpboundsandcutoff}, we prove that \eqref{eq:bound with constant 1/2} holds uniformly over all permutations and irreducible representations.
\end{remark}

As a last application, we prove mixing results on products with fixed-point free involutions. The mixing of such products is strongly motivated by the theory of
random maps; the connection is detailed in Section \ref{s: random maps}.

\begin{theorem}\label{thm: products of two permutations and random maps}
For each $n\geq 2$ even, let $\cC_1^{(n)}$ be a fixed-point free conjugacy class of $\kS_n$ and denote by $\cC_2^{(n)}$ the conjugacy class of fixed-point free involutions (that is, of cycle type $[2^{n/2}]$).  Then

\begin{equation}
    \dtv\pg \Unif_{\cC_1^{(n)}}*\Unif_{\cC_2^{(n)}}, \Unif_{\kE((\cC_1^{(n)}, \cC_2^{(n)}))} \pd \xrightarrow[n\to \infty]{} 0 \quad \text{ if and only if } \quad f_2( \cC_1^{(n)} ) = o(n),
\end{equation}
where $\kE((\cC_1^{(n)}, \cC_2^{(n)}))$ is the coset of $\kA_n$ with sign $\sgn(\cC_1^{(n)})\sgn(\cC_2^{(n)})$.
\end{theorem}

This extends the result of \cite[Theorem 2.2]{ChmutovPittel2016}. Our theorem implies in particular the convergence in distribution of the number of vertices in a certain model of random maps (see Theorem \ref{thm: number of vertices in random maps}), improving \cite[Theorem 3.1]{ChmutovPittel2016} and recovering the result of \cite[Theorem 3]{BudzinskiCurienPetri2019} under a stronger assumption. Note that Theorem \ref{thm: products of two permutations and random maps} is strictly stronger than Theorem \ref{thm: number of vertices in random maps}, and therefore is not contained in \cite[Theorem 3]{BudzinskiCurienPetri2019}.

\subsection{Structure of the paper}

Section \ref{s:preliminaries} is devoted to preliminaries on Young diagrams and some elementary results. There we recall in particular the notions of hook length and hook product, and the hook length formula.

In Section \ref{s:sliced hook products}, we introduce \textit{sliced hook products} and \textit{augmented dimensions}, and analyze
their properties.

We prove Theorem \ref{thm:ALS virtual degree asymptotic bound} in Section \ref{s:proof of the main result} and show that the result is sharp up to the value of the constant.

Section \ref{s: Character bounds in terms of the number of cycles} is devoted to the proof of Theorem \ref{thm: borne caractères nombre de cycles}, which shows character bounds in terms of the number of cycles.

We improve the bounds from Liebeck and Shalev \cite{LiebeckShalev2004} on the Witten zeta function in Section \ref{s: bounds Witten zeta improvement Liebeck Shalev}, proving Proposition \ref{prop: bounds witten zeta intro iff}.

In Section \ref{s: Products of two conjugacy invariant random permutations} we prove sufficient conditions for products of two conjugacy invariant random permutations to be close to being uniform.

In Section \ref{s: beyond the lulov pak conjecture}, we apply the previous results and characterize which fixed-point free conjugacy classes mix in 2 steps, proving Theorem \ref{thm: application Lulov Pak extended}.

Section \ref{s: random maps} is devoted to the proof of Theorem \ref{thm: products of two permutations and random maps} and the study of random maps.

Finally, in Section \ref{s: further mixing time estimates} we prove Theorems \ref{thm: mixing time 2 if few cycles} and \ref{thm: cutoff large support no short cycles} about mixing time estimates for specific conjugacy classes.

\section{Preliminaries}\label{s:preliminaries}

\subsection{Young diagrams, representations and integer partitions}
\label{ssec:young diagrams representations and partitions}

We say that $\lambda := [\lambda_1, \ldots, \lambda_k]$ is a \textbf{partition} of the integer $n$ if $\lambda_i \in \mathbb{Z}_{\geq 1} := \{1, 2, \ldots \}$ for all $1 \leq i \leq k$, $\lambda_i \geq \lambda_{i+1}$ for all $i \leq k-1$, and $\sum_{i=1}^k \lambda_i = n$. We write $\lambda \vdash n$ if $\lambda$ is a partition of $n$. It turns out that integer partitions of $n$ are in bijection with irreducible representations of $\kS_n$, which makes them a good tool to study characters. 

A common and useful way to code an integer partition (and thus a representation of $\kS_n$) is through \textbf{Young diagrams}. The Young diagram of shape $\lambda := \cg \lambda_1, \ldots, \lambda_k \cd$ (where $\lambda$ is a partition of an integer $n$) is an array of boxes, whose $i$-th row is made of $\lambda_i$ boxes, see Figure \ref{fig:youngdiagram521} for an example.

\begin{figure}[!ht]
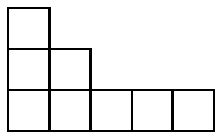

\begin{center}
\begin{ytableau}
    \none   & \\
    \none   &  & \\
    \none   &  &  &  &  & \\
    \none   & \none & \none & \none
\end{ytableau}
\end{center}
\caption{The Young diagram coding the partition $[5,2,1]$ of the integer $8$. It has 5 boxes on the first row, 2 on the second, and 1 on the third.}
\label{fig:youngdiagram521}
\end{figure}

We also denote by $\lambda$ this Young diagram, and say that $n$ is its size.

\subsection{Hook lengths, hook products and the hook length formula}
Let $\lambda$ be a Young diagram. If $u \in \lambda$, the \textbf{hook} of $u$ in $\lambda$ is the set of boxes on the right or above $u$, including $u$. We call hook length the size of this set and denote it by $H(\lambda, u)$. See Figure \ref{fig:hooklength} for an example.

\begin{figure}[!ht]
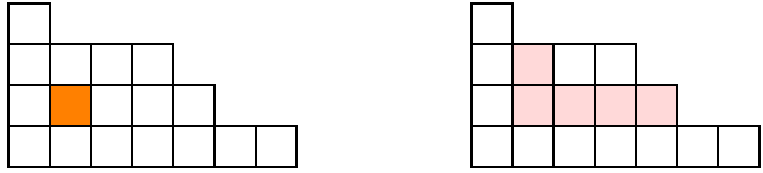

\begin{center}
\begin{ytableau}
    \none   & \\
    \none   & & & & \\
    \none   & & *(orange) & & & \\
    \none   & & & & & & & & \none&\none&\none
\end{ytableau}
\begin{ytableau}
    \none   & \\
    \none   & & *(pink!60) & & \\
    \none   & & *(pink!60)& *(pink!60) & *(pink!60)& *(pink!60)\\
    \none   & & & & & & & 
\end{ytableau}
\end{center}
\caption{
    Left: the Young diagram associated to $\lambda = [7,5,4,1]$, with a box $u$ colored in orange.  Right: the hook associated to $u$ is in pink. Its length is $H(\lambda, u) = 5$.}
    \label{fig:hooklength}
\end{figure}

Let us now consider a subset of boxes $E\subset \lambda$. We define
\begin{equation}
    H(\lambda, E) := \prod_{u\in E} H(\lambda, u)
\end{equation}
the \textit{hook product of $E$ in $\lambda$}. See Figure \ref{fig:hookproduct} for an example.

\begin{figure}[!ht]
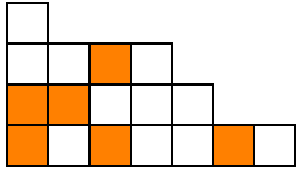

\begin{center}
\begin{ytableau}
    \none   & \\
    \none   & & &*(orange) & \\
    \none   &*(orange) & *(orange) & & & \\
    \none   & *(orange) & & *(orange)& & & *(orange)& & \none&\none&\none
\end{ytableau}
\end{center}
\caption{The Young diagram coding $\lambda = [7,5,4,1]$, with the set $E = \ag (1,1), (3,1), (6,1), (1,2),(2,2), (3,3)\ad$ in orange. Here, $H(\lambda, E) = 10\cdot 7\cdot 2 \cdot 7\cdot 5 \cdot 2=9800$.}
\label{fig:hookproduct}
\end{figure}

We can now state the hook length formula. Let $\lambda$ be a Young diagram and $n:=|\lambda|$ denote its number of boxes. A standard tableau of $\lambda$ is a filling of $\lambda$ with the numbers from 1 to $n$ such that the numbers are increasing on each row and column. We denote by $\ST(\lambda)$ the set of standard tableaux of $\lambda$, and $d_\lambda := |\ST(\lambda)|$, see Figure \ref{fig:standardtableaux}. It is well-known that $d_\lambda$ is the dimension of the irreducible representation of the symmetric group associated to $\lambda$.
\begin{figure}[!ht]
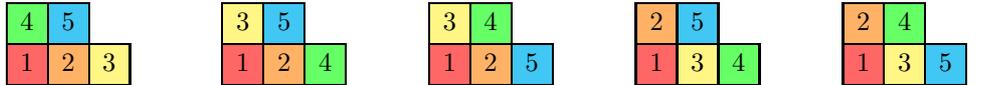

\begin{center}
\begin{ytableau}
    \none   & *(green!60)4 & *(cyan!60)5 \\
    \none   &*(red!60)1& *(orange!60)2 & *(yellow!60)3& \none
\end{ytableau}
\begin{ytableau}
    \none   & *(yellow!60)3 & *(cyan!60)5\\
    \none   & *(red!60)1 & *(orange!60)2& *(green!60)4& \none
\end{ytableau}\begin{ytableau}
     \none   & *(yellow!60)3& *(green!60)4\\
    \none   & *(red!60)1 & *(orange!60)2& *(cyan!60)5& \none
\end{ytableau}\begin{ytableau}
    \none   & *(orange!60)2 & *(cyan!60)5\\
    \none   & *(red!60)1 & *(yellow!60)3 & *(green!60)4 & \none
\end{ytableau}\begin{ytableau}
    \none   & *(orange!60)2 & *(green!60)4 \\
    \none   & *(red!60)1 & *(yellow!60)3 & *(cyan!60)5 & \none[.]
\end{ytableau}
\end{center}
\caption{
    The diagram $\lambda = [3,2]$ and its $d_\lambda = |\ST(\lambda)|= 5$ standard tableaux.}
\label{fig:standardtableaux}
\end{figure}
The hook length formula was discovered by Frame, Robinson, and Thrall \cite{FrameRobinsonThrall1954}, and enables computing the number of standard tableaux of a diagram looking only at its hook lengths:
for any diagram $\lambda$ of size $n$, we have
\begin{equation}
    d_\lambda = \frac{n!}{H(\lambda,\lambda)}.
\end{equation}
\begin{example}
    Consider again $\lambda = [3,2]$, which has size 5. We have $\frac{n!}{H(\lambda, \lambda)} = \frac{5!}{4\cdot 3 \cdot 1\cdot 2 \cdot 1} = \frac{120}{24} = 5$. Therefore, by the
    hook length formula, we recover that $d_\lambda = 5$.
\end{example}

\subsection{Diagram notation}\label{s:diagram notation}
Let us define some notation that we will use for diagrams. To keep the notation light, we will usually use the same pieces of notation for both the diagrams and their lengths. For example we will say the first row of $\lambda$ is $\lambda_1$ and has length $\lambda_1$. If we want to emphasize that we are considering a diagram or a size, we will add brackets or absolute values. For example we can say that the first row $\cg \lambda_1 \cd$ has length $\bg \lambda_1 \bd$.

\begin{definition}[Definition-notation]
\label{def:definitionnotation}
Let $\lambda$ be a Young diagram.    
\begin{itemize}
    \item We denote by $\lambda_i$ the $i$-th row of $\lambda$ and by $\lambda_i'$ its $i$-th column.
    \item We denote by $\delta_i$ the $i$-th diagonal box of $\lambda$, that is, $\delta_i = \lambda_i \cap \lambda'_i$.
    \item We denote by $\delta = \delta(\lambda)$ the diagonal of $\lambda$, that is, $\delta = \cup_i \delta_i$. 
    \item We denote by $r$ the truncated diagram (obtained by removing the first row of $\lambda$), by $s = s(\lambda) := \lambda_1 \cup \lambda_1'$ the external hook of $\lambda$, and by $c= \lambda \backslash s$ its center (i.e.\ the diagram, except the first row and the first column).
    \item We define $\lambda_{\geq i} := \cup_{j\geq i} \lambda_j$ and $\lambda_{\leq i}:=\cup_{j\leq i} \lambda_j$.
    \item Similarly, we define $\lambda'_{\geq i} := \cup_{j\geq i} \lambda'_j$, $\lambda'_{\leq i}:=\cup_{j\leq i} \lambda'_j$, and $\lambda'_{i_1 \to i_2} := \cup_{i_1 \leq j\leq i_2} \lambda'_j$.
    \item We denote by $\lambda_a^i = a_i := \lambda_i \cap \lambda'_{\geq i+1}$ and $\lambda_b^i = b_i := \lambda'_i \cap \lambda_{\geq i+1}$ the truncated $i$-th rows and columns of $\lambda$.
    \item We denote by $\lambda^i$ the $i$-th hook of $\lambda$, that is the boxes that are on the right of $\delta_i$ on $\lambda_i$ or above $\delta_i$ on $\lambda'_i$. Formally we can write $\lambda^{i} = (\lambda_i \cap \lambda'_{\geq i}) \cup (\lambda_{\geq i} \cap \lambda'_i)$.
    \item We define $\lambda^{\geq i} := \cup_{j\geq i} \lambda^j$, $\lambda^{\leq i}:=\cup_{j\leq i} \lambda^j$.
    \item We denote by $u_i$ the part of the $i$-th column which is above the first row. That is $u_i = \lambda_i' \cap \lambda_{\geq 2}$.
    \item We define $u_{\geq i} := \cup_{j\geq i} u_j$, $u_{\leq i}:=\cup_{j\leq i} u_j$.
    \item 
    We denote by $\lambda_{\geq i,\geq j}$ the subdiagram of $\lambda$ whose boxes are on a row after the $i$-th and on a column after the $j$-th. 
\end{itemize}
\end{definition}

Figure \ref{fig:illustrationsdiagrams} illustrates these definitions on diagrams.

\begin{figure}
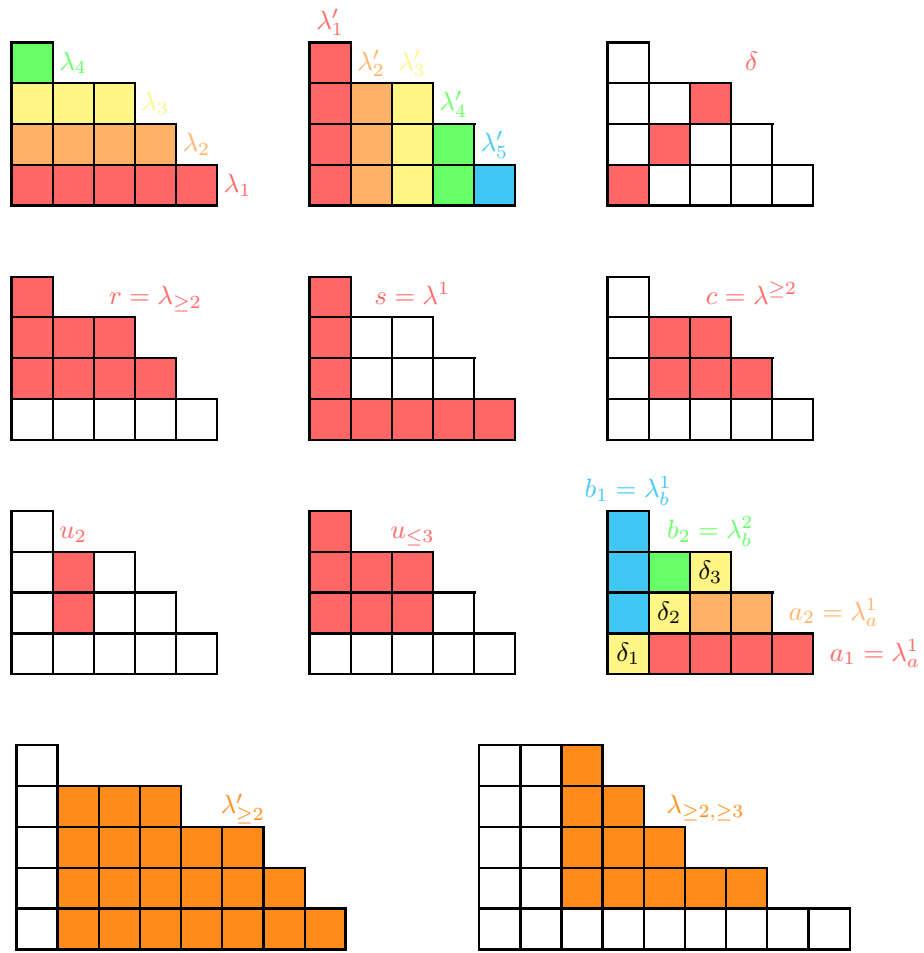

\begin{center}
    \begin{ytableau}
        \none \\
        \none &*(green!60)&\none[\textcolor{green!60}{\lambda_4}]\\
        \none &*(yellow!60)&*(yellow!60)&*(yellow!60)&\none[\textcolor{yellow!60}{\lambda_3}]\\
        \none &*(orange!60)&*(orange!60)&*(orange!60)&*(orange!60) &\none[\textcolor{orange!60}{\lambda_2}]\\
        \none &*(red!60)&*(red!60)&*(red!60)&*(red!60)&*(red!60) &\none[\textcolor{red!60}{\lambda_1}]
    \end{ytableau}
    \begin{ytableau}
        \none &\none[\textcolor{red!60}{\lambda_1'}]\\
        \none &*(red!60)&\none[\textcolor{orange!60}{\lambda_2'}]&\none[\textcolor{yellow!60}{\lambda_3'}]\\
        \none &*(red!60)&*(orange!60) &*(yellow!60)&\none[\textcolor{green!60}{\lambda_4'}]\\
        \none &*(red!60)&*(orange!60)&*(yellow!60)&*(green!60)&\none[\textcolor{cyan!60}{\lambda_5'}]\\
        \none &*(red!60)&*(orange!60)&*(yellow!60) &*(green!60) &*(cyan!60) &\none
    \end{ytableau}
    \begin{ytableau}
        \none \\
        \none & &\none &\none&\none[\textcolor{red!60}{\delta}]\\
        \none &&&*(red!60)&\none\\
        \none &&*(red!60)&&\\
        \none &*(red!60)&&&& &\none
    \end{ytableau}
\end{center}

\begin{center}
     \begin{ytableau}
        \none \\
        \none &*(red!60) &\none &\none&\none[\textcolor{red!60}{r = \lambda_{\geq 2}}]\\
        \none &*(red!60)&*(red!60)&*(red!60)\\
        \none &*(red!60)&*(red!60)&*(red!60)&*(red!60) \\
        \none &&&&& &\none
    \end{ytableau}
    \begin{ytableau}
        \none  \\
        \none &*(red!60) &\none &\none[\textcolor{red!60}{s = \lambda^1}]&\none\\
        \none &*(red!60)&&&\none\\
        \none &*(red!60)&&&\\
        \none &*(red!60)&*(red!60)&*(red!60) &*(red!60) &*(red!60) &\none 
    \end{ytableau}
    \begin{ytableau}
        \none \\
        \none &&\none &\none&\none[\textcolor{red!60}{c = \lambda^{\geq 2}}]\\
        \none &&*(red!60)&*(red!60)&\none\\
        \none &&*(red!60)&*(red!60)&*(red!60)\\
        \none &&&&& &\none
    \end{ytableau}
\end{center}

\begin{center}
    \begin{ytableau}
        \none &\none   \\
        \none &\none &  &\none[\textcolor{red!60}{u_2}]\\
        \none &\none&&*(red!60)&&\none\\
        \none &\none&&*(red!60)&&\\
        \none &\none&&&&& &\none
    \end{ytableau}
     \begin{ytableau}
        \none  \\
        \none &*(red!60) &\none &\none[\textcolor{red!60}{u_{\leq 3}}]\\
        \none &*(red!60)&*(red!60)&*(red!60)&\none\\
        \none &*(red!60)&*(red!60)&*(red!60)&\\
        \none &&&&& &\none 
    \end{ytableau}  
 \begin{ytableau}
        \none &\none[\textcolor{cyan!60}{b_1 = \lambda_b^1}]\\
        \none &*(cyan!60) &\none &\none[\textcolor{green!60}{b_2 = \lambda_b^2}] &\none&\none\\
        \none &*(cyan!60)&*(green!60)&*(yellow!60) \delta_3&\none\\
        \none &*(cyan!60)&*(yellow!60)\delta_2&*(orange!60)&*(orange!60)&\none &\none[\textcolor{orange!60}{a_2 = \lambda_a^2}]\\
        \none &*(yellow!60)\delta_1&*(red!60)&*(red!60)&*(red!60)&*(red!60) &\none &\none[\textcolor{red!60}{a_1 = \lambda_a^1}]
    \end{ytableau}
\end{center}

 \begin{center}
  \begin{ytableau}
    \none    \\
    \none & \\
    \none  &&*(orange!90)&*(orange!90)&*(orange!90) &\none &\none[\textcolor{orange!90}{\lambda'_{\geq 2}}]\\
    \none   &&*(orange!90)&*(orange!90)&*(orange!90) &*(orange!90)&*(orange!90) \\
    \none   &&*(orange!90)&*(orange!90)&*(orange!90) &*(orange!90) &*(orange!90) &*(orange!90)\\
    \none   & &*(orange!90) &*(orange!90)  &*(orange!90)&*(orange!90)&*(orange!90)&*(orange!90)&*(orange!90) &\none &\none
\end{ytableau}
    \begin{ytableau}
        \none \\
        \none & &&*(orange!90)\\
        \none & &&*(orange!90)&*(orange!90) &\none &\none[\textcolor{orange!90}{\lambda_{\geq 2,\geq 3}}]\\
        \none & &&*(orange!90)&*(orange!90)&*(orange!90)\\
        \none & &&*(orange!90)&*(orange!90)&*(orange!90)&*(orange!90)&*(orange!90)\\
        \none & &&&&&&&& 
    \end{ytableau}
\end{center}
    \caption{Examples of Definition \ref{def:definitionnotation}.}
    \label{fig:illustrationsdiagrams}
\end{figure}

\subsection{Some elementary results}

We collect here standard or elementary results that we will use throughout the paper. We give proofs for completeness.

\begin{lemma}\label{lem:standard or elementary results}
    \begin{enumerate}
        \item For any $n\geq 1$ and $\lambda \vdash n$ we have $d_\lambda \leq \sqrt{n!}$.
    \item For $n \geq 6$, we have $n! \leq (n/2)^n$.
    \item Let $n\geq 1$ and $\lambda\vdash n$. Denote by $s = s(\lambda)$ the external hook of $\lambda$. Then $d_s  = \binom{s-1}{\lambda_a^1}$.
    \end{enumerate}
\end{lemma}

\begin{proof}
    \begin{enumerate}
        \item A classical formula from representation theory is the following: for any finite group $G$ we have $|G| = \sum_{\rho \in \widehat{G}} d_\rho^2$, where $\widehat{G}$ is the set of all irreducible representations of $G$. This identity represents for example the dimensions on both sides of the Fourier isomorphism (see \cite[Section 1.3]{LivreMeliot2017RepresentationTheoryofSymmetricGroups}). Hence, for any $\lambda \in \widehat{G}$ of a finite group $G$ we have 
\begin{equation}
    d_\lambda^2 \leq \sum_{\rho \in \widehat{G}} d_\rho^2 = |G|. 
\end{equation}
The result follows taking square roots on both sides, since for $n\geq 1$ we have $|\mathfrak{S}_n| = n!$.
\item We proceed by induction. The result holds for $n=6$ since $6! = 720 \leq 729 = 3^6 = (6/2)^6$. Let now $n\geq 6$ and assume that $n! \leq (n/2)^n$. Then
\begin{equation}
    \frac{(n+1)!}{((n+1)/2)^{n+1}} = \frac{n!}{(n/2)^n} (n+1) \frac{2}{n+1}\pg\frac{n}{n+1}\pd^n \leq 2 \pg\frac{n}{n+1}\pd^n.
\end{equation}
Since $\ln$ is concave, we have $\ln(1+x) \geq x\ln 2$ for $0\leq x\leq 1$, and therefore, we have for $n\geq 1$
\begin{equation}
    2 \pg\frac{n}{n+1}\pd^n = 2 e^{-n\ln(1+1/n)}\leq 2 e^{-n\frac{\ln 2}{n}} = 1,
\end{equation}
concluding the proof.
\item Recall that $d_s$ is the number of standard tableaux of $s$. Observe that a standard tableau necessarily has 1 placed in the bottom-left corner. Furthermore, since $s$ is a hook, the rest of the standard tableau is determined by which numbers we chose to place on the first row, for which there are $\binom{s-1}{\lambda_a^1}$ possibilities. \qedhere
    \end{enumerate}
\end{proof}

\section{Sliced hook products}\label{s:sliced hook products}

We introduce here a new tool in the study of Young diagrams, which we call \textit{sliced hook products}. Sliced hook products are defined in Section \ref{s: sliced hook products definitions}. The idea is to replace the whole hook product of a diagram by a simpler product.

We study these sliced hook products in Section \ref{s: sliced hook products bounds} and obtain precise approximations of (classical) hook products. The key result is the bound 
\begin{equation}
    d_\lambda \geq \binom{n}{|c|}d_s d_c e^{-6\sqrt{|c|}}
\end{equation}
presented in Proposition \ref{prop: borne abdelta et dimension centre}, which improves on \cite[Lemma 2.1]{LarsenShalev2008}. Here, $s$ denotes the external hook of $\lambda$ and $c=\lambda\backslash s$ is the center of $\lambda$, as defined in Section \ref{s:diagram notation} and drawn in Figure \ref{fig:illustrationsdiagrams}.

In Section \ref{s: virtual degrees and augmented dimensions}, we introduce the \textit{augmented dimension} $d_\lambda^+$, which is hybrid between $d_\lambda$ and $D(\lambda)$ and well-suited for inductions, and consider some of its useful properties.

\subsection{Definitions}\label{s: sliced hook products definitions}

Despite being elegant, the hook length formula may be tricky to use, because of how hard it is to estimate hook products. A powerful idea to approximate hook products is to take into account the contribution to hook lengths of only some boxes in the hook of a box, making the resulting hook product easier to compute and to use.

We first extend the definition of hook lengths to any set of boxes. We represent boxes by the coordinates of their top right corner in the plane $\bbZ^2$. As such a box can be seen as an element of $\bbZ^2$ and a set of boxes is a subset $S\subset \bbZ^2$, see Figure \ref{fig:hooklengthsfragmented}.

\begin{definition}
    Let 
    $S \subset \bbZ^2$ be a finite set of boxes, and let $u = (x,y) \in \bbZ^2$ (not necessarily in $S$). We define the hook length of $u$ with respect to $S$ by
\begin{equation}
    H(S, u) = \bg \ag (x',y') \in S \du [x=x' \text{ and } y\leq y'] \text{ or } [y=y' \text{ and } x\leq x'] \ad \bd.
\end{equation}
\end{definition}

\begin{figure}
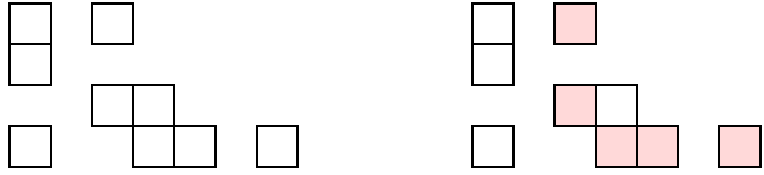

\begin{center}
\begin{ytableau}
    \none & &\none & \\
    \none & \\
    \none &\none &\none && \\
    \none & &\none&\none &&&\none & &\none &\none &\none 
\end{ytableau}
\begin{ytableau}
    \none & &\none  &*(pink!60) \\
    \none & \\
    \none &\none &\none &*(pink!60)& \\
    \none & &\none&\none &*(pink!60)&*(pink!60)&\none & *(pink!60) 
\end{ytableau}
\end{center}

\caption{The hook of $u= (3,1)$ in the set of boxes $S = \ag (1,1), (4,1),(5,1),(7,1), (3,2), (4,2),(1,3),(1,4),(3,4)\ad $. In pink are the boxes in the hook of $u$, and we have $H(S,u) = 5$.}
\label{fig:hooklengthsfragmented}
\end{figure}

\begin{definition}
    Let $\lambda$ be an integer partition and $P = \ag \nu_1, \nu_2, \ldots, \nu_r\ad$ be a set partition of $\lambda$ (i.e.\ such that no $\nu_i$ is empty and $\sqcup_i \nu_i = \lambda$), as represented in Figure \ref{fig:setpartitionsofdiagrams}. We emphasize that $P$ is not a partition of the set $\ag \lambda_1, \lambda_2, \ldots \ad$, but a partition of the set of boxes of the diagram $\lambda$. The hook product sliced along $P$ is the product
\begin{equation}
   H^{*P}(\lambda, \lambda) := \prod_{i= 1}^r H(\nu_i, \nu_i).
\end{equation}

If $u\in \nu_i$, we call the quantity $H(\nu_i, u)$ the \textit{hook length} of $u$ sliced along $P$, or sliced hook length.

More generally, if $E$ is a subset of $\lambda$, we define the following (partial)
sliced hook product:
\begin{equation}
    H^{*P}(\lambda, E) := \prod_{i= 1}^r H(\nu_i, \nu_i\cap E).
\end{equation}
We will refer to “slicing” for the procedure consisting of replacing $H(\lambda, \lambda)$ by $H^{*P}(\lambda,\lambda)$, as well as for the associated partition $P$.
\end{definition}

\begin{figure}
\begin{center}
\begin{ytableau}
    \none &*(red!60)  \\
    \none &*(red!60)&*(orange!60) \\
    \none &*(blue!60)&*(green!60)&*(orange!60) \\
    \none &*(red!60)&*(green!60)&*(red!60)&*(orange!60)&*(red!60)
\end{ytableau}
\end{center}

\caption{The set partition $P= \ag \nu_1, \nu_2, \nu_3, \nu_4 \ad$ of the diagram $\lambda = [5,3,2,1]$, where $\nu_1 = \ag(1,1), (3,1), (5,1), (1,3), (1,4) \ad$ is represented in red, $\nu_2 = \ag (2,1), (2,2) \ad$ in green, $\nu_3 = \ag (4,1), (3,2), (2,3) \ad$ in orange, and $\nu_4 = \ag (1,2) \ad$ in blue.}
\label{fig:setpartitionsofdiagrams}
\end{figure}

Let us give some examples of slicings. We will each time provide an example where the $\nu_i$ are represented in different colors, and the boxes of the diagrams $\lambda$ are filled with their \textit{sliced hook lengths}.

\begin{definition}
    Let $\lambda$ be a Young diagram. 
    \begin{itemize}
        \item  We call $\lambda_1$-slicing (of $\lambda$) a slicing along $\ag\lambda_1, \lambda_{\geq 2}\ad$, see Figure \ref{fig:lambdadown1slicing}.
        \item We call $\lambda^1$-slicing (of $\lambda$) a slicing along $\ag\lambda^1, \lambda^{\geq 2}\ad$, see Figure \ref{fig:lambdaup1slicing}. 
        \item We call $ab\delta$-slicing (of $\lambda$) a slicing along $P:= \ag\lambda^i_a\ad_i \sqcup  \ag\lambda^i_b\ad_i \sqcup\ag\delta_i\ad_i$, see Figure \ref{fig:abdeltaslicing}. We denote the (partial) $ab\delta$-sliced hook products by $H^{*ab\delta}(\lambda, \cdot)$.
    \end{itemize}

\end{definition}

\begin{figure}
\begin{center}
\begin{ytableau}
        \none &*(red!60)1\\
        \none &*(red!60)3&*(red!60)1\\
        \none &*(red!60)4&*(red!60)2\\
        \none &*(red!60)6&*(red!60)4&*(red!60)1 &\none &\none[\textcolor{red!60}{\lambda_{\geq 2}}]\\
        \none &*(red!60)8&*(red!60)6&*(red!60)3&*(red!60)1\\
        \none &*(blue!60)14&*(blue!60)13&*(blue!60)12&*(blue!60)11&*(blue!60)10&*(blue!60)9&*(blue!60)8&*(blue!60)7&*(blue!60)6&*(blue!60)5&*(blue!60)4&*(blue!60)3&*(blue!60)2&*(blue!60)1 &\none[\textcolor{blue!60}{\lambda_1}]
    \end{ytableau}
\end{center}
\caption{The $\lambda_1$-slicing for $\lambda = [14,4,3,2,2,1]$. Numbers in the boxes correspond to the hook lengths after slicing.}
\label{fig:lambdadown1slicing}
\end{figure}

\begin{figure}[!ht]
    \begin{center}
    \begin{ytableau}
        \none &*(blue!60)1\\
        \none &*(blue!60)2&*(red!60)1\\
        \none &*(blue!60)3&*(red!60)2\\
        \none &*(blue!60)4&*(red!60)4&*(red!60)1 &\none &\none[\textcolor{red!60}{\lambda^{\geq 2}}]\\
        \none &*(blue!60)5&*(red!60)6&*(red!60)3&*(red!60)1\\
        \none &*(blue!60)19&*(blue!60)13&*(blue!60)12&*(blue!60)11&*(blue!60)10&*(blue!60)9&*(blue!60)8&*(blue!60)7&*(blue!60)6&*(blue!60)5&*(blue!60)4&*(blue!60)3&*(blue!60)2&*(blue!60)1 &\none[\textcolor{blue!60}{\lambda^1}]
    \end{ytableau}
\end{center}
\caption{The $\lambda^1$-slicing for $\lambda = [14,4,3,2,2,1]$. Numbers in the boxes correspond to the hook lengths after slicing.}
\label{fig:lambdaup1slicing}
\end{figure}

\begin{figure}[!ht]
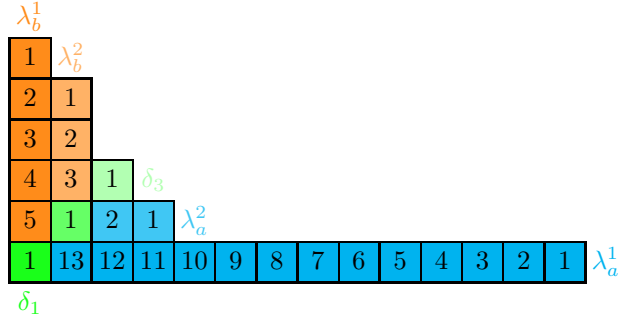

The $ab\delta$-slicing of the diagram $\lambda = [14,4,3,2,2,1]$ can be represented as follows.
    \begin{center}
    \begin{ytableau}
        \none &\none[\textcolor{orange!90}{\lambda_b^1}] \\
        \none &*(orange!90)1 &\none[\textcolor{orange!60}{\lambda_b^2}]\\
        \none &*(orange!90)2&*(orange!60)1\\
        \none &*(orange!90)3&*(orange!60)2\\
        \none &*(orange!90)4&*(orange!60)3&*(green!30)1 &\none[\textcolor{green!30}{\delta_3}]\\
        \none &*(orange!90)5&*(green!60)1&*(cyan!60)2&*(cyan!60)1 &\none[\textcolor{cyan!60}{\lambda_a^2}]\\
        \none &*(green!90)1&*(cyan!90)13&*(cyan!90)12&*(cyan!90)11&*(cyan!90)10&*(cyan!90)9&*(cyan!90)8&*(cyan!90)7&*(cyan!90)6&*(cyan!90)5&*(cyan!90)4&*(cyan!90)3&*(cyan!90)2&*(cyan!90)1 &\none[\textcolor{cyan!90}{\lambda_a^1}] \\
        \none &\none[\textcolor{green!90}{\delta_1}]
    \end{ytableau}
\end{center}

\caption{The $ab\delta$-slicing for $\lambda = [14,4,3,2,2,1]$. Numbers in the boxes correspond to the hook lengths after slicing. If $E$ is the subdiagram $[3,2]$ of $\lambda$ (i.e.\ $E := \{ (1,1), (2,1), (3,1), (1,2), (2,2) \}$), we have $H^{*ab\delta}(\lambda, E) = 1\cdot 13 \cdot 12 \cdot 5\cdot 1 = 780$.}
\label{fig:abdeltaslicing}
\end{figure}

\subsection{Bounds on sliced hook products}\label{s: sliced hook products bounds}
Let us define, for any Young diagrams $\mu \subset \lambda$ and any set partition $P$ of $\lambda$, the ratio 
\begin{equation}
    R_P(\lambda, \mu) := \frac{H(\lambda, \mu)}{H^{*P}(\lambda, \mu)}.
\end{equation}
If $P$ is the $ab\delta$-(resp.\  $\lambda_1$-, $\lambda^1$-)slicing, we denote the corresponding ratio by $R_{ab\delta}$ (resp.\  $R_{\lambda_1}$, $R_{\lambda^1}$).

We start with rewriting the ratio, in the case of a slicing with respect to the first row.
\begin{lemma}
\label{lem:premiers resultats}
If $\mu, \lambda$ are two Young diagrams such that $\mu \subset \lambda$, then we have:
\begin{itemize}
      \item[(i)]  $R_{\lambda_1}(\lambda, \mu) = \frac{H(\lambda, \mu_1)}{H(\lambda_1, \mu_1)}$;
       \item[(ii)] $R_{\lambda_1}(\lambda, \mu) \leq R_{\lambda_1}(\lambda, \lambda)$.
    \end{itemize}    
\end{lemma}

\begin{proof}
Let us first prove (i). Since the first row of $\lambda$ does not appear in the hook products of boxes on the second row and above, we have $H(\lambda,\mu_{\geq 2}) = H(\lambda_{\geq 2},\mu_{\geq 2})$. Therefore,
\begin{equation}
    R_{\lambda_1}(\lambda, \mu) = \frac{H(\lambda, \mu_1)H(\lambda, \mu_{\geq 2})}{H(\lambda_1, \mu_1)H(\lambda_{\geq 2}, \mu_{\geq 2})} = \frac{H(\lambda, \mu_1)}{H(\lambda_1, \mu_1)}.
\end{equation}
Now we prove (ii). For each $u\in \lambda$, we have $H_{\lambda_1}(\lambda, u) \leq H(\lambda, u)$, i.e.\ $\frac{H_{\lambda_1}(\lambda, u)}{H(\lambda, u)}\leq 1$. Using (i), we get that
 \begin{equation}
 \begin{split}
     \frac{R_{\lambda_1}(\lambda, \mu)}{R_{\lambda_1}(\lambda, \lambda)} &= \frac{H(\lambda, \mu_1)/H(\lambda_1, \mu_1)}{H(\lambda, \lambda_1)/H(\lambda_1, \lambda_1)} = \frac{H(\lambda_1, \lambda_1)/H(\lambda_1, \mu_1)}{H(\lambda, \lambda_1)/H(\lambda, \mu_1)} = \frac{H(\lambda_1, \lambda_1\backslash\mu_1)}{H(\lambda, \lambda_1\backslash\mu_1)} \\
     &  = \prod_{u\in \lambda_1 \backslash\mu_1}\frac{H(\lambda_1, u)}{H(\lambda, u)}\leq 1,
 \end{split}
\end{equation}
where $\lambda_1 \backslash \mu_1$ denotes the set of boxes that are in $\lambda_1$ but not in $\mu_1$, that is, the set of boxes whose top right corner is in $\ag (i,1) \du \mu_1 + 1 \leq i \leq \lambda_1 \ad$.
\end{proof}

\begin{definition}
Let $p \geq 1$ and $u= (u_i)_{1 \leq i \leq p}, v= (v_i)_{1 \leq i \leq p}$ be two $p$-tuples of real numbers. We say that $v$ majorizes $u$ and write $u \preceq v$ if, for all $1 \leq \ell \leq p$,
    \begin{equation} \label{eq:plusgrand}
        \sum_{i=1}^\ell u_i \leq \sum_{i=1}^\ell v_i.
    \end{equation}
    and
    \begin{equation}
        \sum_{i=1}^p u_i=\sum_{i=1}^p v_i.
    \end{equation}
We write $u \prec v$ if, in addition, they are not equal.
\end{definition}
The following lemma is a classical inequality which we will use several times.
\begin{lemma}\label{lem:truccombinatoire}
     Let $p \geq 1$ and $u=(u_i)_{1 \leq i \leq p}, v=(v_i)_{1 \leq i \leq p}$ be two (weakly) decreasing $p$-tuples of positive real numbers such that $v \succeq u$. Then
    \begin{equation}
\label{eq:egalalafin ui vi}
    \prod_{i=1}^p v_i \leq \prod_{i=1}^p u_i.
\end{equation}
\end{lemma}
\begin{proof}
By Karamata's inequality, applied to the convex function $-\ln$, we have $\sum_{i=1}^p-\ln u_i \leq \sum_{i=1}^p-\ln v_i$. Multiplying by $-1$ and taking the exponential concludes the proof.
\end{proof}

Given two integers $0\leq k\leq n$, we write
\begin{equation}
    n^{\downarrow k} := n(n-1)\cdots (n-k+1) = \prod_{0\leq i \leq k-1} (n-i) = \frac{n!}{(n-k)!} = k!\binom{n}{k}.
\end{equation}

Let us now prove various bounds on $\lambda_1$-slicings.

\begin{proposition}\label{prop: propriétés produits équerres découpe}
        Let $\mu\subset\lambda$ be (non-empty) Young diagrams. Recall the notation $r=\lambda_{\geq 2}$.
\begin{enumerate}
    \item Recall from Definition \ref{def:definitionnotation} that $u_{\leq k}(\lambda) = \lambda_{\geq 2}\cap \lambda'_{\leq k}$ is the set of boxes of $\lambda$ that are in the first $k$ columns but are not on the first row, as represented in Figure \ref{fig:illustrationsdiagrams}. We have 
    \begin{equation}
        R_{\lambda_1}(\lambda, \mu)   \leq \frac{\pg\lambda_1+\left\lceil \frac{u_{\leq \mu_1}(\lambda)}{\mu_1}\right\rceil\pd^{\downarrow\mu_1}}{\lambda_1^{\downarrow \mu_1}}.
    \end{equation}
    
\item We have
\begin{equation}
    R_{\lambda_1}(\lambda, \lambda) \leq \binom{\lambda_1 + \left\lceil \frac{r}{\lambda_1}\right\rceil}{\lambda_1}.
\end{equation}
\item If $r\leq \lambda_1$, then
\begin{equation}
    R_{\lambda_1}(\lambda, \lambda) \leq 1+\frac{r}{\lambda_1-r+1}.
\end{equation}
\item Finally, we have
\begin{equation}
    R_{\lambda_1}(\lambda, \lambda) \leq e^{3\sqrt{r}}.
\end{equation}
\end{enumerate}
\end{proposition}

\begin{proof}
\begin{enumerate}
    \item First observe that, as a consequence of Lemma \ref{lem:truccombinatoire}, we have
    \begin{equation}\label{eq: borne R lambda mu avec les u_i par convexité sans partie entière}
        \prod_{i = 1}^{\mu_1} (\lambda_1-(i-1) + u_i) \leq \prod_{i = 1}^{\mu_1} \pg\lambda_1-(i-1) + \frac{u_{\leq \mu_1}(\lambda)}{\mu_1}\pd.
    \end{equation}

\noindent This implies that, setting $m := \left\lceil \frac{u_{\leq \mu_1}(\lambda)}{\mu_1}\right\rceil$, we have by Lemma \ref{lem:premiers resultats} (i):
\begin{equation}
    \begin{split}
        R_{\lambda_1}(\lambda, \mu) &= \frac{H(\lambda, \mu_1)}{H(\lambda_1, \mu_1)}
     = \prod_{i = 1}^{\mu_1} \frac{\lambda_1-(i-1)+ u_i }{\lambda_1-(i-1)}\\
     &\leq  \prod_{i = 1}^{\mu_1} \frac{\lambda_1-(i-1) + m}{\lambda_1-(i-1)} = \frac{(\lambda_1+m)^{\downarrow\mu_1}}{\lambda_1^{\downarrow \mu_1}}.
    \end{split}
\end{equation}
\noindent This concludes the proof of (a).

\item In the case $\mu = \lambda$, we have $\frac{u_{\leq \mu_1}(\lambda)}{\mu_1} = \frac{r}{\lambda_1}$. We therefore have directly, by (a):
\begin{equation}
    R_{\lambda_1}(\lambda, \lambda) \leq \frac{(\lambda_1+m)^{\downarrow\lambda_1}}{\lambda_1^{\downarrow \lambda_1}} = \binom{\lambda_1 + m}{\lambda_1}.
\end{equation}
\item Denote $\mu = [\lambda_1,n-\lambda_1]$. This diagram has the same number of boxes on (and over) the first row as $\lambda$, but has only two rows. Recall that $u_i(\lambda)$ denotes the number of boxes in the $i$-th column of $\lambda$ which are above the first row, i.e.\ $u_i(\lambda) = \lambda_i'-1$. By definition of $R_{\lambda_1}$, and since $H(\lambda_1, \lambda_1) = \lambda_1!=\mu_1! = H(\mu_1, \mu_1)$, we have 
\begin{equation}\label{eq: inter point c R lambda flattened}
    \frac{R_{\lambda_1}(\lambda, \lambda)}{R_{\mu_1}(\mu, \mu)} = \frac{H(\lambda, \lambda_1)/H(\lambda_1, \lambda_1)}{H(\mu, \mu_1)/H(\mu_1, \mu_1)} = \frac{H(\lambda, \lambda_1)}{H(\mu, \mu_1)} = \prod_{i = 1}^{\lambda_1} \frac{\lambda_1-(i-1)+ u_i(\lambda) }{\lambda_1-(i-1) + u_i(\mu)}.
\end{equation}
Moreover, the sequences $(\lambda_1-(i-1) + u_i(\lambda))_{1\leq i\leq \lambda_1}$ and $(\lambda_1-(i-1) + u_i(\mu))_{1\leq i\leq \lambda_1}$ are decreasing and satisfy $(\lambda_1-(i-1) + u_i(\lambda))_{1\leq i\leq \lambda_1} \succeq (\lambda_1-(i-1) + u_i(\mu))_{1\leq i\leq \lambda_1}$. We deduce from Lemma \ref{lem:truccombinatoire} that the right hand side of \eqref{eq: inter point c R lambda flattened} is upper bounded by 1, i.e.\ that $R_{\lambda_1}(\lambda, \lambda) \leq R_{\mu_1}(\mu, \mu)$. We conclude that
\begin{equation}
\begin{split}
     R_{\lambda_1}(\lambda, \lambda) \leq R_{\mu_1}(\mu, \mu) & = \frac{\pg\prod_{i=1}^r(\lambda_1 -(i-1) + 1)\pd(\lambda_1-r)!}{\lambda_1!} \\
     & = \frac{(\lambda_1+1)^{\downarrow r}}{\lambda_1^{\downarrow r}} = \frac{\lambda_1 + 1}{\lambda_1-r+1} = 1 + \frac{r}{\lambda_1-r+1}.
\end{split} 
\end{equation}
\item We split the proof according to the value of $r$.  
\begin{itemize}
    \item If $r<\lambda_1$, then starting from (c) we have 
\begin{equation}
     R_{\lambda_1}(\lambda, \lambda) \leq \frac{\lambda_1 + 1}{\lambda_1-r+1} \leq \frac{r+ 1}{r-r+1} = r+1 \leq e^{\sqrt{r}},
\end{equation}
and thus $R_{\lambda_1}(\lambda, \lambda) \leq e^{3\sqrt{r}}$.
    \item Assume now that $r\geq \lambda_1$. Let $n = |\lambda|$. Note that here again $\mu = \lambda$, so that $\mu_1 = \lambda_1$ and $u_{\leq \mu_1} = r$. Hence, $m = \left\lceil \frac{r}{\lambda_1}\right\rceil$, and since $r\geq \lambda_1$ we have $m\geq 1$. Hence, $m =  \left\lceil \frac{r}{\lambda_1}\right\rceil \leq \frac{r}{\lambda_1} + 1 \leq 2\frac{r}{\lambda_1}$, so that $\lambda_1 \leq \frac{2r}{m}$.
First assume that $\lambda_1 \leq m$. Using that $p^{\downarrow p} = p!\geq (p/e)^p$ for all $p\geq 0$, we therefore have, using $m\leq 2r/\lambda_1$ in the last inequality,
\begin{equation}
    \binom{\lambda_1 + m}{\lambda_1} = \frac{(\lambda_1+m)^{\downarrow\lambda_1}}{\lambda_1^{\downarrow \lambda_1}} \leq \frac{(2m)^{\lambda_1}}{(\lambda_1/e)^{\lambda_1}} = \pg \frac{2em}{\lambda_1}\pd^{\lambda_1}\leq \pg \frac{4er}{\lambda_1^2}\pd^{\lambda_1}.
\end{equation}
Now observe that we can rewrite $\pg \frac{4er}{\lambda_1^2}\pd^{\lambda_1} = \pg\pg \frac{2\sqrt{er}}{\lambda_1} \pd^{\lambda_1}\pd^2$. Furthermore, for $T>0$ the function $x\in \bbR_+^*\mapsto (T/x)^x$ is maximal at $x= T/e$, so we get, with $T = 2\sqrt{er}$,
\begin{equation}
    \binom{\lambda_1 + m}{\lambda_1} \leq \pg e^{T/e} \pd^2 = e^{\frac{4\sqrt{r}}{\sqrt{e}}} \leq e^{3\sqrt{r}}.
\end{equation}
The case $m\leq \lambda_1$ is proved the same way by symmetry, since $ \binom{\lambda_1 + m}{\lambda_1} =  \binom{\lambda_1 + m}{m}$. The result follows by (b). This concludes the proof of (d). \qedhere
\end{itemize}
\end{enumerate}
\end{proof}

We can now use these results to bound classical hook products and obtain a crucial inequality involving $d_\lambda, d_s$ and $d_c$.

\begin{proposition}\label{prop: borne abdelta et dimension centre}
Let $\lambda$ be a (non-empty) integer partition.
We have  \begin{equation}
    \frac{d_\lambda}{\binom{n}{c}d_s d_{c}} \geq e^{-6\sqrt{c}},
\end{equation}
where we recall that $n= |\lambda|$, $c=\lambda^{\geq 2}$, and $s = n-c = \lambda^1$.
\end{proposition}
\begin{proof}
 First, we have
        \begin{equation}
        \begin{split}
            H(\lambda,\lambda) & = sH(\lambda,\lambda_a^1)H(\lambda, \lambda_b^1)H(\lambda , c) \\ 
            & = s\cdot \lambda_a^1! \cdot \lambda_b^1! \frac{H(\lambda, \lambda_a^1)}{H(\lambda_a^1,\lambda_a^1)}\frac{H(\lambda, \lambda_b^1)}{H(\lambda_b^1,\lambda_b^1)}H(c,c) \\
            & = \frac{s!}{\binom{s-1}{\lambda_b^1}} \frac{H(\lambda, \lambda_a^1)}{H(\lambda_a^1,\lambda_a^1)}\frac{H(\lambda, \lambda_b^1)}{H(\lambda_b^1,\lambda_b^1)}H(c,c).
        \end{split}
        \end{equation}
We deduce from the hook length formula and Lemma \ref{lem:standard or elementary results} that
\begin{equation}\label{eq: inter in prop: borne abdelta et dimension centre}
\begin{split}
        \frac{d_\lambda}{\binom{n}{c}} = \frac{c!s!}{n!}\frac{n!}{H(\lambda,\lambda)} & = \binom{s-1}{\lambda_b^1} \frac{H(\lambda_a^1,\lambda_a^1)}{H(\lambda, \lambda_a^1)}\frac{H(\lambda_b^1,\lambda_b^1)}{H(\lambda, \lambda_b^1)} \frac{c!}{H(c,c)} \\
        & = \binom{s-1}{\lambda_b^1} \frac{H(\lambda_a^1,\lambda_a^1)}{H(\lambda, \lambda_a^1)}\frac{H(\lambda_b^1,\lambda_b^1)}{H(\lambda, \lambda_b^1)} d_c \\
        & = \frac{H(\lambda_a^1,\lambda_a^1)}{H(\lambda, \lambda_a^1)}\frac{H(\lambda_b^1,\lambda_b^1)}{H(\lambda, \lambda_b^1)} d_s d_c.
        \end{split}
\end{equation}
Moreover, by Proposition \ref{prop: propriétés produits équerres découpe} (d) applied twice respectively to the first row and the first column of the diagrams $\lambda'_{\geq 2}$ and $\lambda_{\geq 2}$, we get 
\begin{equation}
\frac{H(\lambda_a^1,\lambda_a^1)}{H(\lambda, \lambda_a^1)}\frac{H(\lambda_b^1,\lambda_b^1)}{H(\lambda, \lambda_b^1)} \geq e^{-3\sqrt{c}}e^{-3\sqrt{c}} = e^{-6\sqrt{c}}.
\end{equation}
Plugging this into \eqref{eq: inter in prop: borne abdelta et dimension centre} concludes the proof.
\end{proof}

\subsection{Virtual degrees and augmented dimensions}\label{s: virtual degrees and augmented dimensions}
If $P$ is a set partition of a diagram $\lambda\vdash n$, we can associate to it a notion of $P$-dimension via the analog of the hook length formula:
\begin{equation}
    d_\lambda^{*P} := \frac{n!}{H^{*P}(\lambda,\lambda)}.
\end{equation}
The virtual degree, defined in \eqref{eq: def virtual degree}, and $ab\delta$-dimension are closely related: we have
\begin{equation}
    D(\lambda) = \frac{d_\lambda^{*ab\delta}}{n}.
\end{equation}

We may therefore refer to $ab\delta$-sliced hook lengths and products as \textit{virtual} hook lengths and products.

Let us now define a last notion of dimension, which will prove to be very convenient in the proof of Theorem \ref{thm:ALS virtual degree asymptotic bound}. 
We define the \textit{augmented dimension} $d_\lambda^+$ of a Young diagram $\lambda$ by
\begin{equation}
    d^+_\lambda := \frac{n!}{H^{+}(\lambda, \lambda)},
\end{equation}
where
\begin{equation}
    H^+(\lambda, \lambda) = \pg \prod_i s_i \pd\pg\prod_i a_i! b_i!\pd.
\end{equation}
(Here $s_i$ denotes the hook started at $\delta_i$ and we recall that $a_i = \lambda_a^i$ and $b_i = \lambda_b^i$.)

The augmented dimension $d_\lambda^+$ is hybrid between $d_\lambda$ and $D(\lambda)$: the on-diagonal augmented hook lengths are the usual (non-sliced) hook lengths, while the off-diagonal augmented hook lengths are the virtual hook lengths.

\begin{figure}[!ht]
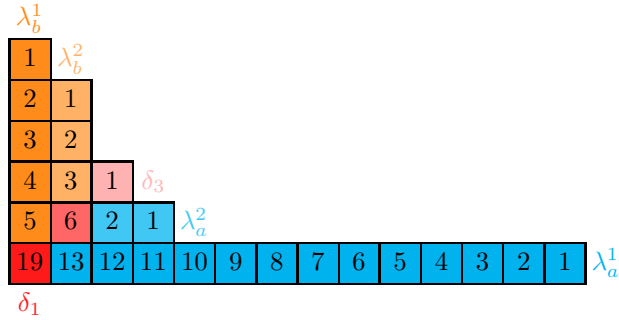

\begin{center}
    \begin{ytableau}
        \none &\none[\textcolor{orange!90}{\lambda_b^1}] \\
        \none &*(orange!90)1 &\none[\textcolor{orange!60}{\lambda_b^2}]\\
        \none &*(orange!90)2&*(orange!60)1\\
        \none &*(orange!90)3&*(orange!60)2\\
        \none &*(orange!90)4&*(orange!60)3&*(red!30)1 &\none[\textcolor{red!30}{\delta_3}]\\
        \none &*(orange!90)5&*(red!60)6&*(cyan!60)2&*(cyan!60)1 &\none[\textcolor{cyan!60}{\lambda_a^2}]\\
        \none &*(red!90)19&*(cyan!90)13&*(cyan!90)12&*(cyan!90)11&*(cyan!90)10&*(cyan!90)9&*(cyan!90)8&*(cyan!90)7&*(cyan!90)6&*(cyan!90)5&*(cyan!90)4&*(cyan!90)3&*(cyan!90)2&*(cyan!90)1 &\none[\textcolor{cyan!90}{\lambda_a^1}] \\
        \none &\none[\textcolor{red!90}{\delta_1}]
    \end{ytableau}
\end{center}
\caption{The augmented hook lengths for $\lambda = [14,4,3,2,2,1]$. Numbers in the boxes correspond to the augmented hook lengths.}
\label{fig:augmentedslicing}
\end{figure}

Note that $H^+(\lambda, \lambda)$ is \textbf{not} a sliced hook product, since we keep the full hook lengths of the boxes on the diagonal.
By definition we have
\begin{equation} \label{eq:réécriture quotient dimensions virtuelle et semivirtuelle}
    \frac{D(\lambda)}{d_\lambda^+} = \frac{\frac{(n-1)!}{\prod_i a_i! b_i!}}{\frac{n!}{\pg\prod_i a_i! b_i!\pd \pg \prod_i s_i \pd}} =\frac{\prod_i s_i}{n}.
\end{equation}
One advantage of using $d_\lambda^+$ is the following product identity (which we prove later on, in Lemma \ref{lem:approximation est une égalité}):
\begin{equation}
    d_\lambda^+ = \binom{n}{c}d_s^+d_c^+.
\end{equation}
This will be convenient when bounding $d_\lambda^+$ by induction on the size of the center $c$ of $\lambda$.

\section{Sharp bounds on virtual degrees}\label{s:proof of the main result}

The aim of this section is to use the results of the previous sections to prove Theorem \ref{thm:ALS virtual degree asymptotic bound}. 
In light of \eqref{eq:réécriture quotient dimensions virtuelle et semivirtuelle}, which rewrites as
\begin{equation} \label{eq:réécriture identité dimension virtuelle et semivirtuelle}
    D(\lambda)= \frac{\prod_i s_i}{n} d_\lambda^+,
\end{equation}
it is enough to prove the two following statements.

\begin{proposition}
\label{prop:estimée virtual degree augmented dimension}
There exists a constant $C_{diag}>0$ such that, for every $n\geq 2$ and every diagram $\lambda \vdash n$, we have
\begin{equation}
    \frac{\prod_i s_i}{n} \leq d_\lambda^{C_{diag}/\ln n}.
\end{equation}
\end{proposition}

\begin{proposition}\label{prop:récurrence sur les dimensions augmentées}
    There exists a constant $C_{aug}>0$ such that, for every $n\geq 2$ and every diagram $\lambda \vdash n$, we have
\begin{equation}
    d_\lambda^+ \leq d_\lambda^{1+C_{aug}/\ln n}.
    \end{equation}
\end{proposition}

Intuitively, Proposition \ref{prop:estimée virtual degree augmented dimension} states that the virtual degree $D(\lambda)$ and the augmented dimension $d_\lambda^+$ are close to each other, and Proposition \ref{prop:récurrence sur les dimensions augmentées} states that $d_\lambda^+$ is close to the usual notion of dimension, $d_\lambda$.

We prove Proposition \ref{prop:estimée virtual degree augmented dimension} in Section \ref{s: proof of prop:estimée virtual degree augmented dimension} and Proposition \ref{prop:récurrence sur les dimensions augmentées} in Section \ref{s: proof of prop:récurrence sur les dimensions augmentées}. Finally, in Section \ref{s:examples} we provide examples that show that Theorem \ref{thm:ALS virtual degree asymptotic bound} is sharp up to the value of the constant.

\subsection{Proof of Proposition \ref{prop:estimée virtual degree augmented dimension}}\label{s: proof of prop:estimée virtual degree augmented dimension}

Let us first give two general upper bounds on $\frac{\prod_i s_i}{n}$.

\begin{lemma}\label{lem: diag bounds product s_i in terms of c}
    Let $n\geq 1$ and $\lambda \vdash n$ such that $c = c(\lambda) \geq 1$.
Then we have
\begin{enumerate}
    \item \begin{equation}
    \frac{\prod_i s_i}{n} \leq (c/\delta(c))^{\delta(c)},
\end{equation}
    where $c=\lambda^{\geq 2}$ is the center of $\lambda$ and $\delta(c):=\delta(\lambda)-1$ is the diagonal length of $c$.
    \item Furthermore,
    \begin{equation}
         \frac{\prod_i s_i}{n} \leq e^{\sqrt{c}\ln c}.
    \end{equation}
\end{enumerate}

\end{lemma}
\begin{proof}
    \begin{enumerate}
        \item First, since $s_1\leq n$, we have
        \begin{equation}
            \frac{\prod_i s_i}{n} \leq \prod_{i\geq 2} s_i.
        \end{equation}
Moreover, by concavity of the logarithm, we have (using that $\delta(c) = \delta(\lambda) - 1$)
\begin{equation}
    \prod_{i\geq 2} s_i = \exp\pg\sum_{2\leq i \leq \delta(\lambda) }\ln(s_i)\pd \leq \exp\pg\sum_{2\leq i \leq \delta(\lambda) }\ln\pg\frac{c}{\delta(c)}\pd\pd = (c/\delta(c))^{\delta(c)},
\end{equation}
concluding the proof of (a).
\item We have $\delta(c) \leq \sqrt{c}$ (since the square of side length $\delta(c)$ is included in the diagram $c$), and $s_i \leq c$ for each $i\geq 2$. It therefore follows that
\begin{equation}
    \frac{\prod_i s_i}{n} \leq \prod_{i\geq 2} s_i \leq c^{\delta(c)} \leq c^{\sqrt{c}} = e^{\sqrt{c}\ln c}. \qedhere
\end{equation}
    \end{enumerate}
\end{proof}

In Proposition \ref{prop: borne abdelta et dimension centre}, we showed that $d_\lambda \geq \binom{s-1}{\lambda_b^1} e^{-6\sqrt{c}}\binom{n}{c}d_{c}$. Depending on the size of $c$, only some of the terms in the lower bound will be useful, here we give simple lower bounds on some of these terms.

\begin{lemma}\label{lem: simple bounds on parts of the lower bounds on the dimension}
    Let $n\geq 1$ and $\lambda \vdash n$ such that $c\geq1$. Then
\begin{enumerate}
    \item  
    \begin{equation}
        \binom{s-1}{\lambda_b^1} \geq s-1.
    \end{equation}
    \item \begin{equation}
        \binom{n}{c} \geq \max\pg (n/c)^c, (n/s)^s\pd.
    \end{equation}
\end{enumerate}
\end{lemma}
\begin{proof}
   \begin{enumerate}
       \item Since $c\geq 1$, we have $1\leq \lambda_b^1 \leq s-2$. We therefore have
       \begin{equation}
           \binom{s-1}{\lambda_b^1} \geq \binom{s-1}{1} = s-1.
       \end{equation}
       \item This bound is a classical bound on binomial coefficients. We have
       \begin{equation}
           \binom{n}{c} = \frac{n^{\downarrow c}}{c!} \geq (n/c)^c,
       \end{equation}
       and since $\binom{n}{c} = \binom{n}{s}$, we also have $ \binom{n}{c} \geq (n/s)^s$ by symmetry, concluding the proof. \qedhere
   \end{enumerate}
\end{proof}

We finally give a lower bound on $d_\lambda$ which depends only on the diagonal length of $\lambda$.

\begin{lemma}\label{lem: lower bound dimension via durfee square}
Let $n\geq 1$ and $\lambda \vdash n$. Then
    \begin{equation}
        d_\lambda \geq \pg\frac{\delta}{2e}\pd^{\delta(\delta-1)},
    \end{equation}
    where $\delta$ is the diagonal length of $\lambda$. In particular, if $n\geq 8e^6$ and $\delta \geq n^{1/3}$, we have \begin{equation}
    d_\lambda \geq e^{n^{2/3}/2}.
\end{equation}
\end{lemma}
\begin{proof}
    Since $d_\lambda$ is the number of standard tableaux of $\lambda$, and $\lambda$ contains the square $\mu:=\cg \delta^\delta\cd$, we have $d_\lambda \geq d_{\mu}$.
Moreover, since the hook lengths in $\mu$ are at most $2\delta$, by the hook length formula we have
\begin{equation}
    d_{\mu} = \frac{(\delta^2)!}{H(\mu, \mu)} \geq \frac{(\delta^2/e)^{\delta^2}}{(2\delta)^{\delta^2}} = \pg \frac{\delta}{2e}\pd^{\delta^2}.
\end{equation}
Since $d_\mu\geq 1$, we deduce that $d_\mu \geq \pg \frac{\delta}{2e}\pd^{\delta(\delta-1)}$.

Assume now that $n\geq 8e^6$ and $\delta \geq n^{1/3}$. Then we have $\delta \geq \pg 8e^6 \pd^{1/3} = 2e^2$ so $\delta/(2e)\geq e$, and we also have $\delta(\delta-1)\geq \delta^2/2$ since $\delta \geq 2$, so $(\delta/(2e))^{\delta(\delta-1)}\geq e^{\delta^2/2}\geq e^{n^{2/3}/2}$.
\end{proof}

We now prove Proposition \ref{prop:estimée virtual degree augmented dimension} for small values of $n$.

\begin{lemma}\label{lem:diag term finite case}
    For any $n_0\geq 2$, for every $2\leq n \leq n_0$ and every $\lambda \vdash n$, we have
    \begin{equation}
        \frac{1}{n}\prod_i s_i \leq d_\lambda^{C/\ln n},
    \end{equation}
where $C = \frac{n_0 (\ln n_0)^2}{\ln 2}$.
\end{lemma}
\begin{proof}
Let $n_0 \geq 2$, $2\leq n \leq n_0$ and $\lambda \vdash n$. If $\lambda$ is flat (horizontal or vertical), then $d_\lambda = 1$ and $\frac{1}{n}\prod_i s_i = 1$ so the result holds. Assume now that $\lambda$ is not flat. 
    Then $d_\lambda \geq 2$ so
\begin{equation}
    \frac{1}{n}\prod_i s_i \leq \prod_{1\leq i\leq \delta(\lambda)} s_i \leq n_0^{n_0} = 2^{\frac{1}{\ln 2}n_0 \ln n_0} \leq d_\lambda^{\frac{1}{\ln 2} n_0 \ln n_0} \leq d_\lambda^{\frac{1}{\ln 2} (n_0 \ln n_0) \frac{\ln n_0}{\ln n}} = d_\lambda^{C/\ln n}. \qedhere
\end{equation}
\end{proof}

We now have all the tools to prove Proposition \ref{prop:estimée virtual degree augmented dimension}.

\begin{proof}[Proof of Proposition \ref{prop:estimée virtual degree augmented dimension}]
Set $n_0 = \lf e^{55} \rf$ and $C_{diag} = e^{64}$. By Lemma \ref{lem:diag term finite case} the result holds for all $2\leq n \leq n_0$, since $\frac{n_0 (\ln n_0)^2}{\ln 2} \leq e^{55}(55^2)/\ln 2 \leq C_{diag}$. 

Assume now that $n> n_0$. We will prove that for such $n$ the bound $\frac{\prod_i s_i}{n} \leq d_\lambda^{C/\ln n}$ holds with $C=500$, which is enough since $C \leq C_{diag}$.

We split the proof into several cases, depending on how large $c$ is.
\begin{itemize}
    \item If $c \in \ag 0, 1\ad$ then $\frac{1}{n}\prod_i s_i \leq \frac{n}{n} = 1$ so the result holds.
    \item Assume that $2\leq c \leq n^{8/9}$. First, by Proposition \ref{prop: borne abdelta et dimension centre}, Lemma \ref{lem: simple bounds on parts of the lower bounds on the dimension} (b), and since $n/c \geq n^{1/9}$, we have 
    \begin{equation}
       d_\lambda \geq \binom{n}{c}e^{-6\sqrt{c}} \geq (n/c)^c e^{-6\sqrt{c}} \geq n^{c/9} e^{-6\sqrt{c}}.
    \end{equation}
Moreover, since $c\geq 2$ and $n\geq e^{55}$, we have $e^{-6\sqrt{c}} \geq e^{-5c} \geq n^{-c/11}$, and therefore   $n^{c/9} e^{-6\sqrt{c}} \geq n^{c(1/9 - 1/11)} \geq n^{c/500}$.

    On the other hand, we have $\frac{1}{n}\prod_i s_i \leq e^{\sqrt{c}\ln c}$ by Lemma \ref{lem: diag bounds product s_i in terms of c} (b). We deduce, using that $\ln c \leq \sqrt{c}$ for all $c \geq 1$ so that $\sqrt{c}\ln c \leq c$, that
    \begin{equation}
       \frac{1}{n}\prod_i s_i \leq e^{\sqrt{c}\ln c}\leq e^c
       = \pg n^{c/500}\pd^{C/\ln n} \leq d_{\lambda}^{C/\ln n}.
    \end{equation}
    \item Assume that $n^{8/9} \leq c \leq n-n^{2/3}$. Then 
    \begin{equation}
    \frac{n}{c} \geq \frac{n}{n-n^{2/3}} = \frac{1}{1-n^{-1/3}} \geq \frac{1}{e^{-n^{-1/3}}} = e^{n^{-1/3}}. 
    \end{equation}
    Moreover, since $n\geq 12^{18}$, we have $12\sqrt{c}\leq 12\sqrt{n} = n^{5/9} (12n^{-1/18}) \leq n^{5/9}$. We deduce, using again Proposition \ref{prop: borne abdelta et dimension centre} and Lemma \ref{lem: simple bounds on parts of the lower bounds on the dimension} (b) for the first inequality, that
    \begin{equation} \label{eq:contrainte 12 puissance 18}
        d_\lambda \geq (n/c)^ce^{-6\sqrt{c}} \geq \pg e^{n^{-1/3}}\pd^{n^{8/9}}e^{-6\sqrt{c}} = e^{n^{5/9}}e^{-6\sqrt{c}} \geq e^{n^{5/9}/2}.
    \end{equation}
    We conclude using Lemma \ref{lem: diag bounds product s_i in terms of c} (b) (and that $c\leq n$) that
    \begin{equation}
        \frac{1}{n}\prod_i s_i \leq e^{\sqrt{c}\ln c} \leq e^{\sqrt{n}\ln n} \leq e^{500\frac{n^{5/9}}{2\ln n}} \leq d_\lambda^{500/\ln n} = d_\lambda^{C/\ln n}.
    \end{equation}
    \item Assume now that $c \geq n-n^{2/3}$, i.e.\ $s \leq n^{2/3}$. Then we have $\delta(\lambda) \geq n^{1/3}$, and since $n\geq  e^{55}$ we also have $n\geq 8e^6$. Therefore by Lemma \ref{lem: lower bound dimension via durfee square} we have $d_\lambda \geq e^{n^{2/3}/2}$. We conclude, proceeding as in the previous case, that \begin{equation}
        \frac{1}{n}\prod_i s_i \leq e^{\sqrt{c}\ln c} \leq e^{\sqrt{n}\ln n} \leq e^{500\frac{n^{2/3}}{2\ln n}} \leq d_\lambda^{500/\ln n} \leq d_\lambda^{C/\ln n}.
    \end{equation}
\end{itemize}
This concludes the proof.
\end{proof}

\subsection{Proof of Proposition \ref{prop:récurrence sur les dimensions augmentées}}\label{s: proof of prop:récurrence sur les dimensions augmentées}

Our aim is to show that $\frac{d_\lambda^+}{d_\lambda}$ is small. We will proceed by induction, slicing the external hook $s=\lambda^1$.

First, we recall from Proposition \ref{prop: borne abdelta et dimension centre} that the dimension of a diagram $d_\lambda$ and the dimension of its center $d_c$ are related as follows: we have $d_\lambda \geq e^{-6\sqrt{c}}\binom{n}{c}d_sd_c$, that is, recalling from Lemma \ref{lem:standard or elementary results} that $d_s=\binom{s-1}{\lambda_a^1}$,
    \begin{equation} \label{eq:lem:bound ratio d lambda d c}
        \frac{d_c}{d_\lambda} \leq \frac{e^{6\sqrt{c}}}{\binom{s-1}{\lambda_a^1} \binom{n}{s}}.
    \end{equation}

\begin{remark}\label{rem: approx d lambda d s d c}
    We also always have $d_\lambda \leq \binom{n}{c}d_sd_c$. Indeed, by first choosing which numbers from 1 to $n$ we place in $c$, we obtain $|\ST(\lambda)|\leq \binom{n}{c}|\ST(s)||\ST(c)|$. This argument was written by Diaconis and Shahshahani in their proof of cutoff for random transpositions (with $r = \lambda_{\geq 2}$ in place of $c$), see \cite[Fact 1 in Chapter 3D, p.\ 39-40]{LivreDiaconis1988}. One can therefore see Proposition \ref{prop: borne abdelta et dimension centre} as a quantified version of the intuitive approximation $d_\lambda \approx \binom{n}{c}d_sd_c$, for large values of $s$ and $c$.
\end{remark}

Comparing the augmented dimensions $d_\lambda^+$ and $d_c^+$ is much simpler by design, since the hook products involved there are already sliced by definition. Also, by definition of $d_s^+$ we have $d_s^+ = \frac{s!}{s \lambda_a^1!\lambda_b^1!} = \frac{s-1}{\lambda_a^1!\lambda_b^1!} = \binom{s-1}{\lambda_a^1}$, and we recall from Lemma \ref{lem:standard or elementary results} that $d_s = \binom{s-1}{\lambda_a^1}$, so that
\begin{equation}\label{eq: elementary identity d s d s plus binom}
    d_s^+ = d_s = \binom{s-1}{\lambda_a^1}.
\end{equation}

In the next lemma we show that the aforementioned intuitive approximation for dimensions actually becomes an equality when it comes to augmented dimensions.
\begin{lemma}
\label{lem:approximation est une égalité}
     Let $n\geq 1$ and $\lambda\vdash n$. Then
     \begin{equation}
    d_\lambda^+ = \binom{n}{s}d_s^+d_c^+.
\end{equation}
In other words, we have
\begin{equation}
         \frac{d_\lambda^+}{d_c^+} = \binom{n}{s}{\binom{s-1}{\lambda_a^1}}.
     \end{equation}
\end{lemma}
\begin{proof}
    By definition we have
    \begin{equation}
        H^{+}(\lambda, \lambda) = s\lambda_a^1! \lambda_b^1!H^{+}(c, c). 
    \end{equation}
We therefore get
\begin{equation}
    \frac{d_\lambda^+}{d_c^+} = \frac{n!/H^{+}(\lambda, \lambda)}{c!/H^{+}(c, c)} = \frac{n^{\downarrow s}}{s\lambda_a^1!\lambda_b^1!} = \frac{n^{\downarrow s}}{s!} \frac{(s-1)!}{\lambda_a^1!\lambda_b^1!} = \binom{n}{s}{\binom{s-1}{\lambda_a^1}}. \qedhere
\end{equation}
\end{proof}
Combining \eqref{eq:lem:bound ratio d lambda d c} and Lemma \ref{lem:approximation est une égalité}, we obtain the following bound.

\begin{corollary}\label{cor: bound semivirtual degrees ratio lambda ratio c}
     Let $n\geq 1$ and $\lambda\vdash n$. Then
    \begin{equation}
        \frac{d_\lambda^+}{d_\lambda} \leq e^{6\sqrt{c}} \frac{d_c^+}{d_c}.
    \end{equation}
\end{corollary}

We now fix constants for the rest of this section. We set
\begin{equation}
    n_0 = \lf e^{51}\rf, \quad c_0 = 10^8, \quad \text{ and } \quad C = C_{aug} = e^{61}.
\end{equation}
The next lemma, whose proof we omit, collects some bounds which we will use just after.

\begin{lemma}
\label{lem:conditions pour les constantes}
   The following hold. 
    \begin{enumerate}[label=(\roman*)]
        \item For all $n \geq n_0$, we have 
        \begin{itemize}
            \item[($i_1$)] $\sqrt{n} \leq \frac{n}{8}$,
            \item[($i_2$)] $n \geq 4 \sqrt{n} \, \ln n$,
            \item[($i_3$)] $\frac{\ln 2}{\ln n} \leq \frac{1}{2}$,
            \item[($i_4$)] $6+\frac{6}{10^3}\ln n \leq \frac{\ln n}{8}$;
        \end{itemize}
        \item for all $c \geq c_0$, $6+6 \ln(c) \leq \frac{\sqrt{c}}{4} \ln(8/7)$;
        \item we have
        \begin{itemize}
            \item[($iii_1$)] $C\geq\frac{\ln(n_0!) \ln n_0}{\ln 2}$,
            \item[($iii_2$)] $C\geq 12 \sqrt{c_0} \ln(c_0!)$.
        \end{itemize}
    \end{enumerate}
\end{lemma}

We first check that our result holds for small values of $n$.

\begin{lemma}\label{lem:semivirt term finite case}
    Let $n$ such that $2\leq n \leq n_0$ and let $\lambda \vdash n$. Then
    \begin{equation}
        d_\lambda^+ \leq d_\lambda^{1+C/\ln n}.
    \end{equation}
\end{lemma}
\begin{proof}
If $\lambda$ is flat (horizontal or vertical), that is, $\lambda_a^1=0$ or $\lambda_b^1=0$, then $d_\lambda = d_\lambda^+ = 1$ and the result holds. If $\lambda$ is not flat, then $d_\lambda \geq 2$ so by Lemma \ref{lem:conditions pour les constantes} ($iii_1$) we get
\begin{equation}
    d_\lambda^+ \leq n! \leq n_0! = 2^{\frac{\ln(n_0!)}{\ln 2}} \leq d_\lambda^{\frac{\ln(n_0!)}{\ln 2}} \leq d_\lambda^{\frac{\ln(n_0!)}{\ln 2}\frac{\ln n_0}{\ln n}} \leq d_\lambda^{C/\ln n} \leq d_\lambda^{1 + C/\ln n}. \qedhere
\end{equation}
\end{proof}

We now consider small values of $c$, assuming that $n$ is large enough.

\begin{lemma}\label{lem:semivirt case n large c small}
Let $n\geq n_0$ and $\lambda \vdash n$ such that $c(\lambda) \leq c_0$. Then
     \begin{equation}
        d_\lambda^+ \leq d_\lambda^{1+C/\ln n}.
    \end{equation}
\end{lemma}

\begin{proof}
If $c = 0$, then $d_\lambda^+ = d_\lambda$ so the result holds. Assume now that $1\leq c\leq c_0$. Then $\lambda_a^1 \geq 1$ and $\lambda_b^1 \geq 1$ so $d_\lambda \geq d_s \geq \binom{s-1}{1} = s-1 \geq \sqrt{n}$ (since $n \geq n_0 \geq 3$ and $c\leq c_0$). Then by Corollary \ref{cor: bound semivirtual degrees ratio lambda ratio c} and Lemma \ref{lem:conditions pour les constantes} ($iii_2$),
\begin{equation}
d_\lambda^+\leq d_\lambda e^{6\sqrt{c}} \frac{d_c^+}{d_c} \leq d_\lambda e^{6\sqrt{c}} d_c^+
\leq d_\lambda e^{6\sqrt{c_0}} c_0!  \leq   d_\lambda \sqrt{n}^{\frac{12\sqrt{c_0} \ln(c_0!)}{\ln n}} \leq  d_\lambda^{1+ C/\ln n}. \qedhere
\end{equation}
\end{proof}

The next result concerns large values of $n$ and specific values of $s$.

\begin{lemma}\label{lem: some inequality (es/n)^{s/2}}
    Let $n\geq n_0$ and $s$ such that $\sqrt{n} \leq s \leq n/8$. Then
\begin{equation}
      \pg\frac{es}{n}\pd^{s/2} \leq e^{-\frac{\sqrt{n}\ln n}{8}}.
\end{equation}  
\end{lemma}

\begin{proof}
    First note that since $n \geq n_0$, we have $\sqrt{n}\leq n/8$ by Lemma \ref{lem:conditions pour les constantes} ($i_1$). We define
    \begin{equation}
        \begin{array}{lrcl}
f : & \cg \sqrt{n}, n/8\cd & \longrightarrow & \bbR \\
    & x & \longmapsto & \pg\frac{ex}{n}\pd^{x/2} \end{array}
\end{equation}
and observe that $g := \ln \circ f$ is convex (since $g''(x) = \frac{1}{2x} >0$ for $x\in \cg \sqrt{n}, n/8\cd$). Hence, $g$ reaches its maximum value when $x$ is either minimal or maximal and therefore the same holds also for $f$. Moreover, $e^2\leq 8$ and $n\geq 4\sqrt{n}\ln n$ (by Lemma \ref{lem:conditions pour les constantes} ($i_2$)) so we have
\begin{equation}
    \frac{f(n/8)}{f(\sqrt{n})} = \frac{(e/8)^{n/16}}{(e/\sqrt{n})^{{\sqrt{n}/2}}} \leq \frac{\sqrt{n}^{\sqrt{n}/2}}{e^{n/16}} = \frac{n^{\sqrt{n}/4}}{e^{n/16}} \leq 1.
\end{equation}
We deduce that $f$ is maximal at $\sqrt{n}$, and conclude that
\begin{equation}
    \max_{\sqrt{n}\leq s \leq n/8} f(s) \leq f(\sqrt{n}) = (e/\sqrt{n})^{\sqrt{n}/2}\leq (n^{-1/4})^{\sqrt{n}/2} = e^{-\frac{\sqrt{n}\ln n}{8}}. \qedhere
\end{equation}
\end{proof}
We now have all the tools to prove Proposition \ref{prop:récurrence sur les dimensions augmentées}.
\begin{proof}[Proof of Proposition \ref{prop:récurrence sur les dimensions augmentées}]

We proceed by (strong) induction on $n$. Our recurrence hypothesis is 
    \begin{equation}
        H(n): \text{“}d_\lambda^+ \leq d_{\lambda}^{1+C/\ln n} \text{ for every } \lambda \vdash n\text{”}.
    \end{equation}
    
\medskip

By Lemma \ref{lem:semivirt term finite case}, $H(n)$ holds for all $2\leq n \leq n_0$, which initializes the induction.

\medskip

Let now $n > n_0$ and assume that $H(m)$ holds for every $m\leq n-1$. Let $\lambda \vdash n$.

The case $0 \leq c \leq c_0$ follows from Lemma \ref{lem:semivirt case n large c small}. Since, in addition, $s\geq \sqrt{n}$ for every diagram, we can assume from now on that $c_0 \leq c\leq n-\sqrt{n}$.

\medskip

By Corollary \ref{cor: bound semivirtual degrees ratio lambda ratio c}, 
    \begin{equation}
        \frac{d_\lambda^+}{d_\lambda^{1+\frac{C}{\ln n}}} = \frac{d_\lambda^+}{d_\lambda} d_\lambda^{-\frac{C}{\ln n}} \leq e^{6\sqrt{c}} \frac{d_c^+}{d_c} d_\lambda^{-\frac{C}{\ln n}} = e^{6\sqrt{c}} \frac{d_c^+}{d_c^{1+\frac{C}{\ln c}}} \pg\frac{d_c}{d_\lambda} \pd^{C/\ln c} d_\lambda^{\frac{C}{\ln c}-\frac{C}{\ln n}}.
    \end{equation} 
Applying the induction hypothesis to $c$, we deduce that
\begin{equation}
    \frac{d_\lambda^+}{d_\lambda^{1+\frac{C}{\ln n}}} \leq e^{6\sqrt{c}} \pg\frac{d_c}{d_\lambda} \pd^{C/\ln c} d_\lambda^{\frac{C}{\ln c}-\frac{C}{\ln n}}.
\end{equation}
By \eqref{eq:lem:bound ratio d lambda d c} and using that $d_\lambda \leq \binom{n}{c}d_c d_s$ (as explained in Remark \ref{rem: approx d lambda d s d c}),  we deduce that 
\begin{equation}
\begin{split}
     \frac{d_\lambda^+}{d_\lambda^{1+\frac{C}{\ln n}}} 
    & \leq e^{6\sqrt{c}} \pg \frac{e^{6\sqrt{c}}}{d_s \binom{n}{s}}\pd^{\frac{C}{\ln c}} \pg \binom{n}{c}d_c d_s \pd^{\frac{C}{\ln c} - \frac{C}{\ln n}}  \\
    & = e^{6\sqrt{c}\pg  1+\frac{C}{\ln c}\pd} \pg d_s \binom{n}{c} \pd^{-C/\ln n} d_c^{\frac{C}{\ln c} - \frac{C}{\ln n}}  \\
   & = \pg \frac{e^{6\sqrt{c}(1+\frac{\ln c}{C})}}{\pg d_s \binom{n}{c}\pd^{\frac{\ln c}{\ln n}}} d_c^{1-\frac{\ln c}{\ln n}} \pd^{C/\ln c} \\
   & \leq \pg e^{6\sqrt{c}(1+\frac{\ln c}{C})} \binom{n}{c}^{-\frac{\ln c}{\ln n}} d_c^{1-\frac{\ln c}{\ln n}} \pd^{C/\ln c}
\end{split}
\end{equation}
To show our result, it is therefore enough to prove that 

\begin{equation}
\label{eq:ce qu'on veut prouver finalement}
    e^{6\sqrt{c}(1+\frac{\ln c}{C})} \binom{n}{c}^{-\frac{\ln c}{\ln n}} d_c^{1-\frac{\ln c}{\ln n}} \leq 1.
\end{equation}
We split the proof of \eqref{eq:ce qu'on veut prouver finalement} depending on the value of $c$.

\fbox{$n-\sqrt{n} \geq c \geq 7n/8$}

\noindent In this case, we have $\sqrt{n} \leq s\leq n/8$. Then 
    \begin{equation}
        1-\frac{\ln c}{\ln n} = \frac{\ln n - \ln c}{\ln n} = \frac{\ln(\frac{c+s}{c})}{\ln n} = \frac{\ln(1+\frac{s}{c})}{\ln n} \leq \frac{s}{c \ln n}.
    \end{equation}
    In addition, by Lemma \ref{lem:standard or elementary results}, $d_c \leq \sqrt{c!} \leq c^{c/2} \leq n^{c/2}$, so that
    \begin{equation}
    \label{eq:equabourrine}
        d_c^{1-\frac{\ln c}{\ln n}} \leq e^{s/2}.
    \end{equation}
On the other hand, since $1-s/n \geq 1/2$ by assumption, we have 
\begin{equation}
    \frac{\ln c}{\ln n} = 1+\frac{\ln(c/n)}{\ln n} = 1+ \frac{\ln (1-s/n)}{\ln n} \geq 1-\frac{\ln 2}{\ln n},
\end{equation}
and in particular $\frac{\ln c}{\ln n} \geq 1/2$ by Lemma \ref{lem:conditions pour les constantes} ($i_3$).

Recall also from Lemma \ref{lem: simple bounds on parts of the lower bounds on the dimension} (b) that $\binom{n}{c}=\binom{n}{s} \geq (n/s)^s$. Therefore, by \eqref{eq:equabourrine} and Lemma \ref{lem: some inequality (es/n)^{s/2}},
\begin{equation}
     \binom{n}{c}^{-\frac{\ln c}{\ln n}} d_c^{1-\frac{\ln c}{\ln n}} \leq  \pg\frac{s}{n}\pd^{s/2} e^{s/2} = \pg\frac{es}{n}\pd^{s/2} \leq e^{-\frac{\sqrt{n}\ln n}{8}}.
\end{equation}
We conclude, using that $\ln c \leq \ln n$ and $\sqrt{c}\leq \sqrt{n}$, and by Lemma \ref{lem:conditions pour les constantes} ($i_4$) (which we can apply since $C \geq 10^3$), that 
\begin{equation}
    e^{6\sqrt{c}(1+\frac{\ln c}{C})} \binom{n}{c}^{-\frac{\ln c}{\ln n}} d_c^{1-\frac{\ln c}{\ln n}} \leq e^{\sqrt{n}\pg 6 + \frac{6\ln n}{C} -\frac{\ln n}{8}\pd} \leq 1.
\end{equation}

\bigskip

\fbox{$c_0 \leq c \leq 7n/8$}
    Since $c\geq c_0 \geq 6$, we have by Lemma \ref{lem:standard or elementary results} that $d_c \leq \sqrt{c!}\leq (c/2)^{c/2}$. 
    Therefore,
    \begin{equation}
    \begin{split}
        d_c^{1-\frac{\ln c}{\ln n}} = d_c^{\frac{\ln (n/c)}{\ln n}} \leq \pg (c/2)^{c/2}\pd^{\frac{\ln(n/c)}{\ln n}} & = \pg (c/n)^{c/2}\pd^{\frac{\ln(n/c)}{\ln n}}\pg n^{c/2}\pd^{\frac{\ln(n/c)}{\ln n}} \pg (1/2)^{c/2}\pd^{\frac{\ln(n/c)}{\ln n}}\\
        &= \pg (c/n)^{c/2}\pd^{\frac{\ln(n/c)}{\ln n}} (n/c)^{c/2} \, \cdot \, 2^{-\frac{c}{2}\frac{\ln(n/c)}{\ln n}} \\
        & = \pg\frac{n}{c} \pd^{\frac{c}{2}\pg 1- \frac{\ln(n/c)}{\ln n}\pd}2^{-\frac{c}{2}\frac{\ln(n/c)}{\ln n}}.
    \end{split}
    \end{equation}
On the other hand we still have
$\binom{n}{c} \geq (n/c)^c$
and
$\frac{\ln c}{\ln n} = 1- \frac{\ln(n/c)}{\ln n}$, so
\begin{equation}
    \binom{n}{c}^{-\frac{\ln c}{\ln n}} \leq \pg\frac{n}{c} \pd^{-c\pg 1- \frac{\ln(n/c)}{\ln n}\pd},
\end{equation}
and overall we get
\begin{equation}
\begin{split}
    \binom{n}{c}^{-\frac{\ln c}{\ln n}} d_c^{1-\frac{\ln c}{\ln n}} & \leq  \pg\frac{n}{c} \pd^{-c\pg 1- \frac{\ln(n/c)}{\ln n}\pd} \pg\frac{n}{c} \pd^{\frac{c}{2}\pg 1- \frac{\ln(n/c)}{\ln n}\pd}2^{-\frac{c}{2}\frac{\ln(n/c)}{\ln n}}\\
    &= \pg\frac{c}{n} \pd^{\frac{c}{2}\pg 1- \frac{\ln(n/c)}{\ln n}\pd} 2^{-\frac{c}{2}\frac{\ln(n/c)}{\ln n}}.
\end{split}
\end{equation}
\noindent Observe that both factors are clearly $\leq 1$.
Now, if $c\geq \sqrt{n}$ we have $1-\frac{\ln(n/c)}{\ln n}\geq 1/2$ so 
\begin{equation}
    \pg\frac{c}{n} \pd^{\frac{c}{2}\pg 1- \frac{\ln(n/c)}{\ln n}\pd} \leq \pg\frac{c}{n} \pd^{\frac{c}{4}} \leq (7/8)^{c/4}.
\end{equation}
On the other hand, if $c\leq \sqrt{n}$, we have $\frac{\ln(n/c)}{\ln n} \geq 1/2$ so
\begin{equation}
  2^{-\frac{c}{2}\frac{\ln(n/c)}{\ln n}}\leq 2^{-\frac{c}{4}} \leq (7/8)^{c/4},
\end{equation}
so we always have 
\begin{equation}
    \binom{n}{c}^{-\frac{\ln c}{\ln n}} d_c^{1-\frac{\ln c}{\ln n}} \leq (7/8)^{c/4}.
\end{equation}

\noindent Hence, we get (using Lemma \ref{lem:conditions pour les constantes} (ii) and the fact that $C \geq 1$) that

\begin{equation}
    e^{6\sqrt{c}(1+\frac{\ln c}{C})} \binom{n}{c}^{-\frac{\ln c}{\ln n}} d_c^{1-\frac{\ln c}{\ln n}} \leq e^{6\sqrt{c}\pg 1 + \frac{\ln c}{C} \pd - \frac{c}{4} \ln(8/7)} \leq 1.
\end{equation}
This concludes the proof of \eqref{eq:ce qu'on veut prouver finalement}, and hence of the proposition.
\end{proof}

\begin{proof}[Proof of Theorem \ref{thm:ALS virtual degree asymptotic bound}]
    By \eqref{eq:réécriture quotient dimensions virtuelle et semivirtuelle},  Proposition \ref{prop:estimée virtual degree augmented dimension}, and Proposition \ref{prop:récurrence sur les dimensions augmentées}, we have 
\begin{equation}
    \frac{D(\lambda)}{d_\lambda} = \frac{D(\lambda)}{d_\lambda^+} \frac{d_\lambda^+}{d_\lambda} = \frac{\prod_i s_i}{n} \frac{d_\lambda^+}{d_\lambda} \leq d_\lambda^{C_{diag}/\ln n} d_\lambda^{C_{aug}/\ln n},
\end{equation}
i.e.\ $D(\lambda) \leq d_\lambda^{1+ (C_{diag} + C_{aug})/\ln n}$. Moreover, since we used the constants $C_{diag} = e^{64}$ and $C_{aug} = e^{61}$ in the previous proofs, the bound $D(\lambda) \leq d_\lambda^{1+ C/\ln n}$ holds with e.g.\ $C = e^{65}$.
\end{proof}

\subsection{Sharpness of the bound}\label{s:examples}
We consider here two examples, namely square-shaped diagrams and almost-flat diagrams, which show sharpness of Theorem \ref{thm:ALS virtual degree asymptotic bound}. In other words, Theorem \ref{thm:ALS virtual degree asymptotic bound} provides the best possible asymptotic bound on virtual degrees that is uniform over all irreducible characters. 
For square-shaped diagrams $\lambda = [p^p]$, we show that $D(\lambda) = d_{\lambda}^{1+\frac{4\ln 2}{\ln n} (1+o(1))}$ where $n=p^2$, and for $\lambda = [n-2,2]$, we show that $D(\lambda) = d_{\lambda}^{1+\frac{\ln 2}{2\ln n}(1+o(1))}$.
We did not optimize the value of the constant $C$ in Theorem \ref{thm:ALS virtual degree asymptotic bound}, but we conjecture that as $n\to \infty$ we have $D(\lambda) \leq d_\lambda^{1+\frac{4\ln 2 + o(1)}{\ln n}}$ uniformly for any $\lambda \vdash n$, and that the constant $4 \ln 2$ is optimal.

\subsubsection{Square-shaped diagrams}

For all $n \geq 1$ such that $n$ is a perfect square (i.e.\ $p := \sqrt{n} \in \mathbb{N}$), let $\lambda_n = [p^p] = [p, \ldots, p]$ be the square-shaped Young diagram (see Figure \ref{fig:squareshape}).

\begin{figure}[!ht]
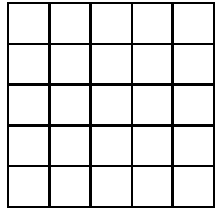

    \centering
\begin{ytableau}
    \none   & & & & & \\
    \none   & & & & & \\
    \none   & & & & & \\
    \none   & & & & & \\
    \none   & & & & & \\
\end{ytableau}
    \caption{The Young diagram $\lambda_{25}$}
    \label{fig:squareshape}
\end{figure}

Let us compute $d_{\lambda_n}$. We have

\begin{equation}
    H(\lambda_n,\lambda_n)=\prod_{i,j=1}^{p} (i+j-1)
    =\prod_{i=1}^p \frac{(i+p-1)!}{(i-1)!}
    = \frac{\prod_{j=0}^{2p-1} j!}{(\prod_{i=0}^{p-1} i!)^2}
    =\frac{G(2p+1)}{G(p+1)^2},
\end{equation}
where $G$ denotes Barnes' $G$-function.

In particular, the asymptotic expansion of its logarithm is known (see e.g.\ \cite[Lemma $5.1$]{Adamchik2014}): as $z \rightarrow \infty$,
\begin{equation}
\label{eq:asymptotic expansion barnes}
    \ln G(1+z)=\frac{z^2}{2} \ln z - \frac{3z^2}{4} + O(z).
\end{equation}
This provides

\begin{equation}
    \begin{split}
        \ln H(\lambda_n,\lambda_n) &= \ln G(2p+1) - 2 \ln G(p+1)\\
    &= \frac{(2p)^2}{2} \ln(2p) - \frac{3(2p)^2}{4} -2\frac{p^2}{2} \ln p + 2\frac{3p^2}{4} + O(p)\\
    &= 2p^2\ln p + 2p^2 \ln 2 - 3p^2 - p^2 \ln p +\frac{3}{2}p^2 + O(p)\\
    &= p^2 \ln p + (2\ln 2-\frac{3}{2}) p^2 + O(p).
    \end{split}
\end{equation}
Thus, we get using Stirling's approximation:

\begin{equation}
    \begin{split}
         \ln d_{\lambda_n} &= \ln n! - \ln H(\lambda_n,\lambda_n)\\
    &= n \ln n - n + O(\ln n) - p^2 \ln p - (2\ln 2-\frac{3}{2}) p^2 + O(p)\\
    &= 2 p^2 \ln p - p^2 - p^2 \ln p - (2\ln 2-\frac{3}{2}) p^2 + O(p)\\
    &= p^2 \ln p + (\frac{1}{2} - 2\ln 2)p^2 + O(p).
    \end{split}
\end{equation}
On the other hand, we have

\begin{equation}
    \begin{split}
         D(\lambda_n) = \frac{(n-1)!}{\prod a_i! \prod b_i!} = \frac{(n-1)!}{(\prod_{i=1}^{p-1} i!)^2} = \frac{(n-1)!}{G(1+p)^2}.
    \end{split}
\end{equation}
Using again Stirling's formula along with \eqref{eq:asymptotic expansion barnes}, we get that
\begin{equation}
    \begin{split}
         \ln D(\lambda_n) &= \ln((n-1)!)-2\ln G(p+1)\\
    &= n \ln n - n + O(\ln n) -2\frac{p^2}{2} \ln p + 2\frac{3p^2}{4} + O(p)\\
    &= 2p^2 \ln p - p^2 - p^2 \ln p + \frac{3}{2} p^2 + O(p)\\
    &= p^2 \ln p + \frac{1}{2} p^2 + O(p).
    \end{split}
\end{equation}
We deduce that 

\begin{equation}
    \begin{split}
        \ln D(\lambda_n)-\ln d_{\lambda_n} = \pg 2\ln 2\pd  p^2+O(p)
=\frac{4\ln 2}{\ln n} \ln d_{\lambda_n} (1+o(1)).
    \end{split}
\end{equation}
Hence, as $n \rightarrow \infty$:
\begin{equation}
    D(\lambda_n) = d_{\lambda_n}^{1+\frac{4\ln 2}{\ln n} (1+o(1))}.
\end{equation}

\subsubsection{Almost-flat diagrams}

The second case that we consider is the case of a diagram $\mu_n$ of size $n$ with two rows, one of length $n-2$ and the second of length $2$ (see Figure \ref{fig:almostflat}).

\begin{figure}[!ht]
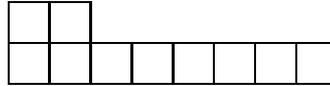

    \centering
\begin{ytableau}
    \none   & & \\
    \none   & & & & & & & & \\
\end{ytableau}
    \caption{The Young diagram $\mu_{10}$}
    \label{fig:almostflat}
\end{figure}
In this case, we can also compute

\begin{equation}
    H(\mu_n,\mu_n) = (n-4)!*(n-2)*(n-1)*2 = 2 \frac{n!}{n(n-3)},
\end{equation}
so that
\begin{equation}
    d_{\mu_n}=\frac{n(n-3)}{2}.
\end{equation}
On the other hand, we have 
\begin{equation}
    D(\mu_n) = \frac{(n-1)!}{(n-3)!} = (n-1)(n-2).
\end{equation}
Hence, we have
\begin{equation}
        \ln\left( \frac{D(\mu_n)}{d_{\mu_n}} \right) = \ln \left( \frac{2(n-1)(n-2)}{n(n-3)} \right)= \ln 2 + o(1) = \frac{\ln 2}{2\ln n} \ln d_{\mu_n} (1+o(1)). 
\end{equation}
We conclude that
\begin{equation}
    D(\mu_n) = d_{\mu_n}^{1+\frac{\ln 2}{2\ln n}(1+o(1))}.
\end{equation}

\section{Character bounds in terms of the number of cycles}\label{s: Character bounds in terms of the number of cycles} 

The goal of this section is to prove Theorem \ref{thm: borne caractères nombre de cycles}, which provides bounds on characters in terms of the number of cycles of the permutation $\sigma$.

To prove bounds on characters depending on the number of cycles of permutations, Larsen and Shalev \cite{LarsenShalev2008} introduced the notion of cycle growth sequence. The cycle growth sequence $(b_k)_{k \geq 1}$ is defined by 
$n^{b_k} = \max(F_k(\sigma), 1)$, where $F_k = f_1 + \ldots + f_k$ denotes the number of cycles of $\sigma$ of length at most $k$, and Larsen and Shalev set 
\begin{equation}
    B(\sigma) = \sum_{k\geq 1} \frac{b_k}{k(k+1)}.
\end{equation}
They showed that asymptotically $E(\sigma) \leq B(\sigma) + o(1)$, and that if $\cyc(\sigma) \leq n^\alpha$ then $B(\sigma) \leq \alpha$. Plugging these bounds into Theorem \ref{thm:LS form E(sigma) + o(1)}, they deduced in \cite[Theorem 1.4]{LarsenShalev2008} that if $\cyc(\sigma) \leq n^\alpha$ then 
\begin{equation}
    |\ch^\lambda(\sigma)| \leq d_\lambda^{\alpha + o(1)},
\end{equation}
improving on their bound from \cite{LarsenShalev2009wordmapswaring}.

\begin{remark}
    The $o(1)$ term above has two components: one coming from the bound $D(\lambda)\leq d_\lambda^{1+o(1)}$ (Theorem \ref{thm:LS form E(sigma) + o(1)}) and one from the bound $E(\sigma) \leq B(\sigma) + o(1)$. The bounds at the end of the proof of Theorem 1.1 in \cite{LarsenShalev2008} actually show that $E(\sigma) \leq B(\sigma) + O\pg \frac{\ln \ln n}{\ln n} \pd$. Combining this observation with Theorem \ref{thm:ALS improved character bound} would give the bound $|\ch^\lambda(\sigma)|\leq d_\lambda^{\alpha + O\pg\frac{\ln \ln n}{\ln n}\pd}$ if $\cyc(\sigma) \leq n^\alpha$, which is another formulation of the character bound from Lifshitz and Marmor \cite[Theorem 1.7]{LifshitzMarmor2024charactershypercontractivity}. The first aim of this section is to prove Theorem \ref{thm: borne caractères nombre de cycles}, which provides a sharper bound.
\end{remark}

 An equivalent formulation (and a more direct proof) of the bound $E(\sigma) \leq B(\sigma)+o(1)$ is given in \cite[Lemma 2.8]{KellerLifshitzScheinfeld2024}: fix $\varepsilon>0$ and $\alpha>0$, then as $n\to \infty$, for all $\sigma \in \kS_n$ satisfying $\cyc(\sigma) \leq n^\alpha$ we have $E(\sigma) \leq \alpha + \varepsilon$.

In Proposition \ref{prop: borne E sigma avec max cyc 2} (and Corollary \ref{cor: borne E sigma n puissance a}), we show the following stronger result: if $2\leq \cyc(\sigma) \leq n^{\alpha}$ then $E(\sigma) \leq \alpha$. The proof relies on a convexity argument and a telescoping product. Combining it with Theorem \ref{thm:ALS improved character bound}, it allows us to extend the character bounds of Lifshitz and Marmor \cite[Theorem 1.7]{LifshitzMarmor2024charactershypercontractivity} -- which hold when $\cyc(\sigma) \leq \frac{n}{(\ln n)^{C}}$ for some constant $C>0$ -- to permutations with $\cyc(\sigma) =o(n)$. This is the content of Theorem \ref{thm: borne caractères nombre de cycles}.

\subsection{A simple bound on orbit growth exponents}

Let $n\geq 2$ and $\sigma \in \kS_n$. We begin by proving the following surprisingly simple bound on $E(\sigma)$ that depends only on the number of cycles $\cyc(\sigma)$.

\begin{proposition}\label{prop: borne E sigma avec max cyc 2}
    Let $n\geq 2$ and $\sigma \in \mathfrak{S}_n$.  Then
\begin{equation}
    E(\sigma) \leq \frac{\ln \left(\max(\cyc(\sigma),2)\right)}{\ln n}.
\end{equation}
\end{proposition}

We first present a simple consequence of Proposition \ref{prop: borne E sigma avec max cyc 2}.

\begin{corollary}\label{cor: borne E sigma n puissance a}
    Let $n\geq 2$ and $\sigma \in \mathfrak{S}_n$.
    Assume that
    $\cyc(\sigma)\geq 2$, and let  $0< \alpha\leq 1$ such that $\cyc(\sigma) = n^\alpha$. Then 
     \begin{equation}
        E(\sigma) \leq \alpha.
    \end{equation}
\end{corollary}
\begin{proof}
    Since $\cyc(\sigma)\geq 2$, from Proposition \ref{prop: borne E sigma avec max cyc 2} we have
        $E(\sigma) \leq \frac{\ln \cyc(\sigma)}{\ln n} = \frac{\ln n^\alpha}{\ln n} = \alpha$.
\end{proof}
We now build the tools that will be needed to prove Proposition \ref{prop: borne E sigma avec max cyc 2}.

Let $\imin := \imin(\sigma) = \min\ag i\geq 1 \du f_i \geq 1 \ad$ be the length of the smallest cycle in $\sigma$.
We can rewrite the definition of the orbit growth sequence in terms of $\imin$.

\begin{lemma}\label{lem: orbit growth sequence logarithmic form}
    Let $n\geq 2$ and $\sigma \in \mathfrak{S}_n$. Then 
    \begin{equation}
    e_i =  \begin{cases}
         0 & \text{ for } i<\imin\\
         \frac{\ln \pg\imin f_{\imin}\pd}{\ln n} & \text{ for } i=\imin\\
        \frac{\ln (\Sigma_i/\Sigma_{i-1})}{\ln n} = \frac{\ln(1 + if_i/\Sigma_{i-1})}{\ln n} & \text{ for } \imin<i\leq n. 
        \end{cases}
    \end{equation}
    
\end{lemma}
\begin{proof}
    By definition of $\imin$  we have $f_i = 0$ for $i<\imin$, so for $i<\imin$ we have $\Sigma_i = 0$ and therefore $e_i=0$. It follows that $n^{e_{\imin}} = n^{e_1 + \ldots+e_{\imin}}= \Sigma_{\imin} = \imin f_{\imin}$, and we deduce the formula for $e_{\imin}$ taking logarithms.
    
    Let now $i>\imin$. Then $\Sigma_{i-1}\geq \Sigma_{\imin} \geq 1$ and therefore 
    \begin{equation}
        \Sigma_i = n^{e_1 + \ldots + e_i} = n^{e_1 + \ldots + e_{i-1}}n^{e_i} = \Sigma_{i-1}n^{e_i}.
    \end{equation}
We then obtain the first equality for $e_i$ by dividing both sides by $\Sigma_{i-1}$ and taking the logarithm, and the second equality by rewriting $\Sigma_i = \Sigma_{i-1} + if_i$.
\end{proof}

We can now prove a general simple bound on $E(\sigma)$. We recall that $F_i = f_1 + \ldots + f_i$ is the number of cycles of $\sigma$ of length at most $i$.

\begin{proposition}\label{prop: borne E sigma avec imin}
    Let $n\geq 2$ and $\sigma \in \mathfrak{S}_n$. Then 
\begin{equation}
    E(\sigma) \leq \frac{1}{\imin} \frac{\ln \pg \imin \cyc(\sigma)\pd}{\ln n}.
\end{equation}

\end{proposition}
\begin{proof}
Let $i> \imin$. By Lemma \ref{lem: orbit growth sequence logarithmic form}, and using that $(1+x)^{a} \leq 1+ax$ for $x \geq 0$ and $0 \leq a\leq 1$, we have  
    \begin{equation}
        \frac{\imin}{i}e_i  = \frac{\ln\pg \pg 1 + \frac{if_i}{\Sigma_{i-1}}\pd^{\imin/i} \pd}{\ln n} \leq  \frac{\ln\pg  1 + \frac{\imin f_i}{\Sigma_{i-1}}\pd}{\ln n}.
    \end{equation}
Moreover since $f_j = 0$ for $1\leq j <\imin$, we have $\Sigma_{i-1} \geq \imin F_{i-1}$, and therefore
\begin{equation}
    \frac{\imin}{i}e_i \leq \frac{\ln\pg  1 + \frac{ f_i}{F_{i-1}}\pd}{\ln n} = \frac{\ln\pg  F_i/F_{i-1}\pd}{\ln n}.
\end{equation}
We deduce, recalling from Lemma \ref{lem: orbit growth sequence logarithmic form} that $e_{\imin} = \frac{\ln \pg\imin f_{\imin}\pd}{\ln n} = \frac{\ln \pg \imin F_{\imin}\pd}{\ln n}$, that

\begin{equation}
\begin{split}
     E(\sigma) = \sum_{i\geq 1} \frac{e_i}{i} & = \frac{e_{\imin}}{\imin} + \frac{1}{\imin}\sum_{i>\imin} \frac{\imin}{i}e_i \\ & \leq \frac{1}{\imin\ln n} \pg \ln \pg \imin F_{\imin} \pd + \sum_{\imin<i\leq n} \ln\pg\frac{F_i}{F_{i-1}}\pd \pd \\
     & = \frac{1}{\imin\ln n} \ln \pg \imin F_{\imin}\prod_{\imin <i\leq n} \frac{F_i}{F_{i-1}}  \pd\\
     & =  \frac{\ln (\imin F_n)}{\imin\ln n},
\end{split}
\end{equation}
which concludes the proof since $F_n = \cyc(\sigma)$.
\end{proof}

We can now prove Proposition \ref{prop: borne E sigma avec max cyc 2}.

\begin{proof}[Proof of Proposition \ref{prop: borne E sigma avec max cyc 2}]
First, if $\imin(\sigma) = 1$, this follows immediately from Proposition \ref{prop: borne E sigma avec imin}.

Second, if $\cyc(\sigma) =1$ then $\sigma$ is an $n$-cycle; in this case we have $E(\sigma) = \frac{1}{n} \leq \frac{\ln 2}{\ln n}$ and the result also holds.

Assume finally that $\imin\geq 2$ and $\cyc(\sigma) \geq 2$. Since $x\mapsto \frac{\ln x}{x-1}$ is decreasing on $[2, \infty)$, we have
\begin{equation}
    \ln \cyc(\sigma) \geq \ln 2 = \frac{\ln 2}{2-1} \geq \frac{\ln \imin}{\imin - 1}.
\end{equation}
Therefore $(\imin - 1)\ln \cyc(\sigma) \geq  \ln \imin$, i.e.\ $\imin\ln \cyc(\sigma) \geq  \ln \imin + \ln \cyc(\sigma)$, which rewrites as 
\begin{equation}
   \ln \cyc (\sigma)\geq  \frac{1}{\imin} \ln \pg \imin \cyc(\sigma)\pd.
\end{equation}
Dividing both sides by $\ln n$, and applying Proposition \ref{prop: borne E sigma avec imin}, we conclude that
\begin{equation}
    \frac{\ln \pg \cyc(\sigma)\pd}{\ln n} \geq \frac{1}{\imin} \frac{\ln \pg \imin \cyc(\sigma)\pd}{\ln n} \geq E(\sigma). \qedhere
\end{equation}
\end{proof}

\subsection{Character bound}\label{s: character bound in terms of cyc sigma}

Combining Theorem \ref{thm:ALS improved character bound} and Proposition \ref{prop: borne E sigma avec max cyc 2}, we can prove Theorem \ref{thm: borne caractères nombre de cycles}.

\begin{proof}[Proof of Theorem \ref{thm: borne caractères nombre de cycles}]
    By Theorem \ref{thm:ALS improved character bound} and Proposition \ref{prop: borne E sigma avec max cyc 2}, we have (since $E(\sigma) \leq 1$)
    \begin{equation}
        \bg \ch^\lambda(\sigma)\bd  \leq d_\lambda^{E(\sigma)\pg 1 + \frac{C}{\ln n} \pd} \leq d_\lambda^{E(\sigma) + \frac{C}{\ln n}} \leq d_\lambda^{\frac{\ln \left(\max(\cyc(\sigma),2)\right)}{\ln n} + \frac{C}{\ln n}}.
    \end{equation}
This allows us to conclude since $
    \ln \left(\max(\cyc(\sigma),2)\right) \leq \ln \pg 2\cyc(\sigma) \pd \leq \ln \left(\cyc(\sigma)\right) + 1$.
\end{proof}

\section{Bounds on the Witten zeta function} \label{s: bounds Witten zeta improvement Liebeck Shalev}

Recall from \eqref{eq:def witten zeta} our definition of the Witten zeta function: for $n\geq 1$, $A \subseteq \widehat{\mathfrak{S}_n}$, and $s\geq 0$,
\begin{equation}
    \zeta_n (A,s) := \sum_{\lambda\in A} \frac{1}{(d_\lambda)^{s}}.
\end{equation}
For $k\in [0, \infty)$, define
\begin{equation}
    \Lambda^n(k) = \ag \lambda \vdash n \du \lambda_1' \leq \lambda_1 \leq n-k\ad,
\end{equation}
where we recall that $\lambda_1$ and $\lambda_1'$ respectively denote the length of the first row and column.

Liebeck and Shalev \cite[Proposition 2.5]{LiebeckShalev2004} proved that if $s>0$ and $k$ are fixed, then as $n \to \infty$ we have the bound
\begin{equation}
    \zeta_n (\Lambda^n(k),s) = O\pg n^{-ks} \pd.
\end{equation}

This bound is very convenient. It was used for example by Larsen and Shalev \cite{LarsenShalev2008} in combination with the Diaconis--Shahshahani upper bound lemma (\cite[Lemma 1, page 24 in Chapter 3B]{LivreDiaconis1988}) to derive mixing time results from character bounds. 

For some applications that we will consider, we want asymptotic bounds on $\zeta_n(\Lambda^{n}(k),s_n)$ where the argument $s_n$ is not fixed, but tends to $0$ as $n \rightarrow \infty$ instead. Bounds of this kind were proved by Hough \cite[Proposition 4]{Hough2016}, which we realized only after writing the first version of this paper. We note that Proposition \ref{prop: bounds witten zeta intro iff} could also be derived from the estimates of Hough.

The goal of this section is to further refine these estimates. Our strategy is similar to that of \cite{LiebeckShalev2004}, except that we consider the sums $\sum_{\lambda\in A} d_\lambda^{-s}$ associated to a more efficient decomposition into subsets $A$ of Young diagrams. More precisely, Liebeck and Shalev split the diagrams $\lambda$ according to whether $n-\lambda_1$ is larger than $n/3$; we replace $n/3$ by $n^{2/3}$, which is more suited for our purposes. On the other hand, Hough’s proofs were different, and relied on the Larsen--Shalev bound on virtual degrees (recalled in Theorem \ref{thm:LS virtual degree asymptotic bound}).

Our arguments are elementary and we give self-contained proofs.

\subsection{Elementary bounds on dimensions}
We start with elementary lower bounds on dimensions.
\begin{lemma}\label{lem: lower bounds dimensions in preparation of Witten zeta}
    There exists $n_0\geq 1$ such that the following holds. Let $n\geq n_0$ and $\lambda \vdash n$. Denote $r = r(\lambda) := n-\lambda_1$.
\begin{enumerate}
    \item If $\lambda_1 \geq n- n^{2/3}$, then $n^r \geq d_\lambda \geq n^r \frac{e^{-3r^2/n}}{r!} \geq n^{r/6}$.
    \item If $\lambda_1' \leq \lambda_1$, then $d_\lambda \geq e^{c_0 r}$, where $c_0>0$ is a small universal constant.
\end{enumerate}
\end{lemma}
\begin{proof}
\begin{enumerate}
\item We have $H(\lambda, \lambda) \geq H(\lambda_1, \lambda_1) =(n-r)!$. We deduce from the hook length formula that $d_\lambda \leq n!/(n-r)! \leq n^r$. 
Conversely, the number of standard tableaux of $\lambda$ is at least the number of standard tableaux of $\lambda$ with $1,2,\ldots, r$ placed first on the first row (we have $\lambda_1 \geq r$ by assumption), which is equal to $\binom{n-r}{r}d_{\lambda_{\geq 2}}$. We deduce that, for $n$ large enough, uniformly in $\lambda \vdash n$ with $r \leq n^{2/3}$,
\begin{equation}\label{eq: some inter eq 3r sur n}
    d_\lambda \geq \binom{n-r}{r} = \frac{1}{r!}(n-r)^{\downarrow r} \geq \frac{1}{r!}(n-2r)^r \geq \frac{1}{r!}n^re^{-3r^2/n}.
\end{equation}
Finally, $3r/n \leq 1$ so $e^{-3r^2/n} \geq  e^{-r} \geq n^{-r/6}$, and $n^r/r!\geq n^r/r^r \geq n^r/(n^{2/3})^r = n^{r/3}$. Therefore, the right hand side of \eqref{eq: some inter eq 3r sur n} is larger than $n^{r/3}n^{-r/6} = n^{r/6}$.
\item Assume first that $\lambda_1' \leq \lambda_1 < n/10$. Then, since every hook length in $\lambda$ is at most $\lambda_1 + \lambda_1' - 1 \leq 2\lambda_1$, we have $H(\lambda, \lambda) \leq (2\lambda_1)^n\leq (n/5)^n$. Since $n!\geq (n/e)^n$, we deduce from the hook length formula that
\begin{equation}
    d_\lambda = \frac{n!}{H(\lambda,\lambda)} \geq \frac{(n/e)^n}{(n/5)^n} = (5/e)^n \geq (5/e)^r.
\end{equation}
Now assume that $\lambda_1 \geq n/10$ and $\lambda_1' \leq \lambda_1$. We may assume that $n\geq 3$. Since the case $\lambda = [n]$ is trivial, we may also assume that $r\geq 1$. Set $p:= \max(\lf r/40\rf, 1)$. Then $\lambda$ contains a subdiagram $\mu$ which has $2p$ boxes on the first row and $p$ boxes over the first row.
Then, using that the skew diagram $d_{\mu\backslash[p]}$ consists of two disconnected parts of size $p$, we obtain, counting the number of standard tableaux that
\begin{equation}
    d_\lambda \geq d_\mu \geq d_{\mu\backslash[p]} = \binom{2p}{p}d_{\mu_{\geq 2}} \geq \binom{2p}{p} \geq e^{p/2} \geq e^{r/160},
\end{equation}
which concludes the proof.
\qedhere
\end{enumerate}
\end{proof}

\subsection{Asymptotics of the Witten zeta function}
In this section we prove Theorem \ref{thm: witten zeta asymptotics generating function} and Proposition \ref{prop: bounds witten zeta intro iff}. Recall that $p(n)$ denotes the number of integer partitions of the integer $n$.

\begin{lemma}\label{lem: witten partitions 1}
    Let $c>0$ be fixed. We have
     \begin{equation}
        \zeta_n (\Lambda^n(n^{2/3}),c/\ln n) \xrightarrow[n\to \infty]{} 0.
    \end{equation}
\end{lemma}
\begin{proof}
     By Lemma \ref{lem: lower bounds dimensions in preparation of Witten zeta} (b), with the same constant $c_0$, we have as $n\to \infty$:
\begin{equation}
     \zeta_n (\Lambda^n(n^{2/3}),c/\ln n) = 
   \sum_{r\geq n^{2/3}} \sum_{\substack{\lambda \vdash n \du  \\ r(\lambda) = r \text{ and } \lambda_1\geq \lambda_1'}} (d_\lambda)^{-c/\ln n} \leq \sum_{r\geq n^{2/3}} p(r) e^{-c_0rc/\ln n}. 
\end{equation}
Therefore, for $n$ large enough, we have:
\begin{equation}
     \zeta_n (\Lambda^n(n^{2/3}),c/\ln n) \leq \sum_{r\geq n^{2/3}} p(n) e^{-c_0n^{2/3}c/\ln n} \leq np(n) e^{-n^{0.6}}. 
\end{equation}

Hence, as $n \rightarrow \infty$, using that $p(n) = e^{O(\sqrt{n})}$, we get:

\begin{equation}
    \zeta_n (\Lambda^n(n^{2/3}),c/\ln n) \leq e^{O(\sqrt{n})} e^{-n^{0.6}} \to 0. \qedhere
\end{equation}

\end{proof}
\begin{lemma}\label{lem: witten partitions 2}
    Let $c>0$ and $\varepsilon\in (0,1)$ be fixed. There exists $r_0 = r_0(c, \varepsilon)$ such that, for $n$ large enough, we have
     \begin{equation}
        \zeta_n (\Lambda^n(r_0),c/\ln n) \leq \varepsilon. 
    \end{equation}
\end{lemma}
\begin{proof}
    Let $R\geq 1$. Let $n\geq 3$, which is assumed to be large in the whole proof. By Lemmas \ref{lem: witten partitions 1} and \ref{lem: lower bounds dimensions in preparation of Witten zeta} (a), we have:
\begin{equation}
    \begin{split}
     \zeta_n (\Lambda^n(R),c/\ln n)
     & =
     \zeta_n (\Lambda^n(R),c/\ln n) - \zeta_n (\Lambda^n(n^{2/3}), c/\ln n) +o(1) \\
     & = 
   \sum_{R\leq r < n^{2/3}} \sum_{\substack{\lambda \vdash n \du  r(\lambda) = r}} (d_\lambda)^{-c/\ln n} +o(1)\\
   & \leq \sum_{R\leq r < n^{2/3}} p(r) \pg n^{-r/6}\pd^{c/\ln n} +o(1)\\
   & \leq \sum_{r=R}^\infty p(r) e^{-rc/6}+o(1).
   \end{split}
\end{equation}
The last sum is finite since $p(r) = e^{O(\sqrt{r})}$. Furthermore, $\sum\limits_{r=R}^\infty p(r) e^{-rc/6} \rightarrow 0$ as $R\to \infty$, and the desired result follows.
\end{proof}
\begin{lemma} \label{lem: witten partitions 3}
Fix $c,r\geq 0$. As $n\to \infty$, we have
\begin{equation}
    e^{-cr} \leq \min_{\lambda\vdash n \du r(\lambda) = r}d_{\lambda}^{-c/\ln n} \leq \max_{\lambda\vdash n \du r(\lambda) = r}d_{\lambda}^{-c/\ln n} \leq e^{-cr}+o(1).
\end{equation}
\end{lemma}
\begin{proof}
    Since $r$ is fixed, $r \leq n^{2/3}$ for $n$ large enough. We have $(n^r)^{-c/\ln n} = e^{-cr}$ and, as $n\to \infty$, $(e^{-3r^2/n}/r!)^{-c/\ln n} = 1+o(1)$, and hence $(n^{r}e^{-3r^2/n}/r!)^{-c/\ln n} = e^{-cr}(1+o(1))=e^{-cr} + o(1)$. The result then follows from Lemma \ref{lem: lower bounds dimensions in preparation of Witten zeta} (a).
\end{proof}
\begin{lemma}\label{lem: witten partitions 4}
    Let $c>0$ and let $k,R$ be two fixed integers such that $0\leq k\leq R$. We have
     \begin{equation}
        \zeta_n (\Lambda^n(k),c/\ln n) - \zeta_n (\Lambda^n(R+1),c/\ln n) \xrightarrow[n\to \infty]{} \sum_{r=k}^{R} p(r)e^{-cr}.
    \end{equation}
\end{lemma}
\begin{proof}
This follows immediately from Lemma \ref{lem: witten partitions 3}, since the sum has a bounded number of terms.
\end{proof}

For $k\geq 0$ and $x\in [0,1)$, we set
\begin{equation}
    \Phi_{\geq k}(x) = \sum_{r=k}^\infty p(r)x^r.
\end{equation}
For $n\geq 1$ and $k\geq 0$, we set
\begin{equation}
    \Lambda^n_{\textup{sym}}(k) = \ag \lambda \in \widehat{\mathfrak{S}_n} \du \max(\lambda_1, \lambda'_1) \leq n-k \ad.
\end{equation}
\begin{proposition}\label{prop: convergence witten partitions level at least k}
    Let $c>0$ and $k\geq 0$ be fixed. Then
    \begin{equation}
         \frac{1}{2}\zeta_n (\Lambda_{\textup{sym}}^n(k),c/\ln n) \xrightarrow[n\to \infty]{} \Phi_{\geq k}(e^{-c}).
    \end{equation}
\end{proposition}
\begin{proof}
    Fix $\varepsilon>0$. By Lemma \ref{lem: witten partitions 2}, there exists $r_0 = r_0(c, \varepsilon)$ such that, for $n$ large enough, we have $\zeta_n (\Lambda^n(r_0),c/\ln n) \leq \varepsilon$. Up to taking $r_0$ larger we also have $\sum_{r=r_0}^{\infty}p(r)e^{-cr}\leq \varepsilon$, i.e.\ $\Phi(e^{-c}) - \sum_{r=0}^{r_0-1}p(r)e^{-cr} \leq \varepsilon$. We may also assume that $r_0 \geq k$. Since a diagram and its conjugate have the same dimension, by symmetry we have for $n$ large
     \begin{equation}
     \begin{split}
           \zeta_n (\Lambda^n(k)\backslash\Lambda^n(r_0),c/\ln n) & = \frac{1}{2}\sum_{\lambda \in \Lambda^n_{\textup{sym}}(k)\backslash \Lambda^n_{\textup{sym}}(r_0)} d_\lambda^{-c/\ln n} \\
        & \leq \frac{1}{2}\sum_{\lambda \in \Lambda^n_{\textup{sym}}(k)} d_\lambda^{-c/\ln n} \\ & \leq \zeta_n (\Lambda^n(k)\backslash\Lambda^n(r_0),c/\ln n)
            + \zeta_n (\Lambda^n(r_0),c/\ln n),
     \end{split}
    \end{equation}
where we have used that, if $\lambda_1 \geq n-r_0$, then necessarily $\lambda_1>\lambda_1'$.
Therefore, using also Lemma \ref{lem: witten partitions 4}, we have for $n$ large enough:
\begin{equation}
    \bg \frac{1}{2}\sum_{\lambda \in \Lambda^n_{\textup{sym}}(k)} d_\lambda^{-c/\ln n} - \Phi_{\geq k}(e^{-c}) \bd \leq 3\varepsilon
\end{equation}
This concludes the proof.
\end{proof}

The asymptotic results for the Witten zeta function immediately follow, as we now see.

\begin{proof}[Proof of Theorem \ref{thm: witten zeta asymptotics generating function}]
    This is the case $k=0$ of Proposition \ref{prop: convergence witten partitions level at least k}, noting that $\Lambda^n_{\textup{sym}}(0)$ is the set of all partitions of $n$.
\end{proof}

\begin{proposition}\label{prop: witten zeta iff condition S n ** new}
    Let $k\geq 1$. Let $(s_n)_{n\geq 3}$ be a sequence of positive real numbers. Then
\begin{equation}
    \zeta_n (\Lambda_{\mathrm{sym}}^{n}(k), s_n)
          \xrightarrow[n\to\infty]{} 0 \quad \text{ if and only if }  \quad s_n \ln n   \xrightarrow[n\to\infty]{} \infty.
\end{equation}
\end{proposition}
\begin{proof}
Let $k\geq 1$. By Proposition \ref{prop: convergence witten partitions level at least k}, for any $c>0$ we have
\begin{equation}
    \zeta_n(\Lambda^n_{\textup{sym}}(k), c/\ln n) \xrightarrow[n\to \infty]{} 2\Phi_{\geq k}(e^{-c}).
\end{equation}
Since, for $n$ large enough, the function $(0,\infty) \ni c\mapsto \zeta_n(\Lambda^n_{\textup{sym}}(k), c/\ln n)$ is decreasing, and $0\ne \Phi_{\geq k}(e^{-c}) \xrightarrow[c\to \infty]{} 0$, the result follows.
\end{proof}

We finally prove Proposition \ref{prop: bounds witten zeta intro iff}.

\begin{proof}[Proof of Proposition \ref{prop: bounds witten zeta intro iff}]
     This is the case $k=1$ of Proposition \ref{prop: witten zeta iff condition S n ** new}, since $\Lambda^n_{\textup{sym}}(1) = \widehat{\kS_n}^{**}$.
\end{proof}

\subsection{Asymptotics at larger scales}

In Proposition \ref{prop: bounds witten zeta intro iff}, we showed that if $(s_n)_{n \geq 3}$ is such that $s_n\ln n \to \infty$, then $\zeta_n(\Lambda^n_{\textup{sym}}(1), s_n) \to 0$. We now quantify the speed of convergence to 0. This will be useful later, in the proof of Theorem \ref{thm: number of vertices in random maps} in which we investigate the number of vertices of some random maps.

\begin{lemma}\label{lem: Witten t u}
    Let $n\geq 3$. Let $t>0$ and $u>0$. Then 
    \begin{equation}
        \zeta_n(\Lambda^n_{\textup{sym}}(1), t+u) \leq \zeta_n(\Lambda^n_{\textup{sym}}(1), t) (n-1)^{-u}.
    \end{equation}
\end{lemma}
\begin{proof}
    Since $\min\limits_{\lambda \in \Lambda^n_{\textup{sym}}(1)} d_\lambda = n-1$, we have
    \begin{equation}
        \zeta_n(\Lambda^n_{\textup{sym}}(1), t+u) = \sum_{\lambda\in \Lambda^n_{\textup{sym}}(1)} d_\lambda^{-(t+u)} \leq \sum_{\lambda\in \Lambda^n_{\textup{sym}}(1)} d_\lambda^{-t} (n-1)^{-u} = \zeta_n(\Lambda^n_{\textup{sym}}(1), t) (n-1)^{-u}. \qedhere
    \end{equation}
\end{proof}

\begin{proposition}\label{prop: witten zeta larger scale asymptotics}
    Let $(s_n)_{n \geq 3}$ be a sequence of positive real numbers. Assume that $s_n \ln n \to \infty$ as $n\to \infty$. Then, as $n\to \infty$, we have:
    \begin{equation}
        \zeta_n(\Lambda^n_{\textup{sym}}(1), s_n) \sim 2(n-1)^{-s_n}.
    \end{equation}
\end{proposition}
\begin{proof}
    First observe that
\begin{equation}
    \zeta_n(\Lambda^n_{\textup{sym}}(1), s_n) \geq d_{[n-1,1]}^{-s_n} + d_{[2,1^{n-2}]}^{-s_n} = 2(n-1)^{-s_n},
\end{equation}
so we only have to provide an upper bound to $\zeta_n(\Lambda^n_{\textup{sym}}(1), s_n)$.

Let $\varepsilon>0$. Let $c>0$ be large enough so that $\Phi_{\geq 1}(e^{-c}) \leq (1+\varepsilon)e^{-c}$. Applying Lemma \ref{lem: Witten t u} with $t=c/\ln n$ and $u = s_n - c/\ln n$, we get that
\begin{equation}
\begin{split}
    \zeta_n(\Lambda^n_{\textup{sym}}(1), s_n) & \leq   \zeta_n(\Lambda^n_{\textup{sym}}(1), c/\ln n) (n-1)^{-s_n + c/\ln n} \\
    & = (2\Phi_{\geq 1}(e^{-c})+o(1)) (e^{c}+o(1)) (n-1)^{-s_n} \\
    & \leq (2(1+\varepsilon) +o(1))(n-1)^{-s_n}.
\end{split}
\end{equation}
This concludes the proof since $\varepsilon>0$ is arbitrary.
\end{proof}

\section{Products of two conjugacy invariant random permutations}\label{s: Products of two conjugacy invariant random permutations}

\subsection{The Diaconis--Shahshahani upper bound lemma}

Let $n\geq 1$ and $t\geq 1$. Let  $ \boldsymbol{\cC} = (\cC_1,\ldots, \cC_t)$ be a family of conjugacy classes of $\kS_n$. Let $\sgn(\boldsymbol{\cC}) = \prod_{i=1}^t \sgn\pg \cC_i \pd$ and let $\kE\pg \boldsymbol{\cC} \pd$ be the coset of $\kA_n$ in $\kS_n$ with sign $\sgn(\boldsymbol{\cC})$. Note in particular that, for $\sigma \in \kS_n$, we have
\begin{equation}
    \Unif_{\kE(\boldsymbol{\cC})}(\sigma) = \frac{1 + \sgn(\sigma)\sgn(\boldsymbol{\cC})}{n!}.
\end{equation}
Let $(\sigma_i)_{1\leq i\leq t}$ be independent variables such that $\sigma_i \sim \Unif_{\cC_i}$ for $1\leq i\leq t$. We want to understand when the product $\sigma_1 \sigma_2 \cdots \sigma_t$ is close to uniform in $\kE(\boldsymbol{\cC})$. In other words we want to control the quantity
\begin{equation}
   \delta(\boldsymbol{\cC}) := \dtv\pg \Unif_{\cC_1} * \cdots * \Unif_{\cC_t} \; , \; \Unif_{\kE\pg \boldsymbol{\cC} \pd}  \pd,
\end{equation}
and in particular understand under which conditions it is close to 0.  

The Diaconis--Shahshahani upper bound lemma (see \cite[Lemma 1 on Page 24 in Chapter 3B]{LivreDiaconis1988}) can also be written for products of elements in different conjugacy classes, see \cite[Lemma 6.6]{LarsenShalev2008}. Different versions are also written in \cite{Gamburd2006, ChmutovPittel2016, Hough2016}. A detailed and clear proof when $\boldsymbol{\cC} = (\cC,\ldots,\cC)$ was given by Hough \cite[Section 2]{Hough2016}.

We state here a general version that also takes into account the signs of the conjugacy classes. 

\begin{lemma}\label{lem: lemme DS étendu à des classes potentiellement différentes}
Let $n\geq 1$ and $t\geq 1$. Let $\boldsymbol{\cC} = (\cC_1,\ldots, \cC_t)$ be a family of conjugacy classes of $\kS_n$. Then 
    \begin{equation}
       4\delta(\boldsymbol{\cC})^2 \leq  \sum_{\lambda \in \widehat{\kS_n}^{**}} \pg d_\lambda \prod_{i= 1}^t \chi^\lambda(\cC_i)\pd^2.
    \end{equation}
\end{lemma}

\begin{proof}
    We adapt the proof of \cite[Equation (4)]{Hough2016}. We have
    \begin{equation}
    \label{eq:debut for}
        \delta(\boldsymbol{\cC}) = \frac{1}{2} \sum_{\sigma \in \kS_n} \left| \Unif_{\cC_1} * \cdots * \Unif_{\cC_t} (\sigma) - \Unif_{\kE\pg \boldsymbol{\cC} \pd}(\sigma) \right|.
    \end{equation}
Observe that the Fourier coefficients of $\Unif_{\kE\pg \boldsymbol{\cC} \pd}$ evaluated at the trivial representation $[n]$ and the sign representation $[1^n]$ are 
\begin{equation}
    \widehat{\Unif_{\kE(\boldsymbol{\cC})}}([n])=1 \quad \text{ and } \quad  \widehat{\Unif_{\kE(\boldsymbol{\cC})}}([1^n])=\sgn(\boldsymbol{\cC}),
\end{equation}
and $\widehat{\Unif_{\kE(\boldsymbol{\cC})}}(\lambda)=0$ for $\lambda \in \widehat{\kS_n}^{**}$.
Hence, using the Cauchy--Schwarz inequality on \eqref{eq:debut for} and then the Plancherel identity, we obtain

\begin{equation}
     \delta(\boldsymbol{\cC}) \leq \frac{1}{2} \left( \sum_{\lambda \in \widehat{\kS_n}^{**}} \left( \ch^\lambda(\Id) \right)^2 \prod_{i=1}^t \left(\frac{\ch^\lambda(\cC_i)}{\ch^\lambda(\Id)}\right)^2 \right)^{\frac{1}{2}} = \frac{1}{2} \left(  \sum_{\lambda \in \widehat{\kS_n}^{**}} d_\lambda^2 \prod_{i=1}^t (\chi^\lambda(\cC_i))^2 \right)^\frac{1}{2}.
\end{equation}
The result follows.
\end{proof}

\subsection{A sufficient condition for a product of two permutations to be close to uniform}

We are particularly interested in the case of $t=2$ conjugacy classes. We will write results with asymptotic notation (as $n\to \infty$); writing also $f(n) = \omega_+(g(n))$ if $g(n) = o(f(n))$ and $f(n)>0$ (for $n$ large enough).

Plugging the bounds from Theorem \ref{thm:ALS improved character bound} into Lemma \ref{lem: lemme DS étendu à des classes potentiellement différentes} immediately gives the following lemma.

\begin{lemma}\label{lem: borne produit deux classes avec E sigma et O (1/ln n)}
    Let $n\geq 1$. Let $\cC_1, \cC_2$ be conjugacy classes of $\kS_n$ and denote $\boldsymbol{\cC} = (\cC_1, \cC_2)$. Then 
\begin{equation}
    4\delta(\boldsymbol{\cC})^2 \leq  \sum_{\lambda \in \widehat{\kS_n}^{**}} d_\lambda^{2(E(\cC_1) + E(\cC_2) - 1) +O(1/\ln n)}.
\end{equation}
\end{lemma}

Using our bounds on the Witten zeta function we deduce the following sufficient condition for $\Unif_{\cC_1} * \Unif_{\cC_2}$ to be asymptotically close to uniform.

\begin{proposition}\label{prop: sufficient condition for mixing with sum of two E(sigma)}
    For each $n\geq 1$, let $\cC_1^{(n)}$ and $ \cC_2^{(n)}$ be two conjugacy classes of $\kS_n$, and denote $\boldsymbol{\cC}^{(n)} = (\cC_1^{(n)},\cC_2^{(n)})$. Assume that as $n\to \infty$ we have
    \begin{equation}
        E(\cC_1^{(n)}) + E(\cC_2^{(n)}) = 1-\omega_+(1/\ln n).
    \end{equation}
    Then
    \begin{equation}
        \delta(\boldsymbol{\cC}^{(n)}) \xrightarrow[n\to \infty]{} 0.
    \end{equation}
\end{proposition}
\begin{proof}
Plugging the assumption $E(\cC_1^{(n)}) + E(\cC_2^{(n)}) = 1-\omega_+(1/\ln n)$ into Lemma \ref{lem: borne produit deux classes avec E sigma et O (1/ln n)}, and Proposition \ref{prop: witten zeta iff condition S n ** new}, we obtain
\begin{equation}
     4\delta(\boldsymbol{\cC})^2 \leq  \sum_{\lambda \in \widehat{\kS_n}^{**}} d_\lambda^{-\omega_+(1/\ln n)} = o(1). \qedhere
\end{equation}
\end{proof}

\subsection{Some character bounds}

We gather here some bounds on $E(\cC)$, that we will later use in applications in combination with Proposition \ref{prop: sufficient condition for mixing with sum of two E(sigma)}.

\begin{proposition}\label{prop:bounds on E(sigma) different conditions}
    For each $n\geq 1$, let $\cC^{(n)}$ be a conjugacy class of $\kS_n$.
\begin{enumerate}
    \item Assume that $f_1(\cC^{(n)}) = O(1)$ and $f_2(\cC^{(n)}) = o(n)$. Then 
    \begin{equation}
        E(\cC^{(n)}) = \frac{1}{2} - \omega_+\pg \frac{1}{\ln n}\pd.
    \end{equation}
    \item Assume that $\cyc(\cC^{(n)})=o(\sqrt{n})$. Then 
    \begin{equation}
        E(\cC^{(n)}) = \frac{1}{2} - \omega_+\pg \frac{1}{\ln n}\pd.
    \end{equation}
    \item Assume that $f_1(\cC^{(n)}) = o(n)$. Then 
    \begin{equation}
        E(\cC^{(n)}) = 1 - \omega_+\pg \frac{1}{\ln n}\pd.
    \end{equation}
\end{enumerate}
\end{proposition}
\begin{proof}
In this proof we write $e_i, f_i$ for $e_i(\cC^{(n)}), f_i(\cC^{(n)})$, and $f = \max(f_1,1)$.
    \begin{enumerate}
        \item Since $f = O(1)$ and $f_2 = o(n)$, we have $e_1 = \frac{\ln f}{\ln n} = O\pg \frac{1}{\ln n}\pd$ and 
        \begin{equation}
            e_2 \leq e_1 + e_2 = \frac{\ln(\max(1, f_1 + 2f_2))}{\ln n} = \frac{\ln o(n)}{\ln n} = 1 - \omega_+\pg\frac{1}{\ln n}\pd,
        \end{equation}
        so 
\begin{equation}
    E(\cC^{(n)}) \leq  \frac{1}{2}  - \frac{1}{6}\pg 1- e_2 \pd + \frac{2}{3}e_1 = \frac{1}{2} - \omega_+\pg \frac{1}{\ln n} \pd + O\pg \frac{1}{\ln n}\pd = \frac{1}{2} - \omega_+\pg \frac{1}{\ln n} \pd.
\end{equation}   
\item This follows from Proposition \ref{prop: borne E sigma avec max cyc 2}, since $\frac{\ln o(\sqrt{n})}{\ln n} = \frac{1}{2}-\omega_+\pg \frac{1}{\ln n} \pd$.
\item We have
\begin{equation}
    E(\cC^{(n)}) \leq e_1 + \frac{1-e_1}{2} = 1 -\frac{1-e_1}{2} = 1 - \frac{1}{2}\pg 1 - \frac{\ln f}{\ln n} \pd = 1 - \frac{1}{2} \frac{\ln(n/f)}{\ln n},
\end{equation}
and by assumption $f = o(n)$ so $\frac{\ln(n/f)}{\ln n} = \frac{\ln \omega_+(1)}{\ln n} = \omega_+ \pg \frac{1}{\ln n} \pd$. The result follows. \qedhere
    \end{enumerate}
\end{proof}

Plugging the bounds from Proposition \ref{prop:bounds on E(sigma) different conditions} into Theorem \ref{thm:ALS improved character bound}, we immediately obtain the following character bounds.

\begin{corollary}\label{cor:bounds on characters with different conditions}
    For each $n\geq 1$, let $\cC^{(n)}$ be a conjugacy class of $\kS_n$. The following holds as $n\to \infty$, uniformly over all $\lambda\in \widehat{\kS_n}$.

    \begin{enumerate}
    \item If $f_1(\cC^{(n)}) = O(1)$ and $f_2(\cC^{(n)}) = o(n)$ then $\bg \ch^\lambda(\cC^{(n)}) \bd \leq d_\lambda^{1/2 - \omega_+(1/\ln n)}$.
    \item If $\cyc(\cC^{(n)})=o(\sqrt{n})$ then $\bg \ch^\lambda(\cC^{(n)}) \bd \leq d_\lambda^{1/2 - \omega_+(1/\ln n)}$.
    \item If $f_1(\cC^{(n)}) = o(n)$ then $\bg \ch^\lambda(\cC^{(n)}) \bd \leq d_\lambda^{1- \omega_+(1/\ln n)}$.
\end{enumerate}

\end{corollary}

\section{Characterization of fixed-point free conjugacy classes that mix in two steps}\label{s: beyond the lulov pak conjecture}

\subsection{Random walks and fixed-point free conjugacy classes}

Let $(\cC^{(n)})$ be a sequence of fixed-point free conjugacy classes (that is, such that $f_1(\cC^{(n)}) = 0$). We consider the sequence of random walks on $\kS_n$, with respective increment measures $\Unif_{\cC^{(n)}}$. In this section we restrict to families $\boldsymbol{\cC} = (\cC_1,\ldots,\cC_t)$ where all $\cC_i$ are identical (to some conjugacy class $\cC$). We write $\kE(\cC, t)$ for $\kE(\boldsymbol{\cC})$ and $\mathrm{d}^{(n)}(t)$ for $\delta(\cC)$, that is
\begin{equation}
    \mathfrak{E}(\cC, t) = \begin{cases}
        \mathfrak{S}_n \backslash\mathfrak{A}_n & \text{ if } \cC \subset \mathfrak{S}_n \backslash\mathfrak{A}_n \text{ and } t \text{ is odd}\\
        \mathfrak{A}_n & \text{ otherwise}
    \end{cases},
\end{equation}
is the coset of $\kA_n$ on which the walk is supported after $t$ steps, and
\begin{equation}
    \mathrm{d}^{(n)}(t) = \dtv\pg \Unif_{\cC^{(n)}}^{*t}, \Unif_{\kE(\cC^{(n)},t)}\pd
\end{equation}
for $t\geq 0$.

Note that after 1 step the walk (started at $\Id$) is concentrated on $\cC^{(n)}$ so $\mathrm{d}^{(n)}(1) = 1-o(1)$ and the walk cannot have mixed yet. Larsen and Shalev proved in \cite{LarsenShalev2008} that the mixing time of such a sequence of walks is 2 or 3, i.e.\ (in light of what precedes) that $\mathrm{d}^{(n)}(3) \to 0$. This settled a conjecture due to Lulov and Pak (\cite[Conjecture 4.1]{LulovPak2002}).

A natural refinement is then to understand for which conjugacy classes the mixing time is 2 and for which ones it is 3; in other words to understand when we have $\mathrm{d}^{(n)}(2) \to 0$. The goal of this section is to prove such a characterization, Theorem \ref{thm: application Lulov Pak extended}.

\subsection{Bounds on the tail of the number of fixed points}

A way to prove a lower bound on the total variation distance  between two measures $\mu$ and $\nu$ is to find a good \textit{splitting} event $A$, which is an event $A$ such that $\mu(A)$ is large and $\nu(A)$ is small. For random transpositions, Diaconis and Shahshahani considered the number of fixed points. Here we also consider the number of fixed points but in a different way. We show that if the supports of permutations in $\cC_1, \cC_2$ have order $n$ transpositions and $x_1 \sim\Unif_{\cC_1}$, $x_2\sim\Unif_{\cC_2}$ are independent, then in the product permutation $x_1x_2$ many transpositions compensate, leading to a significantly higher probability to have many fixed points than for a uniform permutation.
There are strong hints that in this case the distribution of the number of fixed points is asymptotically not $\Poiss(1)$ such as asymptotic bounds on the second moment in \cite[Section 3]{KammounMaida2020conjugacyinvariant}, but this is not enough to conclude the non-convergence in distribution to $\Poiss(1)$, and in this section we prove the non-convergence.

\medskip

For $n\geq 1$ and $m\geq 0$ we set
    \begin{equation}
        E_m = E_m^{(n)} = \ag \sigma \in \kS_n \du f_1(\sigma) \geq m \ad.
    \end{equation}

For $n\geq 1$ we set $X = X_n \sim \Unif_{\kA_n}$, $Y= Y_n \sim \Unif_{\kS_n}$, and $Z = Z_n\sim \Poiss(1)$.

We recall a classical fact on the distribution of the number of fixed points of random permutations, usually attributed to Goncharov. 

\begin{lemma}[\cite{Goncharov1942distribution}]\label{lem: convergence distribution points fixes permutations aléatoires vers Poisson}
    We have
    \begin{equation}
        \dtv\pg \cL(f_1(Y_n)), \Poiss(1) \pd \xrightarrow[n\to \infty]{} 0.
    \end{equation}
\end{lemma}
There are many proofs of this result and the convergence is very fast, see the discussions in \cite{DiaconisMiclo2023fixedpoints} and \cite{Fulman2024commutators}.
Here we will only need rough bounds on the tails for the uniform distribution on the alternating group $\kA_n$.

\begin{lemma}\label{lem:borne sup simple queue de Poisson}
Let us fix $m\geq 0$. Then there exists $n_0 = n_0(m)$ such that for all $n\geq n_0$ we have

\begin{equation}
    \Unif_{\kA_n}(E_m) \leq \frac{3}{m!} \quad \text{ and } \quad \Unif_{\kS_n \backslash \kA_n}(E_m) \leq \frac{3}{m!}.
\end{equation}
\end{lemma}
\begin{proof}
    We prove the result for $\Unif_{\kA_n}(E_m)$; the same arguments apply to $\Unif_{\kS_n \backslash \kA_n}(E_m)$, up to replacing $X_n$ by a variable that is uniform on $\kS_n \backslash \kA_n$.
    
    First, since $|\kS_n| = 2|\kA_n|$, we have $\bbP(f_1(X_n)\geq m) \leq 2 \bbP(f_1(Y_n)\geq m)$ for all $n$. Moreover, by Lemma \ref{lem: convergence distribution points fixes permutations aléatoires vers Poisson}, for $n$ large enough we have 
    \begin{equation}
        2\bbP(f_1(Y_n)\geq m) \leq e \bbP(Z_n \geq m) = \sum_{j\geq m} \frac{1}{j!} \leq \frac{3}{m!}.
    \end{equation}
This concludes the proof since $\bbP(f_1(X_n)\geq m) = \Unif_{\kA_n}(E_m)$.
\end{proof}

Let us now show that the cancellations of transpositions after two steps lead to many fixed points.
The use of consecutive transpositions in the next proof is inspired by the arguments with unseparated pairs in \cite{DiaconisEvansGraham2014unseparatedpairs}.

\begin{lemma}\label{lem:lower bound fixed points via transpositions}
    Let $0<\alpha<1/2$ and $m\geq 2$ even. There exists $n_1 = n_1(\alpha, m)$ such that for every $n\geq n_1$ and every conjugacy classes $\cC_1, \cC_2$ of $\mathfrak{S}_n$ such that $\min(f_2(\cC_1), f_2(\cC_2)) \geq \alpha n$, we have
    \begin{equation}
        \Unif_{\cC_1^{(n)}}*\Unif_{\cC_2^{(n)}}(E_m) \geq \frac{C(m,\alpha)}{(m/2)!},
    \end{equation}
    where $C(m,\alpha)=\frac12 \min_{\gamma \in [\alpha,1/2]} e^{-2\alpha \gamma} (2\alpha\gamma)^{m/2}=\frac12 e^{-2\alpha^2} (2\alpha^2)^{m/2}$.
\end{lemma}

\begin{proof}
First, observe that
\begin{equation}
C(m,\alpha) := \frac12 \min_{\gamma \in [\alpha,1/2]} e^{-2\alpha \gamma} (2\alpha\gamma)^{m/2} = \frac{1}{2} \min_{\delta \in [2\alpha^2,\alpha]} e^{-\delta}  \delta^{m/2}.
\end{equation}
Let $g: \bbR_+\ni \delta \mapsto e^{-\delta} \delta^{m/2} \in \bbR_+$. Since $m/2\geq\delta$ by assumption, for $\delta \in [2\alpha^2,\alpha]$ we have $g'(\delta)=(m/2-\delta)e^{-\delta} \delta^{m/2-1} \geq 0$. $g$ is therefore increasing on $[2\alpha^2,\alpha]$, which implies that $C(m,\alpha)=\frac12 e^{-2\alpha^2} (2\alpha^2)^{m/2}$.

We prove the result by contradiction. Fix $m \geq 2$ even, and assume that there exists a sequence $(\cC_1^{(n)}, \cC_2^{(n)})_{n \geq 1}$ with $\min (f_2(\cC_1^{(n)}), f_2(\cC_2^{(n)})) \geq \alpha n$ such that, for infinitely many values of $n$, $\Unif_{\cC_1^{(n)}}*\Unif_{\cC_2^{(n)}}(E_m) < C(m,\alpha)/(m/2)!$. 
Up to extracting a subsequence, one can also assume that there exists $\gamma \in [\alpha,1/2]$ such that $\frac{f_2(\cC_2^{(n)})}{n} \rightarrow \gamma$.

For $n \geq 1$, let $q := 2\lf\frac{\alpha n}{2} \rf$ and let $x_1 \sim \Unif_{\cC_1^{(n)}}$, $x_2 \sim \Unif_{\cC_2^{(n)}}$ be independent. Since $f_2(\cC_1^{(n)}) \geq q$, one can assume without loss of generality that 
\begin{equation}
    x_1 := (1 \;\; 2) (3 \;\; 4)\cdots (2q-1 \;\; 2q) \sigma,
\end{equation}
where $\sigma$ is a permutation of the integers from $2q+1$ to $n$.

We denote by $N_n$ the (random) number of transpositions $(2i-1 \;\; 2i)$ for $i \in \ag 1, \ldots, q\ad$ which are also transpositions of $x_2$.
We claim that, as $n \rightarrow \infty$, 
\begin{equation}
\label{eq:convergence vers poisson 2 alpha2}
N_n \overset{(d)}{\rightarrow} \Poiss(2\alpha\gamma).    
\end{equation}
Assuming \eqref{eq:convergence vers poisson 2 alpha2}, observe that $f_1(x_1x_2) \geq 2 N_n$, since each transposition common to $x_1$ and $x_2$ corresponds to two fixed points of $x_1x_2$. Hence, one gets, for all fixed $m \geq 1$:
\begin{equation}
    \P(f_1(x_1x_2)\geq m) \geq \P(N_n \geq m/2) \xrightarrow[n \rightarrow\infty]{} \P(\Poiss(2\alpha\gamma) \geq m/2) \geq e^{-2\alpha\gamma} \frac{(2\alpha\gamma)^{m/2}}{(m/2)!},
\end{equation}
so, for $n$ large enough (depending on $m$ and $\alpha$), we get 
\begin{equation}
\P(f_1(x_1x_2)\geq m) \geq \frac{e^{-2\alpha\gamma}}{2} \frac{(2\alpha\gamma)^{m/2}}{(m/2)!} \geq \frac{C(m,\alpha)}{(m/2)!}.
\end{equation}
We now turn to the proof of \eqref{eq:convergence vers poisson 2 alpha2}. For $r \geq 1$, let $M_r^{(n)} := \E[(N_n)^{\downarrow r}]$ be the $r$-th factorial moment of $N_n$. For any $r \geq 1$ fixed, we get by symmetry
\begin{equation}
\label{eq:moments}
    M_r^{(n)} = q^{\downarrow r}\P\left( (1 \; \; 2),\ldots,(2r-1 \; \; 2r) \text{ appear in } x_2\right).
\end{equation}
Now recall that the cardinality of a conjugacy class $\cC$ of $\kS_n$ is $\frac{n!}{\prod_{k \geq 1} k^{f_k(\cC)} (f_k(\cC))!}$. We thus get, for $n$ large enough so that $q \geq r$:

\begin{equation}
\begin{split}
    \frac{M_r^{(n)}}{q^{\downarrow r}} &= \frac{(n-2r)!}{2^{f_2(\cC_2^{(n)})-r}(f_2(\cC_2^{(n)})-r)!\prod_{k \geq 3}k^{f_k(\cC_2^{(n)})} f_k(\cC_2^{(n)})!}\left(\frac{n!}{\prod_{k \geq 2}k^{f_k(\cC_2^{(n)})} f_k(\cC_2^{(n)})!}\right)^{-1} \\
&= 2^r \frac{f_{2}(\cC_2^{(n)})^{\downarrow r}}{n^{\downarrow 2r}}\\
&\sim 2^r n^{-2r} f_{2}(\cC_2^{(n)})^r\\
&\sim \left(\frac{2\gamma}{n}\right)^r.
\end{split}
\end{equation}
Plugging it into \eqref{eq:moments} and using that $q \sim \alpha n$, one gets
\begin{equation}
    M_r^{(n)} \sim (\alpha n)^r \left(\frac{2\gamma}{n}\right)^r = (2\alpha \gamma)^r.
\end{equation}
One recognizes the $r$-th factorial moment of $\Poiss(2\alpha \gamma)$. Since the Poisson distributions are characterized by their moments, this shows \eqref{eq:convergence vers poisson 2 alpha2}, and thus concludes the proof.
\end{proof}

We can now use this result to prove that $E_m$ is a splitting event.

\begin{proposition}\label{prop:lower bound at time 2 via compensations of transpositions}
Let $0<\alpha<1/2$. There exists $c = c(\alpha)>0$, $m=m(\alpha) \geq 1$ and $n_2 = n_2(\alpha)$ such that for every $n\geq n_2$, and every conjugacy classes $\cC_1, \cC_2$ of $\mathfrak{S}_n$ such that $\min(f_2(\cC_1), f_2(\cC_2)) \geq \alpha n$, we have 
\begin{equation}
    \Unif_{\cC_1}*\Unif_{\cC_2}(E_{m(\alpha)}) - \Unif_{\kE(\boldsymbol{\cC})}(E_{m(\alpha)}) \geq c(\alpha),
\end{equation}
where $\boldsymbol{\cC} = (\cC_1, \cC_2)$. In particular
\begin{equation}
    \dtv\pg \Unif_{\cC_1}*\Unif_{\cC_2},\Unif_{\kE(\boldsymbol{\cC})}\pd \geq c(\alpha).
\end{equation}
\end{proposition} 
\begin{proof}
First observe that for any even integer $m\geq 2$, we have, setting $\beta = e^{-2\alpha^2}\alpha^2$,
\begin{equation}
\begin{split}
     m!\pg\frac{e^{-2\alpha^2}}{2}\frac{(2\alpha^2)^{m/2}}{(m/2)!}\pd =  m^{\downarrow (m/2)} \frac{e^{-2\alpha^2}}{2}(2\alpha^2)^{m/2} \geq \beta^{m/2} m^{\downarrow (m/2)} & \geq \beta^{m/2} (m/2)^{m/2} \\
     & = \pg\beta m/2\pd^{m/2}.
\end{split}
\end{equation}
We now fix $m(\alpha) := 2\lc\frac{5}{\beta} \rc$. Then 
\begin{equation}
    m(\alpha)! \frac{C(m(\alpha),\alpha)}{(m(\alpha)/2)!} \geq \pg\frac{\beta m(\alpha)}{2}\pd^{m(\alpha)/2} \geq \frac{\beta m(\alpha)}{2} \geq 5,
\end{equation}
where $C(m,\alpha)$ is the same constant as in Lemma \ref{lem:lower bound fixed points via transpositions}.
We deduce from Lemma \ref{lem:lower bound fixed points via transpositions} that for $n$ large enough
\begin{equation}
    \Unif_{\cC_1}*\Unif_{\cC_2}(E_{m(\alpha)}) \geq \frac{5}{m(\alpha)!}.
\end{equation}
We conclude using Lemma \ref{lem:borne sup simple queue de Poisson}, that for $n$ large enough
\begin{equation}
    \Unif_{\cC_1}*\Unif_{\cC_2}(E_{m(\alpha)}) - \Unif_{\kE(\boldsymbol{\cC})}(E_{m(\alpha)}) \geq \frac{5}{m(\alpha)!} - \frac{3}{m(\alpha)!} = \frac{2}{m(\alpha)!} =: c(\alpha) >0. \qedhere
\end{equation}
\end{proof}

\subsection{Proof of Theorem \ref{thm: application Lulov Pak extended}}

We now have all the ingredients to prove Theorem \ref{thm: application Lulov Pak extended}.

\begin{proof}[Proof of Theorem \ref{thm: application Lulov Pak extended}]
 First, by Proposition \ref{prop:lower bound at time 2 via compensations of transpositions}, if there exists some $\alpha>0$ such that $f_2(\cC^{(n)}) \geq  \alpha n$ for infinitely many values of $n$, then $\dtv\pg \Unif_{\cC^{(n)}}^{*2},\Unif_{\kA_n}\pd$ does not converge to 0.
 
Assume now that $f_2(\cC^{(n)}) = o(n)$. By Proposition \ref{prop:bounds on E(sigma) different conditions} (a), we have $E(\cC^{(n)}) = \frac{1}{2} - \omega_+\pg \frac{1}{\ln n}\pd$, i.e.\
\begin{equation}
    2E(\cC^{(n)}) = 1 - \omega_+\pg \frac{1}{\ln n}\pd. 
\end{equation}
It follows from Proposition \ref{prop: sufficient condition for mixing with sum of two E(sigma)} (here $\boldsymbol{\cC}^{(n)} = (\cC^{(n)}, \cC^{(n)})$ so $\delta(\boldsymbol{\cC}^{(n)}) = \dtv\pg \Unif_{\cC^{(n)}}^{*2},\Unif_{\kA_n}\pd$) that 
\begin{equation}
    \dtv\pg \Unif_{\cC^{(n)}}^{*2},\Unif_{\kA_n}\pd \xrightarrow[n\to \infty]{} 0,
\end{equation}
which concludes the proof of Theorem \ref{thm: application Lulov Pak extended}.
\end{proof}

\section{Application to random surfaces}\label{s: random maps}

The convergence in distribution of the product of two permutations that are uniform in their conjugacy classes has an interesting application to random maps, which was first spotted by Gamburd \cite{Gamburd2006} (see also Chmutov and Pittel \cite{ChmutovPittel2016} and Budzinski, Curien and Petri \cite{BudzinskiCurienPetri2019}). The main idea is to code a map by two permutations $\alpha,\beta$, where $\alpha$ codes the polygon structure of the faces of the map, and $\beta$ the way these polygons are glued together. More precisely, consider a collection of $n$ oriented polygons with various numbers of sides, the total number of sides being an even number, say $N$. Let $n_j$ be the number of polygons with $j$ sides, so that $\sum_{j}n_j=n$ and $\sum_j j n_j = N$. Labeling the sides of the polygons by $\{1, \ldots, N\}$ uniformly at random, we define a permutation $\alpha$ of $\{1, \ldots, N\}$ by $\alpha(i)=j$ if the side labeled $j$ immediately follows the side labeled $i$ around a polygon in counterclockwise order. Hence, the cycle structure of $\alpha$ codes the number of sides of the polygons.  We now build a surface $M_N$ from it by gluing these $N$ oriented sides by pairs, uniformly at random. This gluing can be seen as a uniform fixed-point free involution $\beta$ of $\{1, \ldots, N\}$, where $\beta(i)=j$ (and thus $\beta(j)=i$) if the sides labeled $i$ and $j$ are glued together. One can then check that each cycle of the product $\alpha \beta$ corresponds to a vertex of the resulting surface, and that the number of orbits of the subgroup generated by $\alpha$ and $\beta$ is the number of connected components of $M_N$.

Since the labels are random, the permutation $\alpha$ is uniform in the conjugacy class corresponding to the $n_j$'s, and $\beta$ is a uniform fixed-point free involution, $\alpha$ and $\beta$ being independent.  

Both \cite{Gamburd2006} and \cite{ChmutovPittel2016} rely on this permutation point of view and on representation theory. Gamburd’s result is for regular graphs (with degree $d\geq 3$, i.e.\ when $n_d = N/d$ and $n_j = 0$ for $j\ne d$) and relies on the character estimates of Fomin and Lulov \cite{FomLul95}. They show that, if $\alpha \sim  [m^{N/m}]$ for some fixed $m\geq 3$, then the permutation $\alpha \beta$ is close to being uniform (for the total variation distance). 
Chmutov and Pittel then used the more general bounds of Larsen and Shalev \cite{LarsenShalev2008} to prove that $\alpha \beta$ is close to uniform in total variation as soon as $n_1 = n_2 = 0$, that is when all cycles of $\alpha$ have length at least 3.

Extending Gamburd's results, Chmutov and Pittel obtain the following (\cite[Theorem 3.1]{ChmutovPittel2016}). Here and in what follows, $\cL(X)$ denotes the law of a random variable $X$.

\begin{theorem}[{\cite[Theorem 3.1]{ChmutovPittel2016}}]
\label{thm:chmutovpittel}
    Assume that $n_1=n_2=0$, that is, the polygons all have at least $3$ sides. Let $V_N$ be the number of vertices of $M_N$, and $\cL(V_N)$ its distribution. Then:
    \begin{equation}
        \dtv(\cL(V_N), \cL(\widetilde{C}_N)) = O(N^{-1}),
    \end{equation}
    uniformly for all sequences $(n_j)_{j \geq 3}$. Here, $\widetilde{C}_N$ denotes the number of cycles in a uniform permutation of $\kA_N$ (resp.\  $\kS_N \backslash \kA_N$) if the classes $\alpha$ and $\beta$ have the same sign (resp.\  opposite signs).
\end{theorem}

They deduce from this theorem a local central limit theorem for $V_N$, see \cite[Theorem 3.2]{ChmutovPittel2016}.

Using a probabilistic approach and relying on the peeling process of a random map, Budzinski, Curien and Petri \cite[Theorem 3]{BudzinskiCurienPetri2019} show that $\dtv(\cL(V_N), \cL(\widetilde{C}_N)) \rightarrow 0$, under the weaker assumptions that $n_1=o(\sqrt{n})$ and $n_2=o(n)$.

Since our results sharpen the Larsen--Shalev character bounds, it is no surprise that we can improve on the results of Gamburd and Chmutov--Pittel following the same strategy. It turns out that the improved bounds on the Witten zeta function are also needed, and we use Proposition \ref{prop:lower bound at time 2 via compensations of transpositions} to prove Theorem \ref{thm: products of two permutations and random maps}. Its proof is an adaptation of that of Theorem \ref{thm: application Lulov Pak extended}.

\begin{proof}[Proof of Theorem \ref{thm: products of two permutations and random maps}]
    First assume that $f_2( \cC_1^{(n)} ) = o(n)$. Then, by Proposition \ref{prop:bounds on E(sigma) different conditions} (a), we have $E(\cC_1^{(n)}) = 1/2- \omega_+(1/\ln n)$. Furthermore, we have $E(\cC_2^{(n)}) = 1/2$, so $E(\cC_1^{(n)}) + E(\cC_2^{(n)}) = 1- \omega_+(1/\ln n)$. We conclude using Proposition \ref{prop: sufficient condition for mixing with sum of two E(sigma)} that $\dtv\pg \Unif_{\cC_1^{(n)}}*\Unif_{\cC_2^{(n)}}, \Unif_{\kE(\boldsymbol{\cC}^{(n)})} \pd \to 0$ as $n\to \infty$.
On the other hand, if $f_2( \cC_1^{(n)} ) = o(n)$ does not hold, then there exists $0<a\leq 1/2$ such that $f_2( \cC_1^{(n)}) \geq a n$ along a subsequence, and therefore (since $f_2( \cC_2^{(n)}) =n/2 \geq a n$), by Proposition \ref{prop:lower bound at time 2 via compensations of transpositions} the convergence does not hold.
\end{proof}

In terms of random maps, Theorem \ref{thm: products of two permutations and random maps} immediately implies the following theorem.

\begin{theorem}\label{thm: number of vertices in random maps}
 Consider for all $N$ a sequence $(n_j^{(N)})_{1 \leq j \leq N}$ such that $\sum_{j} j n_j^{(N)}=N$. Assume that $n_1^{(N)}=0$ for all $N$, and that $n_2^{(N)}=o(N)$. Let $V_N$ be the number of vertices of $M_N$, and $\widetilde{C}_N$ be the number of cycles in a uniform permutation of $\kA_N$ (resp.\  $\kS_N \backslash \kA_N$) if the conjugacy classes of $\alpha$ and $\beta$ have the same sign (resp.\  opposite signs).
 Then, as $N \rightarrow \infty$:
    \begin{equation}
        \dtv(\cL(V_N), \cL(\widetilde{C}_N)) \rightarrow 0.
    \end{equation}
\end{theorem}

This provides an algebraic proof of \cite[Theorem 3]{BudzinskiCurienPetri2019}, in the case where $n_1^{(N)}=0$ for all $N$. Combining Theorem \ref{thm: number of vertices in random maps} with \cite[Equation (3.5)]{ChmutovPittel2016}, we immediately find that, under the assumptions of Theorem \ref{thm: number of vertices in random maps}, $V_N$ satisfies a central limit theorem.

Applying our bounds carefully allows us to obtain in addition a local central limit theorem when $n_2^{(N)}$ does not grow too fast. 

In the next theorem, $C_N$ denotes the (random) number of cycles in a uniform permutation in $\kS_N$. It is known that $\E[C_N]=\ln N + O(1)$, and $\Var(C_N)=\ln N+ O(1)$.
\begin{theorem}
\label{thm:log puissance 9}
    Assume that $n_1^{(N)}=0$ for all $N$, and that $n_2^{(N)}=o\left( \frac{N}{(\ln N)^{3}} \right)$. Fix $a>0$. Then, uniformly for all integers $\ell$ satisfying
    \begin{equation}
    \label{eq:ell in the bulk}
    \frac{\ell-\E[C_N]}{\sqrt{\Var(C_N)}} \in [-a,a]
    \end{equation}
    and the admissibility condition that $N-\ell$ is even (resp.\ odd) if $\alpha$ and $\beta$ have the same sign (resp.\  opposite signs),
    we have
    \begin{equation}
    \P(V_N=\ell) = \frac{\left(2+o(1)\right) \exp \left( -\frac{\left( \ell-\E[C_N] \right)^2}{2 \Var(C_N)} \right)}{\sqrt{2\pi \Var(C_N)}},
    \end{equation}
    where we recall that $n=\sum_{j} n_j$ is the number of polygons.
\end{theorem}
\begin{proof}
Observe that, if $n_2^{(N)}=0$ for all $N$ large enough, then the result immediately follows from \cite[Theorem 3.2]{ChmutovPittel2016}. Therefore, it is enough to prove our result in the case where $n_2^{(N)} \geq 1$ for all $N$.
Consider $\cC_1^{(N)}$ such that $f_1(\cC_1^{(N)})=0$ and $2f_2(\cC_1^{(N)})=\frac{N}{\gamma_N} \geq 2$, where $\gamma_N =\omega_+((\ln N)^3)$ and the notation $f_i$ is defined above \eqref{eq:orbit growth exponent}. Let $\cC_2^{(N)}$ be the conjugacy class of fixed-point free involutions.
Write $e_i$ for $e_i(\cC_1^{(N)})$. We have $e_1 = 0$ and $e_2 = \frac{\ln (N/\gamma_N)}{\ln N} = 1- \frac{\ln \gamma_N}{\ln N}$, so
\begin{equation}
    E(\cC_1^{(N)}) = \sum_{i\geq 2}\frac{e_i}{i} \leq \frac{e_2}{2} + \frac{1-e_2}{3} = \frac{1}{2} - \frac{1-e_2}{6} = \frac{1}{2} - \frac{1}{6}\frac{\ln \gamma_N}{\ln N}.
\end{equation}
Since $E(\cC_2^{(N)})=\frac12$, and $\gamma_N \to \infty$ we get:
\begin{equation}
    s_N := -2(E(\cC_1^{(N)})+E(\cC_2^{(N)})-1) \geq \frac{1}{3} \frac{\ln \gamma_N}{\ln N} = \omega_+ (1/\ln N).
\end{equation}
By Proposition \ref{prop: witten zeta larger scale asymptotics}, we therefore get that, for any sequence $t_N = O(1/\ln N)$, as $N \to \infty$,
\begin{equation}
     \sum_{\lambda \in \widehat{\kS_N}^{**}} d_\lambda^{2\left(E(\cC_1^{(N)})+E(\cC_2^{(N)})-1 \right) +t_N} \sim 2(N-1)^{-s_N+t_N} =O\pg  \exp\left( -\frac{1}{3} \ln \gamma_N \right) \pd = O(\gamma_N^{-1/3}).
\end{equation}
Plugging this into Lemma \ref{lem: borne produit deux classes avec E sigma et O (1/ln n)}, and using that $\gamma_N = \omega_+((\ln N)^3)$, we get:
\begin{equation}
\label{eq:o de epsilon n}
    \dtv(\cL(V_N), \cL(\widetilde{C}_N)) \leq \dtv\pg \Unif_{\cC_1^{(N)}}*\Unif_{\cC_2^{(N)}}, \Unif_{\kE(\boldsymbol{\cC}^{(N)})} \pd = O(\gamma_N^{-1/6}) = o\pg \frac{1}{\sqrt{\ln N}} \pd.
\end{equation}
Finally, by \cite[Equation (3.5)]{ChmutovPittel2016}, uniformly for any admissible $\ell$ satisfying \eqref{eq:ell in the bulk}, we have
\begin{equation}
\label{eq:equationnulle1}
\P(\widetilde{C}_N=\ell) = \frac{(2+o(1)) \exp\left( - \frac{(\ell-\E[C_N])^2}{2\Var(C_N)} \right)}{\sqrt{2\pi\Var(C_N)}} \geq(1+o(1)) \frac{c_a}{\sqrt{\ln N}},
\end{equation}
where $c_a=2\frac{\exp(-a^2/2)}{\sqrt{2\pi}}>0$ depends only on $a$, and in particular, for any such $\ell$, we have $\P(\widetilde{C}_N=\ell) = \omega_+\pg \dtv(\cL(V_N), \cL(\widetilde{C}_N)) \pd$. This implies that, uniformly for any such $\ell$, we have $\P(V_N=\ell) = (1+o(1))\P(\widetilde{C}_N=\ell)$, which, in view of the equality in \eqref{eq:equationnulle1}, concludes the proof. 
\end{proof}

Obtaining a local central limit theorem in the general case $n_2^{(N)}=o(N)$ would require a better control on the speed of convergence in Theorem \ref{thm: number of vertices in random maps}.

Let us finally mention that \cite[Proposition 14]{BudzinskiCurienPetri2019} states that the surface $M_N$ is connected with probability tending to $1$ as $N \rightarrow \infty$.

\section{Further mixing time estimates}
\label{s: further mixing time estimates}

This last section consists of the proofs of Theorems \ref{thm: mixing time 2 if few cycles} and \ref{thm: cutoff large support no short cycles} about mixing time estimates for specific conjugacy classes, and complementary results about products of two permutations.

\subsection{Products of two permutations with constraints on cycle types}\label{s: mixing time few cycles}

\subsubsection{Mixing time of permutations with few cycles}
We prove here Theorem \ref{thm: mixing time 2 if few cycles}, which states that conjugacy classes of permutations with $o(\sqrt{n})$ cycles mix in 2 steps.

\begin{proof}[Proof of Theorem \ref{thm: mixing time 2 if few cycles}]
This is the same proof as for Theorem \ref{thm: application Lulov Pak extended}: by Proposition \ref{prop:bounds on E(sigma) different conditions} (b), we have $2E(\cC^{(n)}) = 1 - \omega_+\pg \frac{1}{\ln n}\pd$, and it follows from Proposition \ref{prop: sufficient condition for mixing with sum of two E(sigma)} that $\dtv\pg \Unif_{\cC^{(n)}}^{*2},\Unif_{\kA_n}\pd \xrightarrow[n\to \infty]{} 0$.
\end{proof}
\begin{remark}
    The condition $\cyc(\cC^{(n)}) = o(\sqrt{n})$ is not a characterization. For instance if $\cC^{(n)} \sim [3^{n/3}]$, we have $\dtv\pg \Unif_{\cC^{(n)}}^{*2},\Unif_{\kA_n}\pd \xrightarrow[n\to \infty]{} 0$, while $\cyc(\cC^{(n)}) \asymp n$. However, the condition $\cyc(\cC^{(n)}) = o(\sqrt{n})$ is sharp. For example, if $\cC^{(n)}$ is the conjugacy class of cycles of length $k = n - \lfloor \alpha \sqrt{n}\rfloor$ for some fixed $\alpha>0$, one can show that the number of fixed points of permutations taken according to $\sigma \sim \Unif_{\cC^{(n)}}^{*2}$ is asymptotically $\Poiss(1+ \alpha^2)$, and therefore (looking at the probability to have 0 fixed point) that $\dtv\pg \Unif_{\cC^{(n)}}^{*2},\Unif_{\kA_n}\pd +o(1) \geq e^{-1} - e^{-(1+\alpha^2)}>0$.
\end{remark}

\subsubsection{Products of long cycles with other permutations}

The characters of long cycles are very small: if $\sigma \in \kS_n$ is an $n$-cycle, then $E(\sigma) = 1/n$ and therefore by Theorem \ref{thm:ALS improved character bound} we have $ \bg \ch^\lambda(\sigma)\bd \leq d_\lambda^{O(1/n)}$. Actually, by the Murnaghan--Nakayama rule, if $\sigma$ is an $n$-cycle we have $\ch^\lambda(\sigma) = (-1)^i$ if $\lambda = [n-i, 1^i]$ for some $0\leq i \leq n-1$, and $\ch^\lambda(\sigma) = 0$ for other diagrams, and therefore for all diagrams we have $\bg \ch^\lambda(\sigma)\bd \leq 1$. 

If $\cC^{(n)}$ is the conjugacy class of $n$-cycles, the $L^2$ distance to stationarity then rewrites as 
\begin{equation}\label{eq: L2 distance n cycles}
    \mathrm{d}_{L^2} \pg \Unif_{\cC^{(n)}}^{*t}, \Unif_{\kE(\cC^{(n)},t)} \pd = \sqrt{\sum_{i=1}^{n-2} d_{[n-i, 1^i]}^2 \pg \frac{1}{d_{[n-i, 1^i]}} \pd^{2t}} = \sqrt{\sum_{i=1}^{n-2} d_{[n-i, 1^i]}^{2(1-t)}}.
\end{equation}

The bound above hints that $n$-cycles mix very fast: they need hardly more than one step to mix. Although it does not make sense in terms of random walks to consider times that are not integers, if we replace $t$ by $1 + \omega_+(1/\ln n)$ in \eqref{eq: L2 distance n cycles}, then the right hand side tends to $0$.

Another way to see that $n$-cycles mix well is to show that they can compensate the irregularities of other cycle types. This is the content of the next theorem.

\begin{theorem}\label{thm: n cycles compensate irregularities}
For each $n\geq 2$ let $\cC_1^{(n)}$ denote the conjugacy class of $n$-cycles, and let $\cC_2^{(n)}$ be a conjugacy class of $\kS_n$. Denote $\boldsymbol{\cC}^{(n)} = (\cC_1^{(n)}, \cC_2^{(n)})$. Assume that $f_1( \cC_2^{(n)} ) = o(n)$.
Then we have

\begin{equation}
    \dtv\pg \Unif_{\cC_1^{(n)}}*\Unif_{\cC_2^{(n)}}, \Unif_{\kE(\boldsymbol{\cC}^{(n)})} \pd \xrightarrow[n\to \infty]{} 0.
\end{equation}
\end{theorem}
\begin{proof}
    First we have $E(\cC_1^{(n)}) = 1/n =  O(1/\ln n)$, and by Proposition \ref{prop:bounds on E(sigma) different conditions} (c) we have $E(\cC_2^{(n)}) = 1- \omega_+(1/\ln n)$. We deduce that $E(\cC_1^{(n)}) + E(\cC_2^{(n)}) = 1- \omega_+(1/\ln n)$, and conclude by Proposition \ref{prop: sufficient condition for mixing with sum of two E(sigma)} that $\dtv\pg \Unif_{\cC_1^{(n)}}*\Unif_{\cC_2^{(n)}}, \Unif_{\kE(\boldsymbol{\cC}^{(n)})} \pd \to 0$ as $n\to \infty$.
\end{proof}

\begin{remark}
    The same arguments apply for conjugacy classes $\cC_1^{(n)}$ such that $E(\cC_1^{(n)}) = O(1/\ln n)$, for instance if $\cyc(\cC_1^{(n)}) = O(1)$.
\end{remark}

\begin{proposition}
    For each $n\geq 2$ let $\cC_1^{(n)}$ denote the conjugacy class of $n$-cycles, and let $\cC_2^{(n)}$ be a conjugacy class of $\kS_n$. Denote $\boldsymbol{\cC}^{(n)} = (\cC_1^{(n)}, \cC_2^{(n)})$. Assume that there exists a constant $\gamma>0$ such that for $n$ large enough we have $f_1(\cC_2^{(n)}) \geq \gamma n$. Then 
\begin{equation}
\label{eq:convergence en variation totale selon points fixes}
    \dtv\pg \Unif_{\cC_1^{(n)}}*\Unif_{\cC_2^{(n)}}, \Unif_{\kE(\boldsymbol{\cC}^{(n)})} \pd \underset{n \to \infty}{\not\rightarrow} 0.
\end{equation}
\end{proposition}

\begin{proof}
In this proof we drop the dependence on $n$ from the notation of most variables, and we denote $F = f_1(\cC_2^{(n)})$.
Let $\tau_1\sim \Unif_{\cC_1^{(n)}}$ and $\tau_2\sim \Unif_{\cC_2^{(n)}}$ be independent variables. 
Denote by $\tau^*$ a deterministic permutation of $\cC_2^{(n)}$ whose fixed points are $\{1, \ldots, F\}$. By conjugacy invariance of $\tau_1$ we have
\begin{equation}
    f_1(\tau_1 \circ \tau_2) \overset{(d)}{=} f_1(\tau_1 \circ \tau^*),
\end{equation}
where $\overset{(d)}{=}$ denotes the equality in distribution.
Since $\tau_1$ is a uniformly random $n$-cycle, for any $i \in \{1, \ldots, n\}$ we have 
\begin{equation}
   \P(\tau_1 \circ \tau^*(i)=i)  = \left\{
    \begin{array}{ll}
        0 & \mbox{if } 1 \leq i \leq F, \\
        \frac{1}{n-1} & \mbox{otherwise}.
    \end{array}
\right. 
\end{equation}
In particular, let $\sigma$ be distributed as $\tau_1 \circ \tau_2$ and fix $\beta<\gamma$. We have, for any $n$ large enough so that $\gamma n \geq \beta n +1$: 
\begin{equation}
    \E[f_1(\sigma)] = \sum_{i=1}^{n} \P\left( \tau_1 \circ \tau_2(i)=i \right) = F\cdot 0 + (n-F) \cdot \frac{1}{n-1} =1- \frac{F-1}{n-1}\leq 1-\beta.
\end{equation}
Hence, setting $h_K: x \mapsto x \mathds{1}_{x \leq K}$ for $K \geq 1$, we have for any $K\geq 1$

\begin{equation}
\label{eq:hK de f1 de sigma}
    \E[h_K(f_1(\sigma))] \leq \E[f_1(\sigma)] \leq 1-\beta.
\end{equation}

On the other hand, fix an integer $K_0 := K_0(\beta)$ large enough such that $\sum_{k=0}^{K_0} k \frac{e^{-1}}{k!} \geq 1-\beta/3$. Such a $K_0$ exists, since $\sum_{k \geq 0} k \frac{e^{-1}}{k!} = 1$.
Now, let $\Pi \sim \Unif_{\kE(\boldsymbol{\cC}^{(n)})}$ be a uniformly random permutation in $\kE(\boldsymbol{\cC}^{(n)})$. Since $f_1(\Pi)$ converges in distribution to $\Poiss(1)$, for any fixed $k \geq 0$, we have $\P(f_1(\Pi)=k) \xrightarrow[n\to \infty]{} \frac{e^{-1}}{k!}$. Therefore, for $n$ large enough:

\begin{equation}
     \E[h_{K_0}(f_1(\Pi))] \geq \sum_{k=0}^{K_0} k \frac{e^{-1}}{k!}  - \frac{\beta}{6} \geq 1-\frac{\beta}{2}.
\end{equation}
By \eqref{eq:hK de f1 de sigma}, this implies that, for $n$ large enough (recall that $\cL(X)$ denotes the law of a random variable $X$): 

\begin{equation}
     \E[h_{K_0}(f_1(\Pi))] -  \E[h_{K_0}(f_1(\sigma))] \geq \frac{\beta}{2}.
\end{equation}
Letting $\mu := \cL(f_1(\Pi))$ and $\nu := \cL(f_1(\sigma))$ be the laws of $f_1(\Pi)$ and $f_1(\sigma)$, we get
\begin{equation}
    \frac{\beta}{2} \leq \sum_{k=0}^{K_0} k(\mu(k)-\nu(k)) \leq \sum_{k=0}^{K_0} K_0 \max_{0\leq k \leq K_0}(\mu(k)-\nu(k)) = (K_0 + 1) K_0 \max_{0\leq k \leq K_0}(\mu(k)-\nu(k)).
\end{equation}
In particular, there exists $k \leq K_0$ for which
\begin{equation}
    \bbP(f_1(\Pi) = k) - \bbP(f_1(\sigma) = k) \geq \frac{\beta}{2K_0(K_0+1)},
\end{equation}
which concludes the proof.
\end{proof}

\subsection{Cutoff for conjugacy classes with large support and no short cycles}

Here we prove Theorem \ref{thm: cutoff large support no short cycles}, that is, cutoff for conjugacy classes with no short cycles.

Let us first prove another bound on $E(\sigma)$, this time depending on the length of the smallest non-trivial cycle of $\sigma$.

For $n\geq 2$ and $\sigma \in \mathfrak{S}_n \backslash \{\Id\}$, let $\ibis := \ibis(\sigma) = \min\ag i\geq 2 \du f_i \geq 1 \ad$ be the length of the smallest non-trivial cycle of $\sigma$. We also denote by $\cyc_{\geq i}(\sigma)$ the number of cycles of $\sigma$ of length at least $i$.

\begin{proposition}\label{prop: bound E sigma with number of cycles and ibis}
    Let $n\geq 2$ and $\sigma \in \mathfrak{S}_n\backslash\ag \Id\ad$ with at least one fixed point. Then

    \begin{equation}
        E(\sigma) - 1 \leq -\frac{\ln (n/f_1)}{\ln n}\pg 1-\frac{1}{\ibis}\pd.
    \end{equation}
\end{proposition}
\begin{proof}
    For convenience, in this proof, set $\ell = \ibis(\sigma)$. We define, for $i\geq \ell$,
    \begin{equation}
        G_i = f_1 + \ell \sum_{\ell \leq j \leq i} f_j.
    \end{equation} 

    We have by definition of the orbit growth sequence that $e_1 = \frac{\ln f_1}{\ln n}$, that $e_i = 0$ for $2\leq i<\ell$, 
    and that $e_\ell = \frac{\ln \pg 1 + \frac{\ell f_\ell}{f_1} \pd}{\ln n} = \frac{\ln \pg \frac{G_\ell}{f_1} \pd}{\ln n}$. Then, using that $(1+x)^a \leq 1+ax$ for [$x \geq 0$ and $0 \leq a \leq 1$], we get for $i>\ell$: 
    \begin{equation}
        \frac{\ell}{i}e_i 
        =  \frac{\ln\pg \pg 1 + \frac{i f_i}{\Sigma_{i-1}}\pd^{\ell/i}\pd}{\ln n} \leq \frac{\ln\pg 1 + \frac{\ell  f_i}{\Sigma_{i-1}}\pd}{\ln n} \leq \frac{\ln\pg 1 + \frac{\ell f_i}{ G_{i-1}}\pd}{\ln n} = \frac{\ln\pg G_{i}/ G_{i-1}\pd}{\ln n},
    \end{equation}
where we used in the last inequality that

\begin{equation}
    \Sigma_{i-1} = f_1 + \sum_{\ell\leq j \leq i-1} j f_j \geq f_1 + \ell \sum_{\ell\leq j \leq i-1} f_j = G_{i-1}.
\end{equation}
The terms then compensate as in the proof of Proposition \ref{prop: borne E sigma avec imin}: we have
\begin{equation}
\begin{split}
    E(\sigma) = e_1 + \frac{e_{\ell}}{\ell} + \sum_{i>\ell} \frac{e_i}{i}  = \frac{\ln f_1}{\ln n} + \frac{1}{\ell} \frac{\ln (G_\ell/f_1)}{\ln n} + \frac{1}{\ell} \sum_{i>\ell} \frac{\ell}{i}e_i  \leq \frac{\ln f_1}{\ln n} + \frac{1}{\ell}\frac{\ln(G_n/f_1)}{\ln n}.
\end{split}
\end{equation}
We conclude that
\begin{equation}
    E(\sigma) \leq \frac{\ln f_1}{\ln n} + \frac{1}{\ell}\frac{\ln(n/f_1)}{\ln n} = \frac{\ln f_1}{\ln n}\pg 1- \frac{1}{\ell}\pd + \frac{1}{\ell} = 1 - \frac{\ln (n/f_1)}{\ln n} \pg 1- \frac{1}{\ell} \pd. \qedhere
\end{equation}
\end{proof}

\begin{remark}
    If $\sigma$ consists only of fixed points and $\ibis$-cycles, then $E(\sigma) = e_1 + \frac{1-e_1}{\ibis} = e_1\pg 1-\frac{1}{\ibis}\pd+ \frac{1}{\ibis}$. The bound from Proposition \ref{prop: bound E sigma with number of cycles and ibis} is therefore sharp in this case.
\end{remark}

We now turn to the proof of Theorem \ref{thm: cutoff large support no short cycles}.
\begin{proof}[Proof of Theorem \ref{thm: cutoff large support no short cycles}]

 Let $\varepsilon>0$.
    First assume that $f_1=0$. Recall that $\mathrm{d}^{(n)}(0) = 1-o(1)$ since the walk starts from the identity permutation, and $\mathrm{d}^{(n)}(1) = 1-o(1)$ since the walk after 1 step is concentrated on $\cC^{(n)}$. The result to prove is therefore that $\mathrm{d}^{(n)}(2) \to 0$. 

  In this case we have $\imin =\min \ag i\geq 2 \du f_i(\cC^{(n)}) \geq 1 \ad \to \infty$ by assumption, so by Proposition \ref{prop: borne E sigma avec imin} we have $E(\cC^{(n)}) \leq \frac{1}{\imin} = o(1)$.
    It follows directly from Proposition \ref{prop: sufficient condition for mixing with sum of two E(sigma)} that $\mathrm{d}^{(n)}(2) \to 0$.

 From now on we assume that $f_1(\cC^{(n)}) \geq 1$.
    The lower bound is a classical counting argument that can be done with coupon collector methods. We briefly recall it. A second moment argument shows that, if one multiplies $\lf \frac{\ln n}{\ln(n/f)}(1-\varepsilon)\rf$ permutations (i.i.d., uniform in a conjugacy class) that have support size $n-f$, then, with probability $1-o(1)$, at least $n^{\varepsilon}/2$ are fixed points of all these permutations and therefore are fixed points of their product. Since uniform permutations (also on $\kA_n$ and $\kS_n \backslash \kA_n$) have less than $\ln n$ fixed points with probability $1-o(1)$, we deduce that $\mathrm{d}^{(n)}\pg \lf \frac{\ln n}{\ln(n/f)}(1-\varepsilon)\rf\pd \xrightarrow[n\to \infty]{} 1$.

Let us now show the cutoff upper bound. The proof is similar to that of Theorem \ref{thm: application Lulov Pak extended}.
Set for the rest of the proof $t := \lc\frac{\ln n}{\ln (n/f)}(1+\varepsilon)\rc$, and $\ell = \ibis = \min\ag i\geq 2 \du f_i(\cC^{(n)}) \geq 1\ad$.

Let $C$ be the constant appearing in Theorems \ref{thm:ALS virtual degree asymptotic bound} and \ref{thm:ALS improved character bound}. First, by  Proposition \ref{prop: bound E sigma with number of cycles and ibis}, and the facts that $\ell \to \infty$, and $f = o(n)$, we have
\begin{equation}
\begin{split}
     \pg 1+ \frac{C}{\ln n} \pd E(\cC^{(n)}) -1 & = \pg 1+ \frac{C}{\ln n} \pd \pg E(\cC^{(n)}) -1\pd + \frac{C}{\ln n} \\
     & \leq -\pg 1+ \frac{C}{\ln n} \pd\frac{\ln (n/f)}{\ln n} \left( 1-\frac{1}{\ell} \right)  + \frac{C}{\ln n} \\
     & = - \frac{\ln (n/f)}{\ln n}(1+o(1)).
\end{split}
\end{equation}
In particular, it is negative for $n$ large enough. We therefore have as $n \rightarrow \infty$:

\begin{equation}
\begin{split}
    1 + t\pg \pg 1+ \frac{C}{\ln n} \pd E(\cC^{(n)}) -1 \pd & \leq 1 + \pg (1+\varepsilon)\frac{\ln n}{\ln (n/f)}\pd\pg-\frac{\ln (n/f)}{\ln n}(1+o(1))\pd \\
    & = -\varepsilon+o(1).
\end{split}
\end{equation}
We deduce from Lemma \ref{lem: lemme DS étendu à des classes potentiellement différentes}, Theorem \ref{thm:ALS improved character bound}, and bounds on the Witten zeta function (Proposition \ref{prop: witten zeta iff condition S n ** new}), that
\begin{equation}
    4\mathrm{d}^{(n)}(t)^2 \leq \sum_{\lambda\in \widehat{\mathfrak{S}_n}^{**}} \pg d_\lambda |\chi^\lambda(\cC^{(n)})|^{t}\pd^2 \leq \sum_{\lambda\in \widehat{\mathfrak{S}_n}^{**}} d_\lambda^{-2\varepsilon + o(1)} \xrightarrow[n\to \infty]{} 0,
\end{equation}
which concludes the proof.
\end{proof}

\section*{Acknowledgements}
          We are grateful to Nathanaël Berestycki, Persi Diaconis, Valentin Féray and Sam Olesker-Taylor for interesting conversations and helpful comments on early versions of the paper. We also thank the anonymous referee for their careful reading and useful comments.
          
          L.T. was supported by the Pacific Institute for the Mathematical Sciences, the Centre national de la recherche scientifique, and the Simons Foundation, via a PIMS-CNRS-Simons postdoctoral fellowship.
          P.T. was supported by the project ANR-24-CPJ1-0032-01, and the
          project "Rawabranch" number ANR-23-CE40-0008.
          L.T. and P.T. were supported by the Austrian Science Fund (FWF) under grants 10.55776/P33083 and SFB F 1002.

\footnotesize{
\bibliographystyle{alpha}
\bibliography{bibliographieLucas}
}
\end{document}